\documentclass[12pt]{amsart}
\usepackage{amssymb,amsmath,amsthm,mathrsfs,MnSymbol}
\usepackage{graphicx}
\usepackage[all,cmtip]{xy}
\usepackage{tikz-cd}
\usepackage{multirow,scalerel,arydshln,caption,subcaption}   
\usepackage{hyperref}
\usepackage[msc-links]{amsrefs}
\hypersetup{
  colorlinks   = true,	
  urlcolor     = blue,	
  linkcolor    = blue,	
  citecolor   = red	
}
\calclayout
\allowdisplaybreaks
\newtheorem{theorem}{Theorem}
\newtheorem{lemma}[theorem]{Lemma}
\newtheorem{remark}[theorem]{Remark}
\newtheorem{corollary}[theorem]{Corollary}

\newtheorem{proposition}[theorem]{Proposition}
\numberwithin{theorem}{section} \numberwithin{equation}{section}
\newcommand{\beq}{\begin{small} \begin{equation}}
\newcommand{\eeq}{\end{equation} \end{small}}
\newcommand{\beqn}{\begin{small} \begin{equation*}}
\newcommand{\eeqn}{\end{equation*} \end{small}}
\DeclareMathAlphabet{\mathpzc}{OT1}{pzc}{m}{it}
\newcommand\scalemath[2]{\scalebox{#1}{\mbox{\ensuremath{\displaystyle #2}}}}
\newcommand*{\MyScalePicHuge}{0.8}
\newcommand*{\MyScalePicLarge}{0.45}
\newcommand*{\MyScalePicBig}{0.4}
\newcommand*{\MyScalePicMed}{0.35}
\newcommand*{\MyScalePicSmall}{0.3}
\begin{document}
\title[On K3 surfaces of Picard rank 14]{On K3 surfaces of Picard rank 14}
\author{Adrian Clingher}
\address{Department of Mathematics and Statistics, University of Missouri - St. Louis, St. Louis, MO 63121}
\email{clinghera@umsl.edu}
\author{Andreas Malmendier}
\address{Department of Mathematics \& Statistics, Utah State University, Logan, UT 84322}
\email{andreas.malmendier@usu.edu}
\begin{abstract}
We study complex algebraic K3 surfaces with finite automorphism groups and polarized by rank-fourteen, 2-elementary lattices. Three such lattices exist, namely $H \oplus E_8(-1) \oplus A_1(-1)^{\oplus 4}$, $H \oplus E_8(-1) \oplus D_4(-1)$, and $H \oplus D_8(-1) \oplus D_4(-1)$. As part of our study, we provide birational models for these surfaces as quartic projective hypersurfaces and describe the associated coarse moduli spaces in terms of suitable modular invariants. Additionally, we explore the connection between these families and dual K3 families related via the Nikulin construction.
 \end{abstract}
\maketitle
\section{Introduction and summary of results}
Let $\mathcal{X}$ be a smooth complex projective K3 surface. Denote by $\operatorname{NS}(\mathcal{X})$ the N\'eron-Severi lattice of $\mathcal{X}$. This is known to be an even lattice of signature $(1,p_\mathcal{X}-1)$, where $p_\mathcal{X}$ being the Picard rank of $\mathcal{X}$, with $1 \leq p_\mathcal{X} \leq 19$.  A {\it lattice polarization} \cites{MR0357410,MR0429917,MR544937,MR525944,MR728992} on $\mathcal{X}$ is, by definition, a primitive lattice embedding $i \colon L \hookrightarrow \operatorname{NS}(\mathcal{X})$, with $i(L)$ containing a pseudo-ample class. Here, $L$ is a choice of even indefinite lattice of signature $(1,\rho_L-1)$, with $ 1 \leq \rho_L \leq 20$. Two $L$-polarized K3 surfaces $(\mathcal{X},i)$ and $(\mathcal{X}',i')$ are said to be isomorphic\footnote{Our definition of isomorphic lattice polarizations coincides with the one used by Vinberg \cites{MR2682724, MR2718942, MR3235787}. It is slightly more general than the one used in \cite{MR1420220}*{Sec.~1}.},  if there exists an analytic isomorphism $\alpha \colon \mathcal{X} \rightarrow \mathcal{X}'$ and a lattice isometry  $ \beta \in O(L)$, such that $ \alpha^* \circ i' = i \circ \beta $, where $\alpha^*$ is the induced morphism at cohomology level. In general, $L$-polarized K3 surfaces are classified, up to isomorphism, by a coarse moduli space $\mathscr{M}_{L}$, which is known  \cite{MR1420220} to be a quasi-projective variety of dimension $20-\rho_L$.   A \emph{general} $L$-polarized K3 surface $(\mathcal{X},i)$ satisfies $i(L)=\operatorname{NS}(\mathcal{X})$.
\par A special case for the above discussion is given by the polarizations by the lattice $ H \oplus N$. Here $H$ represents the standard hyperbolic lattice of rank two and $N$ is the rank-eight Nikulin lattice; see \cite{MR728142}*{Def.~5.3}.  The moduli space $\mathscr{M}_{H \oplus N}$ is ten-dimensional.  A polarization by the lattice $H \oplus N$ is known \cite{MR2274533} to be equivalent to the existence of a canonical \emph{van Geemen-Sarti involution} $ \jmath_{\mathcal{X}} \colon \mathcal{X} \rightarrow \mathcal{X} $ on the K3 surface $\mathcal{X}$, i.e., a symplectic involution that is given by fiber-wise translations, by a section of order-two, in a Jacobian elliptic fibration on $\mathcal{X}$; the fibration is usually referred to as \emph{alternate fibration}. If one factors $\mathcal{X}$ by the involution $\jmath_{\mathcal{X}}$ and then resolves the eight occurring singularities, a new K3 surface $\mathcal{Y}$ is obtained, related to $\mathcal{X}$ via a rational double-cover map $ \mathcal{X} \dashrightarrow \mathcal{Y}$. This construction is referred to in the literature as the \emph{Nikulin construction}. The surface $\mathcal{Y}$ also has a canonical van Geemen-Sarti involution $ \jmath_{\mathcal{Y}}$ and in turn carries a $H \oplus N$-lattice polarization. Moreover, if one repeats the Nikulin construction on $\mathcal{Y}$, the original K3 surface $\mathcal{X}$ is recovered. The two surface $\mathcal{X}$ and $\mathcal{Y}$ are related via a pair of dual rational double-cover maps:      
\beq
 \label{diag:isogeny1}
\xymatrix{
\mathcal{X} \ar @(dl,ul) _{\jmath_\mathcal{X} } \ar @/_0.5pc/ @{-->}  [rr]
&
& \mathcal{Y} \ar @(dr,ur) ^{\jmath_\mathcal{Y} }  \ar @/_0.5pc/ @{-->}  [ll] \\
}
\eeq
We shall refer to this correspondence as the {\it van Geemen-Sarti-Nikulin duality}. It determines an interesting involution, at the level of moduli spaces:
\beq
 \imath_{\mathrm{vgsn}}  \colon \quad \mathscr{M}_{H \oplus N} \rightarrow \mathscr{M}_{H\oplus N} \,, 
 \qquad \text{with} \quad  \imath_{\mathrm{vgsn}}  \circ \imath_{\mathrm{vgsn}}  = {\rm id}  \,.
\eeq
\par Let us turn to the main content of the present article: the focus of the paper is the study of K3 surfaces with finite automorphism groups and polarized by rank-fourteen lattices of \emph{2-elementary} type. As a reminder, a lattice is called \emph{2-elementary} if its discriminant group is a self-product of $\mathbb{Z}_2$.  As proved by Nikulin \cites{MR544937,MR633160,MR633160b,MR752938,MR3165023}, Picard rank fourteen is the highest rank where there exist more than one 2-elementary, primitive sub-lattice of the K3 lattice for K3 surfaces with finite automorphism groups. The three possibilities, in this rank, are
\beq
\label{eqn:lattices_intro}
 P_{14}=H \oplus E_8(-1) \oplus A_1(-1)^{\oplus 4} , \  P'_{14}=H \oplus D_8(-1) \oplus D_4(-1), \ P''_{14}=H \oplus E_8(-1) \oplus D_4(-1).
\eeq 
Here, $E_n(-1)$, $D_n(-1)$, $A_n(-1)$ are the negative definite even lattices associated with  their corresponding namesake root systems.  
\par This article provides a detailed study for  K3 surfaces polarized by the three lattices above. For each of the above cases, we give explicit 
birational models, given as projective quartic surfaces.  We also describe the associated coarse moduli spaces and studey  By work of Kondo~\cite{MR1029967}, 
Nikulin~\cite{MR556762} and Sterk~\cite{MR786280}, one knows that general members of the lattice polarized K3 families associated with~(\ref{eqn:lattices_intro}) carry finitely many smooth rational curves. We construct the dual graphs of smooth rational curve configurations.  Finally, we study the dual van Geemen-Sarti-Nikulin K3 families (if they exist).

\subsection{The \texorpdfstring{$P_{14}$}{P}-polarized K3 surfaces}
The most intricate case, among the three lattice polarizations listed in~(\ref{eqn:lattices_intro}), is the case of lattice $P_{14}$. The K3 surfaces polarized by the lattice $P_{14}$ fit into a family of projective quartic surfaces as follows: 
\begin{theorem}
\label{thm:intro}
Let $(\alpha, \beta, \gamma, \delta , \varepsilon, \zeta, \eta, \iota, \kappa, \lambda) \in \mathbb{C}^{10}$. Consider the projective surface in $\mathbb{P}^3=\mathbb{P}(\mathbf{X}, \mathbf{Y}, \mathbf{Z}, \mathbf{W})$ defined by the homogeneous quartic equation
\beq
\label{quartic1}
\begin{split}
0=& \ \mathbf{Y}^2 \mathbf{Z} \mathbf{W}-4 \mathbf{X}^3 \mathbf{Z}+3 \alpha \mathbf{X} \mathbf{Z} \mathbf{W}^2+\beta \mathbf{Z} \mathbf{W}^3  \\
- & \frac{1}{2} \big(  2 \gamma \mathbf{X} - \delta \mathbf{W}  \big)   \big(  2 \eta \mathbf{X} - \iota \mathbf{W}  \big)  \mathbf{Z}^2
 - \frac{1}{2} \big(  2 \varepsilon \mathbf{X} - \zeta \mathbf{W}  \big)   \big(  2 \kappa \mathbf{X} - \lambda \mathbf{W}  \big)  \mathbf{W}^2.
\end{split}
\eeq
Assuming general parameters, the surface  $\mathcal{X}$ obtained as the minimal resolution of (\ref{quartic1}) is a K3 surface endowed with a canonical  $P_{14}$-polarization. Conversely, every $P_{14}$-polarized K3 surface has a birational projective model given by Equation~(\ref{quartic1}). 
\end{theorem}
\par The result will be obtained as Theorem~\ref{thm1}, and the dual graph of all smooth rational curves will be determined in Theorem~\ref{thm:polarization14}. One can also tell when two members of the above family are isomorphic.  Let $\mathpzc{G}$ be the subgroup of $\operatorname{Aut}(\mathbb{C}^{10})$ generated by the following set of transformations:
\beq
\begin{array}{rcl}
 (\alpha,\beta, \gamma, \delta, \varepsilon, \zeta, \eta,  \iota, \kappa, \lambda) & \ \to \ &
(\alpha,  \beta,  \varepsilon,  \zeta, \gamma,  \delta , \eta,  \iota, \kappa, \lambda ) \,, \\[0.2em]
(\alpha,\beta, \gamma, \delta, \varepsilon, \zeta, \eta,  \iota, \kappa, \lambda) & \ \to \ &
(\alpha,\beta,  \eta,  \iota,  \varepsilon, \zeta, \gamma, \delta, \kappa, \lambda) \,,\\[0.2em]
(\alpha,\beta, \gamma, \delta, \varepsilon, \zeta, \eta,  \iota, \kappa, \lambda)& \ \to \ &
(\alpha,\beta, \gamma, \delta, \kappa, \lambda, \eta,  \iota, \varepsilon, \zeta) \,, \\ [0.2em]
(\alpha,\beta, \gamma, \delta, \varepsilon, \zeta, \eta,  \iota, \kappa, \lambda)& \ \to \ &
(\Lambda^4 \alpha,  \Lambda^6 \beta,  \Lambda^{10} \gamma,  \Lambda^{12} \delta,  \Lambda^{-2} \varepsilon,  \zeta , \Lambda^{-2} \eta, \iota, \Lambda^{-2} \kappa, \Lambda )  \,,
\end{array}
\eeq
with $\Lambda \in \mathbb{C}^\times$. Then, two K3 surfaces in the above family are isomorphic, as $P_{14}$-polarized K3 surfaces, if and only if their coefficient 10-tuples belong to the same orbit, under the action of $\mathpzc{G}$. This fact leads to the definition of the following invariants:
\begin{small}
\begin{gather}
\label{eqn:invariants_intro_1}
J_4  = \alpha \,, \qquad J'_4 = \gamma \varepsilon  \eta \kappa \,, \qquad J_6 = \beta \,, \\
J'_6  =  \gamma \varepsilon  (\iota \kappa + \eta \lambda)   + \eta \kappa  (\gamma \zeta + \delta \varepsilon) \,,\\
J_8  =   (\gamma \zeta + \delta \varepsilon) (\iota \kappa + \eta \lambda) + \delta \zeta \eta \kappa + \gamma \varepsilon \iota \lambda \,,\\
\label{eqn:invariants_intro_4}
J_{10}  =  \delta \zeta (\iota \kappa + \eta \lambda) + \iota \lambda (\gamma \zeta + \delta \varepsilon)  \,, \qquad
J_{12}  =   \delta \zeta  \iota \lambda \,.
\end{gather}
\end{small}%
These seven invariants may be interpreted as a weighted-projective point, i.e., 
\beqn
\Big[  J_4 :  J'_4 : J_6:  J'_6 :  J_8:  J_{10}:  J_{12}   \Big]   \in  \mathbb{WP}_{(4,4,6,6,8,10,12)}  \,,
\eeqn
associated to a $P_{14}$-polarized K3 surface.  The result is based on the existence of a unique Jacobian elliptic fibration on a general $P_{14}$-polarized K3 surface, given by
\beq
\label{eqn:alt00}
  \mathcal{X}\colon \quad y^2 z  = x^3 + v  A(u, v) \, x^2 z + v^4 B(u, v) \, x z^2  \,,
\eeq
with the defining polynomials
\beq
 A(t) =  t^3 - 3 \alpha t - 2 \beta \,, \qquad
 B(t)  =  \big(\gamma t - \delta\big) \big(\varepsilon t - \zeta\big)  \big(\eta t  - \iota\big)  \big(\kappa t -\lambda\big) \,.
\eeq 
In this context, the following will be proved as Theorem~\ref{thm:moduli_space_alt1}:
\begin{theorem}
\label{thm:intro2}
 The six-dimensional open analytic space
 \beq
 \mathscr{M}_P = \label{eqn:moduli_space_P}
 \left\{ 
\begin{array}{c} 
 \Big[  J_4 :  J'_4 : J_6:  J'_6 :  J_8:  J_{10}:  J_{12}   \Big] \\
 \in \mathbb{WP}_{(4,4,6,6, 8,10,12)}  
 \end{array}
 \left\vert 
\begin{array}{l}
{ \scriptstyle  (  J'_4,  J'_6 ,  J_8,  J_{10},  J_{12} )  \neq 0 , }\\
{ \scriptstyle \not \exists \, r, s \, \in \, \mathbb{C}: \ (J_4, J_6) = ( r^2, r^3) \; \text{and}} \\
{ \scriptstyle (J'_4, J'_6, J_8, J_{10}, J_{12}) = (s, -4 r s, 6r^2 s, -4r^3 s, r^4 s)}
\end{array}  \right. \right\} ,
\eeq
is a coarse moduli space of $P_{14}$-polarized K3 surfaces.\footnote{The weighted projective space, considered as a stack, has a $\mathbb{Z}/2\mathbb{Z}$ stabilizer at a general point. However, we want to keep these even weights as they can be interpreted as the weights of the generators that freely generate the associated algebra of automorphic forms.}
\end{theorem}
\noindent Should one set $J'_4=0$ in the above context, one obtains an enhancement of the polarization to the following rank-fifteen lattice:
\beqn 
P_{15} \ = \ H \oplus E_8(-1) \oplus D_4(-1)  \oplus A_1(-1) 
 \cong   H \oplus E_7(-1) \oplus D_6(-1) 
 \cong   H   \oplus D_{12}(-1) \oplus A_1(-1) .
\eeqn 
And given $J'_4=J'_6=0$, the lattice polarization becomes:
\beq
\label{rank16}
P_{16} \  = \ H \oplus E_8(-1) \oplus D_6(-1)  
\cong  H \oplus E_7(-1) \oplus E_{7}(-1) 
 \cong   H   \oplus D_{14}(-1)    \ .
\eeq
The case~(\ref{rank16}) was studied at length in earlier work \cites{CHM19,MR4160930} by the authors.
\par K3 surfaces with $P_{14}$-polarization provide an interesting case to study from the point of view of the van Geemen-Sarti-Nikulin duality. As we will show, one has a canonical lattice embedding  $H \oplus N \hookrightarrow P_{14}$, which is unique up to an isometry. Therefore, any $P_{14}$-polarized K3 surface also carries an underlying $H \oplus N$-polarization; we explain this point in Remark~\ref{rem:P}. This leads to a canonical embedding 
\beqn
	\mathscr{M}_P \  \hookrightarrow \ \mathscr{M}_{H\oplus N}    \,,
\eeqn
which realizes $\mathscr{M}_P$ as a six-dimensional sub-variety inside the ten-dimensional quasi-projective moduli space $\mathscr{M}_{H \oplus N}$. It is then natural to ask: what are the van Geemen-Sarti-Nikulin duals to $P_{14}$-polarized K3 surfaces? As it turns out, the answer is quite interesting and will be given in Theorem~\ref{thm:end1c}:
\begin{theorem}
\label{thm:VGSNdual}
Let $(\mathcal{X},i)$ be a $P_{14}$-polarized K3 surface. The surface $\mathcal{X}$ carries a canonical van Geemen-Sarti involution $j_{\mathcal{X}} \in \operatorname{Aut}(\mathcal{X}) $. Denote by $\mathcal{Y}$ the new K3 surface obtained after applying the Nikulin construction in the context of $j_{\mathcal{X}}$.  Then, $\mathcal{Y}$ is the minimal resolution of a double cover of $\mathbb{P}^2$ branched over three distinct concurrent lines and a cubic curve.  
\end{theorem}
Surfaces $\mathcal{Y}$ form a special class of \emph{double sextic} K3 surfaces and constitute the family polarized by the lattice $R_{14} =  H  \oplus D_4(-1)^{\oplus 3}$. The converse of Theorem~\ref{thm:VGSNdual} also holds: given a cubic and three concurrent lines in $\mathbb{P}^2$, the K3 surface obtained as minimal resolution of the projective double cover with branch locus given by this curve configuration is the  van Geemen-Sarti-Nikulin dual of a K3 surface with a $P_{14}$-polarization. Moreover, the duality correspondence can be made completely explicit, as one can read the invariants $[  J_4 :  J'_4 : J_6:  J'_6 :  J_8:  J_{10}:  J_{12} ]$ in terms of the coefficients of the three lines and the cubic curve.  Should we restrict to the case $J'_4=0$ or $(J'_4, J'_6)=(0,0)$, the sextic curve configuration on the dual side gets enhanced slightly - the cubic curve acquires a point of tangency or a singularity, respectively, at one of the points of intersection with the three lines.
\subsection{The \texorpdfstring{$P'_{14}$}{P'}-polarized K3 surfaces}
The K3 surfaces polarized by the lattice $P'_{14}$ also fit into a family of projective quartic surfaces as follows: 
\begin{theorem}
\label{thm:pprime}
Let $(f_0,f_1,f_2, g_0, h_0,h_1,h_2) \in \mathbb{C}^{7}$. Consider the projective surface in $\mathbb{P}^3=\mathbb{P}(\mathbf{X}, \mathbf{Y}, \mathbf{Z}, \mathbf{W})$ defined by the homogeneous quartic equation
\beq
\label{quartic2}
\begin{split}
0= \mathbf{Y}^2 \mathbf{Z} \mathbf{W}-4 \mathbf{X}^3 \mathbf{Z}- 2 \Big(  \mathbf{W}^2 + f_2  \mathbf{W} \mathbf{Z} + h_2 \mathbf{Z}^2 \Big) \, \mathbf{X}^2 \phantom{\,.} \\
-  \Big(   f_1 \mathbf{W} \mathbf{Z} -h_2 \mathbf{W}^2 + h_1 \mathbf{Z}^2 \Big)  \,  \mathbf{X} \mathbf{Z}
- \frac{1}{2}  \Big(   f_0 \mathbf{W} \mathbf{Z} +   g_0 \mathbf{W}^2 + h_0 \mathbf{Z}^2 \Big)  \, \mathbf{Z}^2\,.
\end{split}
\eeq
Assuming general parameters,  the surface  $\mathcal{X}'$ obtained as the minimal resolution of (\ref{quartic2}) is a K3 surface endowed with a canonical  $P'_{14}$-polarization. Conversely, every $P'_{14}$-polarized K3 surface has a birational projective model given by Equation~(\ref{quartic2}). 
\end{theorem}
\noindent The result will be obtained as Theorem~\ref{thm_parity}, and the dual graph of all smooth rational curves will be determined in Theorem~\ref{thm:polarization14sd}.  In a manner similar to the case of a $P_{14}$-polarization, one may control when two members of the family are isomorphic, as lattice polarized surfaces.  In order to see this, we define the following invariants:
\beq
\label{eqn:CurlyJinvariants}
\begin{array}{c}
\mathcal{J}_2=f_2, \qquad \mathcal{J}_6=f_1,  \qquad  \mathcal{J}_8 = g_0 + h_1 - h_2^2, \qquad \mathcal{J}_{10} = f_0,\\[4pt]
\mathcal{J}_{12} = g_0 h_2 - h_1 h_2 + h_0, \qquad   \mathcal{J}_{16}= g_0 h_1 -  h_0 h_2, \qquad \mathcal{J}_{20} = g_0 h_0.
\end{array}
\eeq
Let then $\mathpzc{G}' \simeq \mathbb{C}^\times$ be the subgroup of $\operatorname{Aut}(\mathbb{C}^{7})$ given by the transformation
\beqn
 ( \mathcal{J}_2, \mathcal{J}_6, \mathcal{J}_8, \mathcal{J}_{10}, \mathcal{J}_{12}, \mathcal{J}_{16}, \mathcal{J}_{20} )  \ \to \ 
(\Lambda^{2} \mathcal{J}_2, \, \Lambda^6 \mathcal{J}_6, \, \Lambda^8 \mathcal{J}_8, \, \Lambda^{10} \mathcal{J}_{10}, \, \Lambda^{12} \mathcal{J}_{12}, \, \Lambda^{16} \mathcal{J}_{16}, \, \Lambda^{20} \mathcal{J}_{20}) \,,
\eeqn
with $\Lambda \in \mathbb{C}^\times$. Then, two K3 surfaces from the quartic family in Equation~(\ref{quartic2}) are isomorphic as $P'_{14}$-polarized K3 surfaces, if and only if their coefficients belong to the same orbit under the action of $\mathpzc{G}'$. The following will be proved as Theorem~\ref{thm:moduli_space_seldual}:
\begin{theorem}
\label{thm:intro22}
 The six-dimensional open analytic space
 \beq
 \label{eqn:moduli_space_PP}
 \mathscr{M}_{P'} = \left\{ 
\begin{array}{c} 
\Big[  \mathcal{J}_2 :  \mathcal{J}_6 : \mathcal{J}_8:  \mathcal{J}_{10} :  \mathcal{J}_{12}:  \mathcal{J}_{16}:  \mathcal{J}_{20}   \Big] \\
 \in \mathbb{WP}_{(2,6,8,10, 12,16,20)}   
 \end{array}
 \left\vert 
\begin{array}{l}
{ \scriptstyle \not \exists \, r, s \, \in \, \mathbb{C}: \  ( \mathcal{J}_2 ,  \mathcal{J}_6, \mathcal{J}_8,  \mathcal{J}_{10},  \mathcal{J}_{12},  \mathcal{J}_{16},  \mathcal{J}_{20} ) }\\
{ \scriptstyle \quad = \  (s^2, 2 r s^2, 10r^2, s^2r^2, -20r^3, -15r^4, - 4 r^5)}
\end{array}  \right. \right\} ,
\eeq
is a coarse moduli space of $P'_{14}$-polarized K3 surfaces. 
\end{theorem} 
$P'_{14}$-polarized K3 surfaces also form an interesting study case for the van Geemen-Sarti-Nikulin duality. A unique canonical primitive lattice embedding $H \oplus N \hookrightarrow P' _{14}$ exists, and hence, any $P'_{14}$-polarized K3 surface carries an underlying $H \oplus N$-polarization; we explain this point in Remark~\ref{rem:PP}. One has therefore an embedding 
\beqn
\mathscr{M}_{P'} \  \hookrightarrow  \ \mathscr{M}_{H\oplus N}   \,.
\eeqn
However, in contrast to the $P_{14}$-polarized case, $\mathscr{M}_{P'}$ is left invariant by the van Geemen-Sarti-Nikulin duality, and the dual of a $P'_{14}$-polarized K3 surface is again a $P'_{14}$-polarized surface. In Proposition~\ref{prop:involutionCD} we will show that this involution, denoted by
\beqn
 \imath_{\mathrm{vgsn}}'\colon \quad  \mathscr{M}_{P'}  \ \rightarrow \  \mathscr{M}_{P'} \,,
 \qquad \text{with} \quad \imath_{\mathrm{vgsn}} ' \circ \imath_{\mathrm{vgsn}} ' = {\rm id} \,,
\eeqn 
is given by:
\beq
\label{eqn:involutionCD22}
\imath_{\mathrm{vgsn}} ' : \; 
\left(
\begin{array}{lcl}
	\mathcal{J}_2 & \mapsto &
		 -\mathcal{J}_2 \,, \\[3pt]
	\mathcal{J}_6 & \mapsto &
		\phantom{-} \mathcal{J}_6 + \frac{1}{10} \mathcal{J}_2^3 \,, \\[3pt]
	\mathcal{J}_8  & \mapsto &
		\phantom{-} \mathcal{J}_8 - \frac{1}{2} \mathcal{J}_6 \mathcal{J}_2 - \frac{1}{40} \mathcal{J}_2^4 \,,\\[3pt]
	\mathcal{J}_{10} & \mapsto &
		-\mathcal{J}_{10} - \frac{1}{20} \mathcal{J}_6 \mathcal{J}_2^2 - \frac{1}{400} \mathcal{J}_2^5 \,, \\[3pt]
	\mathcal{J}_{12} & \mapsto &
		 - \mathcal{J}_{12} + \frac{1}{2} \mathcal{J}_{10} \mathcal{J}_2  - \frac{3}{20}  \mathcal{J}_8 \mathcal{J}_2^2 + \frac{1}{4} \mathcal{J}_6^2 
 		+ \frac{3}{40} \mathcal{J}_6 \mathcal{J}_2^3 + \frac{1}{400} \mathcal{J}_2^6 \,,\\[4pt]
	 \mathcal{J}_{16}  & \mapsto &
	 	\phantom{-}\mathcal{J}_{16} + \frac{1}{10} \mathcal{J}_{12} \mathcal{J}_2^2 - \frac{1}{2} \mathcal{J}_{10} \mathcal{J}_6 
		- \frac{1}{20} \mathcal{J}_{10} \mathcal{J}_2^3 +\frac{3}{400} \mathcal{J}_8 \mathcal{J}_2^4 \\[4pt]
		&& - \frac{1}{40} \mathcal{J}_6^2 \mathcal{J}_2^2  - \frac{3}{800} \mathcal{J}_6\mathcal{J}_2^5  - \frac{3}{3200} \mathcal{J}_2^8 \,, \\[4pt]
	 \mathcal{J}_{20}  & \mapsto &
		- \mathcal{J}_{20}  -\frac{1}{20}\mathcal{J}_{16} \mathcal{J}_2^2  - \frac{1}{400}\mathcal{J}_{12}  \mathcal{J}_2^4 + \frac{1}{4} \mathcal{J}_{10}^2
		 + \frac{1}{40} \mathcal{J}_{10} \mathcal{J}_6 \mathcal{J}_2^2 + \frac{1}{800} \mathcal{J}_{10} \mathcal{J}_2^5\\[4pt]
  && - \frac{1}{8000}\mathcal{J}_8 \mathcal{J}_2^6 + \frac{1}{1600} \mathcal{J}_6^2 \mathcal{J}_2^4  + \frac{1}{16000} \mathcal{J}_6 \mathcal{J}_2^7 + \frac{1}{800000} \mathcal{J}_2^{10} \,.
\end{array} \right)
\eeq
In Corollary~\ref{cor:selfdual_locus} it will be shown that the self-dual locus is given by
\beq
 \mathcal{J}_2 = 0 \,, \qquad \mathcal{J}_{10} = 0\,, \qquad \mathcal{J}_{20} = 0 \,, \qquad \mathcal{J}_6^2 - 8 \mathcal{J}_{12} = 0 \,.
\eeq
\subsection{The \texorpdfstring{$P''_{14}$}{P''}-polarized K3 surfaces}
The case of a $P''_{14}$-polarization was previously studied by Vinberg  \cite{MR2718942}.  Following Vinberg's notation, we start with a 7-tuple    $( f_{1,2}, f_{2,2}, f_{1,3}, f_{2,3}, f_{3,3}, g_1, g_3) \in \mathbb{C}^{7}$.  We consider the projective surface in $\mathbb{P}^3 = \mathbb{P}(\mathbf{x}_0, \mathbf{x}_1, \mathbf{x}_2, \mathbf{x}_3)$ defined by the homogeneous quartic equation
\beq
\label{eqn:Vinberg_body1}
  \mathbf{x}_0^2 \mathbf{x}_2 \mathbf{x}_3 - 4 \mathbf{x}_1^3 \mathbf{x}_3 - \mathbf{x}_2^4 - \mathbf{x}_1 \mathbf{x}_3^2 \, g(\mathbf{x}_0, \mathbf{x}_1, \mathbf{x}_3) - \mathbf{x}_2 \mathbf{x}_3 \, f(\mathbf{x}_1, \mathbf{x}_2, \mathbf{x}_3) = 0 \,,
\eeq
with
\beq
  g =  g_1 \mathbf{x}_1 + g_3 \mathbf{x}_3 \,, \qquad 
  f =  f_{12} \mathbf{x}_1 \mathbf{x}_2 + f_{22} \mathbf{x}_2^2  + f_{13} \mathbf{x}_1 \mathbf{x}_3  + f_{23} \mathbf{x}_2 \mathbf{x}_3 + f_{33} \mathbf{x}_3^2  \,.
\eeq
One then has:
\begin{theorem}
\label{thm_Vinberg_body}
Assuming general parameters,  the surface  $\mathcal{X}''$ obtained as the minimal resolution of (\ref{eqn:Vinberg_body1}) is a K3 surface endowed with a canonical $P''_{14}$-polarization. Conversely, every $P''_{14}$-polarized K3 surface has a birational projective model given by Equation~(\ref{eqn:Vinberg_body1}). 
 \end{theorem}
The result will be obtained as Theorem~\ref{thm_Vinberg}, and the dual graph of all smooth rational curves will be determined in Theorem~\ref{thm:polarization14vin}.  Two members of the above family are isomorphic if and only if their coefficient sets are related by a transformation in $\mathpzc{G}'' \simeq \mathbb{C}^\times$, given by
\beq
\begin{array}{l}
 \Big( f_{1,2}, f_{2,2}, g_1, f_{1,3}, f_{2,3}, g_3, f_{3,3}\Big) \  \mapsto \ \\[4pt]
   \qquad \qquad \Big( \Lambda^4 f_{1,2},  \Lambda^6 f_{2,2},  \Lambda^8 g_1,  \Lambda^{10} f_{1,3},  \Lambda^{12} f_{2,3},  \Lambda^{16} g_3,  \Lambda^{18} f_{3,3}\Big) \,,
\end{array}
\eeq
for $\Lambda \in \mathbb{C}^\times$. This fact leads one to define invariants associated to the K3 surfaces in the family, namely 
\beq
\label{eqn:inv_PPP}
 \mathscr{J}_4 = f_{1,2},  \  \mathscr{J}_{6}=f_{2,2}, \  \mathscr{J}_8=g_1, \
 \mathscr{J}_{10} = f_{1,3},  \  \mathscr{J}_{12} = f_{2,3} , \  \mathscr{J}_{16}=g_3 , \  \mathscr{J}_{18} = f_{3,3} . 
 \eeq
 In this context, the following will be proved as part of Theorem~\ref{thm:Vinberg_moduli}:
 \begin{theorem}
\label{thm:Vinberg_moduli2}
The six-dimensional open analytic space
\beqn
\mathscr{M}_{P''} = 
 \left\{ 
 \left.
\begin{array}{c} 
\Big[  \mathscr{J}_4 :  \mathscr{J}_6 : \mathscr{J}_8 : \mathscr{J}_{10} : \mathscr{J}_{12} : \mathscr{J}_{16} : \mathscr{J}_{18}  \Big]  \\
 \in  \mathbb{WP}_{(4,6,8,10,12, 16,18)}  
 \end{array}
 \right\vert \
( \mathscr{J}_8 ,  \mathscr{J}_{10},  \mathscr{J}_{12},  \mathscr{J}_{16}, \mathscr{J}_{18} ) \neq 0 
 \right\} ,
\eeqn
is a coarse moduli space of $P''_{14}$-polarized K3 surfaces.
\end{theorem}
Should one set $\mathscr{J}_{16}=0$ in the above context, the $P''_{14}$-polarization is enhanced to $H \oplus E_8(-1) \oplus D_5(-1)$. Furthermore, the locus given by $\mathscr{J}_{16}=\mathscr{J}_{18}=0$ corresponds to $H \oplus E_8(-1) \oplus D_6(-1)$-polarized K3 surfaces. The latter case was previously studied by the authors in \cite{MR4015343}.  Finally, we note that $P''_{14}$-polarized K3 surfaces have no significance from the point of view of the van Geemen-Sarti-Nikulin duality, as $H \oplus N$ has no embedding into $P''_{14}$.  

\subsection{Motivation and general overview}
\label{ssec:general_discussion}
This article extends previous work of the authors and their collaborators for K3 surfaces of high Picard rank \cites{MR2369941, MR2824841, CM:2018b, MR3767270, MR2935386, MR2854198, MR3366121, MR3712162,  MR3798883, MR3995925, MR3992148, MR4015343, MR4099481, CHM19, MR4160930, CMS:2020}. The present study also builds on several other works \cites{MR0429917, MR1013073, MR894512, MR1023921, MR1877757, MR1013162, MR1458752, MR2409557, MR1703212, MR2427457, MR3263663, MR2306633, MR728142, MR3563178, MR3010125}. The nontrivial connection between families of K3 surfaces, their polarizing lattices, and compatible automorphic forms appears in string theory as the eight-dimensional manifestation of the phenomenon called the F-theory/heterotic string duality. This viewpoint has been studied in \cites{MR3366121, MR3274790, MR3712162, MR3417046, MR3933163, MR4160930, MR3991815}.
\par In Picard rank eighteen, a Kummer surface $\mathcal{Y}=\operatorname{Kum}(E_1 \times E_2)$ associated with two non-isogenous elliptic curves $E_1, E_2$ admits several inequivalent elliptic fibrations; see \cites{MR1013073, MR2409557}. It follows\footnote{There is an elliptic fibration with trivial Mordell-Weil group and singular fibers $II^*+ 2 I_0^* + 2 I_1$, labelled $\mathcal{J}_9$ in \cite{MR2409557}. In addition, fibrations $\mathcal{J}_{10}$, $\mathcal{J}_{11}$ provide the equivalent descriptions of the lattice.} that these Kummer surfaces are polarized by the lattice 
\beqn
 H \oplus E_8(-1) \oplus D_4(-1)^{\oplus 2} \cong H \oplus D_{12}(-1) \oplus D_4(-1) \cong H \oplus D_8(-1)^{\oplus 2} \,. 
\eeqn
The surfaces $\mathcal{Y}$ admit an alternate fibration with a Mordell-Weil group that contains a 2-torsion section, and a van Geemen-Sarti involution can be constructed. New K3 surfaces $\mathcal{X}$ are then obtained via the Nikulin construction. We shall refer to $\mathcal{X}$ as the \emph{Inose K3 surfaces} as they admit a birational model isomorphic to a projective quartic surface introduced by Inose \cite{MR578868}. They are polarized by the lattice $H \oplus E_8(-1) \oplus E_8(-1)$; see \cite{MR2369941}. 
\par The entire picture generalizes to Picard rank seventeen: here, the elliptic fibrations on the Jacobian Kummer surfaces $\mathcal{Y}$ were classified in \cite{MR3263663}, and the Kummer surfaces are polarized by the lattice\footnote{This follows from the existence of fibrations {\tt (15)} and {\tt (17)} in \cite{MR3263663}.} 
\beqn
 H \oplus D_7(-1) \oplus D_4(-1)^{\oplus 2} \cong H \oplus D_8(-1) \oplus D_4(-1) \oplus A_3(-1) \,.
\eeqn
The (generalized) Inose K3 surfaces $\mathcal{X}$ are obtained in a similar manner as before and polarized by the rank seventeen lattice $H \oplus E_8(-1) \oplus E_7(-1)$;  the details may be found in \cites{MR2427457, MR2824841, MR2935386}. The Inose K3 surfaces $\mathcal{X}$ can also be viewed as K3 surfaces admitting \emph{Shioda-Inose structures}; see \cites{MR728142, MR2279280, MR3087091}.
\par Aspects of this construction were generalized for K3 surfaces of lower Picard rank in \cites{MR2824841, MR2254405, MR4015343, CHM19}. Since there are no Kummer surfaces of Picard rank lower than seventeen, those needed to be replaced by other K3 surfaces; a suitable choice for Picard number sixteen turned out to be the surfaces $\mathcal{Y}$ obtained as double covers of the projective plane branched over the union of six lines. In this way, the rank-seventeen case is recovered by making the six lines tangent to a common conic. The surfaces $\mathcal{Y}$ are polarized\footnote{This follows from the existence of fibration  {\tt (2.10)} in \cite{MR2254405}.} by the lattice $H \oplus D_6(-1) \oplus D_4(-1)^{\oplus 2}$. Their moduli are well understood and are related to Abelian fourfolds of Weil type~\cites{MR3506391, MR1335243}. Via the van Geemen-Sarti-Nikulin duality one obtains the (generalized) Inose K3 surfaces $\mathcal{X}$ of Picard rank sixteen which are polarized by the lattice $H \oplus E_8(-1) \oplus D_6(-1)$; see \cite{CHM19}.  
\par The cases discussed above share some commonalities: (i) the double sextic K3 surfaces $\mathcal{Y}$  have a concrete geometric construction, derived from special reducible projective sextic curves that form their branch loci; (ii) the Inose K3 surfaces $\mathcal{X}$ are polarized by 2-elementary lattices. 

The present work originated in the authors’ effort to extend the above construction to K3 families of Picard rank lower than $16$. We were initially able to explicitly described the behavior of the van Geemen-Sarti-Nikulin duality in the context of K3 surfaces $ \mathcal{X}$ polarized by the lattice $H \oplus E_7(-1) \oplus D_6(-1)$. Subsequently, we realized that our arguments may be extended to the $ H \oplus E_7(-1) \oplus D_4(-1) \oplus A_1(-1) $ polarization. A summary of this extension is presented in Table~\ref{table:cases}. Ultimately, we were able to obtain an explicit classification of K3 surfaces $\mathcal{X}$, extending to all possible rank-fourteen lattice polarizations of 2-elementary type.    
\begin{table}
\beqn
\begin{array}{|l|l|l|}
\hline
\text{rank}				& \text{Inose K3 surface} \, \mathcal{X}					& \text{double sextic K3 surface} \, \mathcal{Y} \\
\hdashline
					& \text{polarizing lattice \& discriminant}	 				& \text{polarizing lattice \& construction} \\
\hdashline
					& \multicolumn{2}{c|}{\text{applicable moduli in Theorems~\ref{thm:intro} and~\ref{thm:intro2}}} \\
\hline 
\rho = 14				& H \oplus E_8(-1) \oplus A_1(-1)^{\oplus 4}				& H  \oplus D_4(-1)^{\oplus 3}\\
					& D = \mathbb{Z}_2^4								& \text{double sextic of 3 lines and cubic}\\
\hdashline

					& \multicolumn{2}{c|}{[  J_4 :  J'_4:   J_6:  J'_6:  J_8: J_{10}: J_{12} ]  \; \text{or} \;  (\alpha, \beta, \gamma, \delta , \varepsilon, \zeta, \eta, \iota, \kappa, \lambda) } \\
\hline
\rho = 15				& H \oplus E_8(-1) \oplus D_4(-1) \oplus A_1(-1)			& H \oplus D_5(-1) \oplus D_4(-1)^{\oplus 2}\\
					& D = \mathbb{Z}_2^3								& \text{double sextic of 3 lines and tangent cubic} \\
\hdashline
					& \multicolumn{2}{c|}{J'_4=0 \; \text{or} \; (\kappa, \lambda) = (0,1) } \\

\hline 
\rho = 16				& H \oplus E_8(-1) \oplus D_6(-1)						& H \oplus D_6(-1) \oplus D_4(-1)^{\oplus 2}\\
					& D = \mathbb{Z}_2^2								& \text{double sextic of 6 lines} \\
\hdashline
					& \multicolumn{2}{c|}{J'_4 = J'_6 =0 \; \text{or} \; (\eta, \iota)  = (\kappa, \lambda) = (0,1) } \\
\hline
\rho = 17				& H \oplus E_8(-1) \oplus E_7(-1)						& H \oplus D_7(-1) \oplus D_4(-1)^{\oplus 2}\\
					& D = \mathbb{Z}_2									& \text{Jacobian Kummer surface} \\
\hdashline
					& \multicolumn{2}{c|}{J'_4=J'_6 = J_8 =0 \; \text{or} \; (\eta, \iota)  = (\kappa, \lambda) = (\epsilon, \zeta) = (0,1) } \\
\hline
\rho = 18				& H \oplus E_8(-1) \oplus E_8(-1)						& H \oplus E_8(-1) \oplus D_4(-1)^{\oplus 2}\\
					&  D = \{ \mathbb{I} \}								& \text{Kummer surface} \, \operatorname{Kum}(E_1 \times E_2)\\
\hdashline
					& \multicolumn{2}{c|}{J'_4=J'_6 = J_8 = J_{10} = 0\; \text{or} \; (\eta, \iota)  = (\kappa, \lambda) = (\epsilon, \zeta) =(\gamma,\delta) = (0,1) } \\
\hline		
\end{array}
\eeqn
\caption{van Geemen-Sarti-Nikulin duality for K3 surfaces}
\label{table:cases}
\end{table}
\par In the situation above, a description of the moduli space for Picard number seventeen and sixteen in terms of suitable Siegel modular forms or automorphic forms was given in \cites{MR2682724, MR3235787, MR4015343,Nagano2020}.  Let us also connect our previous discussion with Vinberg's seminal work in \cite{MR2718942}: considering algebras of automorphic forms on the bounded symmetric domains of type $IV$, the author constructed families of  K3 surfaces of Picard rank $20-n$ for $4 \le n \le 7$ whose moduli spaces have a function field freely generated by the modular forms on the $n$-dimensional symmetric domain $\mathpzc{D}_n=D_{IV}(n)$ of type $IV$ with respect to the lattice $\Gamma_n=O(2,n; \mathbb{Z})^+$, i.e., all matrices with integer entries in $O(2,n)^+$. Here, the plus sign refers to a certain index-two subgroup of the pseudo-orthogonal group $O(2,n)$. The natural algebra of automorphic forms $A(\mathpzc{D}_n, \Gamma_n)$ on $\mathpzc{D}_n$ with respect to $\Gamma_n$ is freely generated by forms of the weights indicated in the following table:
\beq
\label{eqn:generators}
 \begin{array}{c|l} 
 n & \text{weights} \\
 \hline
 4 & 4, 6, 8, 10, 12 \\
 5 & 4, 6, 8, 10, 12, 18 \\
 6 & 4, 6, 8, 10, 12, 16, 18 \\
 7 & 4, 6, 8, 10, 12, 14, 16, 18
 \end{array}
\eeq
The corresponding K3 surfaces were obtained as families of quartic projective surfaces in \cite{MR2718942}.  As we will prove in Theorem~\ref{thm_Vinberg}, these families of K3 surfaces are polarized by the following lattices:
\beq
\label{eqn:lattices}
 \begin{array}{c|l} 
 n & \text{polarizing lattice} \\
 \hline
 4 & H \oplus E_8(-1) \oplus D_6(-1)\\
 5 & H \oplus E_8(-1) \oplus D_5(-1) \\
 6 & H \oplus E_8(-1) \oplus D_4(-1) \\
 7 & H \oplus E_8(-1) \oplus A_3(-1)
 \end{array}
\eeq
We will prove in Theorem~\ref{thm_Vinberg} that for $5 \le n \le 7$ the corresponding K3 surfaces admit exactly two Jacobian elliptic fibrations, both with a trivial Mordell-Weil group.  Since there is no elliptic fibration with a Mordell-Weil group containing a 2-torsion section, there is no notion of van Geemen-Sarti-Nikulin duality in this case.  However, for $n=4$, the Vinberg family coincides with the family in Equation~(\ref{quartic1}) for $(\eta,\iota)=(\kappa, \lambda)=(0,1)$; see Proposition~\ref{prop:coincide}. The invariants defined in Theorem~\ref{thm:intro2} are then precisely the generators of $A(\mathpzc{D}_4, \Gamma_4)$ in Equation~(\ref{eqn:generators}) defined by Vinberg.  The explicit expressions for these generators in terms of automorphic forms and theta function were given in \cites{MR1103969, MR1204828, MR4015343} and are a direct consequence of the coincidence of two different bounded symmetric domains, namely the domains $D_{IV}(4)$ and $I_{2,2}$.  
\medskip
\par This article is structured as follows: In Section~\ref{sec:lattice} we carry out  a brief lattice-theoretic investigation regarding the possible Jacobian elliptic fibrations appearing on the surfaces $\mathcal{X}$, $\mathcal{X}'$, and $\mathcal{X}''$ in Theorems~\ref{thm:intro},~\ref{thm:pprime}, and~\ref{thm_Vinberg_body}, respectively. We then show that the existence of a unique alternate fibration on $\mathcal{X}$ and $\mathcal{X}'$ allows for the construction of their coarse moduli spaces. In Section~\ref{sec:projective_models} we construct birational projective models for the K3 surfaces $\mathcal{X}$, $\mathcal{X}'$, and $\mathcal{X}''$ with N\'eron-Severi lattices $P_{14}$, $P'_{14}$, and $P''_{14}$, respectively. In Section~\ref{sec:dual_graphs} we determine the dual graphs of all smooth rational curves. To our knowledge, for a $P_{14}$-polarization or $P'_{14}$-polarization these graphs have not appeared in the literature previously. In Section~\ref{sec:Kummer} we  construct the family of K3 surfaces $\mathcal{Y}$, obtained from the family of Inose K3 surfaces $\mathcal{X}$ using the van Geemen-Sarti-Nikulin duality. In Appendix~\ref{sec:rank15} we determine the dual graph of  rational curves on a general K3 surface $\mathcal{X}$ of Picard number 15.
\subsection{Acknowledgments}
The authors would like to thank the referees for their insightful comments, in particular with regards to improving the exposition in Sections~2 and~5. A.C. acknowledges support from a UMSL Mid-Career Research Grant. A.M. acknowledges support from the Simons Foundation through grant no.~202367.
\section{Lattice theoretic considerations for certain K3 surfaces}
\label{sec:lattice}
We start with a brief lattice-theoretic investigation regarding the possible Jacobian elliptic fibration structures appearing on the surface $\mathcal{X}$, $\mathcal{X}'$, and $\mathcal{X}''$ . Recall that a \emph{Jacobian elliptic fibration} on $\mathcal{X}$ is a pair $(\pi,\sigma)$ consisting of a proper map of analytic spaces $\pi\colon \mathcal{X} \to \mathbb{P}^1$, whose general fiber is a smooth curve of genus one, and a section $\sigma\colon \mathbb{P}^1 \to \mathcal{X}$ in the elliptic fibration $\pi$. The group of section of the Jacobian fibration is the \emph{Mordell-Weil group} $\operatorname{MW}(\pi,\sigma)$.
\subsection{Jacobian elliptic fibrations}
Nikulin  \cites{MR544937,MR633160,MR633160b,MR752938,MR3165023} determined the hyperbolic 2-elementary primitive sub-lattices of the K3 lattice which are the Picard lattice of K3 surfaces with finite automorphism groups. There are exactly three such lattices of rank 14, namely the lattices in~(\ref{eqn:lattices_intro}).  We observe that there are two different 2-elementary lattices whose determinant of the discriminant form is $2^4$, but they have different parity, as defined in \cites{MR728992,MR1029967}. 
\par We state the following lemmas covering the rank-fourteen lattices that are the main topic of this article. For convenience, we also include two lattices of other ranks for which similar results apply.
\begin{lemma}
\label{lem:lattice}
Let $\mathcal{X}$ be a general $L$-polarized K3 surface where $L$ is a lattice in~(\ref{eqn:lattices_list}) with the given $\operatorname{rank}$, signature $\operatorname{sign}$, and discriminant group $D(L)$:
\beq
\label{eqn:lattices_list}
 \begin{array}{rl|c|c|c}
 \multicolumn{2}{c|}{ L}  & \operatorname{rank} & \operatorname{sign} & D(L) \\
 \hline
 &&&&\\[-1em]
  P''_{13} = & H \oplus E_8(-1) \oplus A_3(-1)  			& 13 & (1,12) & \mathbb{Z}_4\\
 \hdashline
 P''_{14} =& H \oplus E_8(-1) \oplus D_4(-1) 			& 14 & (1,13) & \mathbb{Z}_2^2\\
 P'_{14} =& H \oplus D_8(-1) \oplus D_4(-1) 			& 14 & (1,13) & \mathbb{Z}_2^4\\
 P_{14} =& H \oplus E_8(-1) \oplus A_1(-1)^{\oplus 4} 	& 14 & (1,13) & \mathbb{Z}_2^4\\
 \hdashline
  P_{15} =& H \oplus E_7(-1) \oplus D_6(-1)  			& 15 & (1,14) & \mathbb{Z}_2^3
 \end{array}
\eeq
Let $(\pi,\sigma)$ be a Jacobian elliptic fibration on $\mathcal{X}$. Then, the Mordell-Weil group has finite order. In particular, we have
\beq
 \operatorname{rank} \operatorname{MW}(\pi,\sigma) = 0 \,.
\eeq
\end{lemma} 
\begin{proof}
For a given $\operatorname{NS}(\mathcal{X})$, it follows, via work of Nikulin \cites{MR544937,MR633160,MR633160b,MR752938,MR3165023} and Kondo \cite{MR1029967}, that the group of automorphisms of $\mathcal{X}$ is finite. We have $\operatorname{Aut}(\mathcal{X}) \simeq \mathbb{Z}_2 \times \mathbb{Z}_2$ for the last four cases and $\operatorname{Aut}(\mathcal{X}) \simeq \mathbb{Z}_2$ for the first one.  Any Jacobian elliptic fibration on $\mathcal{X}$ must have a Mordell-Weil group of finite order and cannot admit an infinite-order section.
\end{proof}
\par Given a Jacobian elliptic fibration $(\pi,\sigma)$ on $\mathcal{X}$, the classes of fiber and section span a rank-two primitive sub-lattice of $\operatorname{NS}(\mathcal{X})$ which is isomorphic to the standard rank-two hyperbolic lattice $H$. The converse also holds: given a primitive lattice embedding $H \hookrightarrow \operatorname{NS}(\mathcal{X})$ whose image contains a pseudo-ample class, it is known from \cite{MR2355598}*{Thm.~2.3} that there exists a Jacobian elliptic fibration on the surface $\mathcal{X}$, whose fiber and section classes span $H$. For a primitive lattice embedding $j \colon H \hookrightarrow L$  we denote by $K= j(H)^{\perp}$ the orthogonal complement in $L$ and by $K^{\text{root}}$ the sub-lattice spanned by the roots of $K$, i.e., the algebraic class of self-intersection $-2$  in $K$. We also introduce the factor group $\mathpzc{W}= K/ K^{\text{root}}$.  The pair $(K^{\text{root}}(-1), \mathpzc{W})$ is called the \emph{frame} associated with the Jacobian elliptic fibration.  

The classification of Jacobian elliptic fibrations on K3 surfaces with 2-elementary Picard lattice given in \cite{MR4130832} immediately implies the following:
\begin{lemma}
\label{prop:ADE}
Let $L$ be a lattice in~(\ref{eqn:lattices_list}). A general $L$-polarized K3 surface admits exactly the Jacobian elliptic fibrations $(\pi,\sigma)$, up to isomorphism, associated with the following frames:
\begin{enumerate}
\item For $L=P''_{13}$ the inequivalent frames are
\beqn
  (E_8  \oplus A_3, \mathbb{I}),   \qquad (D_{11} , \mathbb{I}).
\eeqn
\item For $L=P''_{14}$ the inequivalent frames are
\beqn
  (E_8  \oplus D_4, \mathbb{I}),   \qquad (D_{12} , \mathbb{I}).
\eeqn
\item For $L=P'_{14}$ the inequivalent frames are
\beqn
 (D_8  \oplus D_4, \mathbb{I}) ,  \qquad  (E_7  \oplus A_1 ^{\oplus 5}, \mathbb{Z}/2\mathbb{Z}).
\eeqn  
\item For $L=P_{14}$ the inequivalent frames are
\beqn
 ( D_6 \oplus D_6, \mathbb{I}) ,  \ (D_{10}  \oplus A_1 ^{\oplus 2}, \mathbb{I}) ,  \ (E_7  \oplus D_4   \oplus A_1, \mathbb{I}) ,   \ (E_8  \oplus A_1 ^{\oplus 4}, \mathbb{I}), \ (D_{8} \oplus A_1^{\oplus 4}, \mathbb{Z}/2\mathbb{Z}).
\eeqn
\item For $L=P_{15}$  the inequivalent frames are
\beqn
 (E_7 \oplus D_6, \mathbb{I}) , \ \ (E_8 \oplus D_4 \oplus A_1, \mathbb{I}) , \ \ (D_{12}  \oplus A_1, \mathbb{I}), \ \ ( D_{10} \oplus A_1^{\oplus 3}, \mathbb{Z}/2\mathbb{Z}) .
\eeqn
\end{enumerate}
\end{lemma}
\par Lemma~\ref{prop:ADE} classifies the root lattices attached to the reducible fibers and the Mordell-Weil group associated with a Jacobian elliptic fibration. One can also ask for a more precise classification, up to automorphisms of the surfaces. The differences between the possible classifications is explained in detail in \cite{Braun:2013aa}. In the situation above, assume that we have a second primitive embedding $j' \colon H \hookrightarrow L$, such that the orthogonal complement of the image $j'(H)$, is isomorphic to the lattice $K$. One would like to see whether $j$ and $j'$ correspond to Jacobian elliptic fibrations isomorphic under $\mathrm{Aut}(\mathcal{X})$ or not. The number of distinct primitive lattice embeddings $H \hookrightarrow L$ is referred by Festi and Veniani as the \emph{multiplicity} of the frame.   As proved in  \cite{FestiVeniani20}*{Thm~2.8}, the multiplicity of the frame of an elliptic fibration can be computed using lattice-theoretic arguments.
We have the following:
\begin{proposition}
\label{cor:EFS}
\label{cor:EFS2}
\label{cor:EFS3}
\label{cor:EFS4}
\label{cor:EFS5}
Let $L$ be a lattice in~(\ref{eqn:lattices_list}).  For a general $L$-polarized K3 surface $\mathcal{X}$ the multiplicity associated with the following frames equals one:
\beqn
\label{eqn:lattices_list2}
 \begin{array}{rl|ll}
 \multicolumn{2}{c|}{ L}  &  \multicolumn{2}{|l}{ \big(K^{\text{root}}(-1), \mathpzc{W}\big) \ \text{with multiplicity 1}}\\[0.1em]
 \hline
 &&\\[-1em]
 P''_{13} =& H \oplus E_8(-1) \oplus A_3(-1)  			&  \big( E_8 \oplus A_3, \{ \mathbb{I} \}\big)&\\
 \hdashline
 P''_{14} =& H \oplus E_8(-1) \oplus D_4(-1) 					&  \big( E_8 \oplus D_4, \{ \mathbb{I} \}\big)&\\
 P'_{14} =& H \oplus D_8(-1) \oplus D_4(-1) 					& \big(E_7 \oplus A_1^{\oplus 5}, \mathbb{Z}/2\mathbb{Z}\big)\\
 P_{14} =& H \oplus E_8(-1) \oplus A_1(-1)^{\oplus 4} 			& \big(D_{8} \oplus A_1^{\oplus 4}, \mathbb{Z}/2\mathbb{Z}\big), & \big(E_{8} \oplus A_1^{\oplus 4}, \{ \mathbb{I} \}\big)\\
 \hdashline
  P_{15} =& H \oplus E_7(-1) \oplus D_6(-1)  				& \big(D_{10} \oplus A_1^{\oplus 3}, \mathbb{Z}/2\mathbb{Z}\big), &  \big(E_{8} \oplus D_4 \oplus A_1, \{ \mathbb{I} \}\big)\\
 \end{array}
\eeqn
\end{proposition}
\begin{proof}
For the frames in Lemma~\ref{prop:ADE} we compute the multiplicity, following the procedure outlined in \cite{CHM19}, using the Sage class {\tt QuadraticForm}. 
\end{proof}
We make the following:
\begin{remark}
Families of $P'_{14}$-polarized K3 surfaces and $P''_{14}$-polarized K3 surfaces also appear in Reid’s list of ``Famous 95 Families" of Gorenstein K3 surfaces; see \cite{MR1880661}*{Table~3}. They occur as surfaces in weighted projective three-space.
\end{remark}
\subsection{Coarse moduli spaces}
Recall that a {\it Nikulin involution} \cites{MR728142,MR544937} is an involution $\imath_{\mathcal{X}} \colon \mathcal{X} \rightarrow \mathcal{X}$ on a K3 surface $\mathcal{X}$ that satisfies $\imath_{\mathcal{X}}^*(\omega) = \omega $ for any holomorphic 2-form $\omega$ on $\mathcal{X}$.  When quotienting by this involution and blowing up the fixed locus, one obtains a new K3 surface $\mathcal{Y}$ together with a rational double cover map $ \Phi \colon \mathcal{X} \dashrightarrow \mathcal{Y}$. In general, a Nikulin involution does not determine a Hodge isometry between the transcendental lattices $\mathrm{T}_\mathcal{X}(2)$ and $\mathrm{T}_\mathcal{Y}$. 
\par When a K3 surface $\mathcal{X}$ admits a Jacobian elliptic fibration with a 2-torsion section, then $\mathcal{X}$ admits a special Nikulin involution, called \emph{van Geemen-Sarti involution}; see  \cite{MR2274533}.  The corresponding Jacobian elliptic fibration $\pi_\mathcal{X}\colon \mathcal{X} \to \mathbb{P}^1$ is called an \emph{alternate fibration}; see \cite{MR3995925} for the nomenclature. The van Geemen-Sarti involution $\jmath_\mathcal{X}$ is obtained as the fiber-wise translation by 2-torsion in the alternate fibration. Moreover, the construction induces a Jacobian elliptic fibration $\pi_\mathcal{Y}\colon \mathcal{Y} \to \mathbb{P}^1$ on $\mathcal{Y}$ which in turn also admits a 2-torsion section as well. Thus, we obtain the following diagram:
\beq
 \label{diag:isogeny}
\xymatrix{
\mathcal{X} \ar @(dl,ul) _{\jmath_\mathcal{X} } \ar [dr] _{\pi_\mathcal{X} } \ar @/_0.5pc/ @{-->} _{\Phi} [rr]
&
& \mathcal{Y} \ar @(dr,ur) ^{\jmath_\mathcal{Y} } \ar [dl] ^{\pi_\mathcal{Y} } \ar @/_0.5pc/ @{-->} _{\check{\Phi}} [ll] \\
& \mathbb{P}^1 }
\eeq
As mentioned in the introduction, we will refer to the construction of Diagram~(\ref{diag:isogeny}) as \emph{van Geemen-Sarti-Nikulin duality}. We make the following:
\begin{remark}
Consider the K3 surfaces polarized by the lattices in Equation~(\ref{eqn:lattices_intro}). Only for the polarizing lattices $P_{14}$ and $P'_{14}$ is there an alternate fibration, allowing for the construction of a van Geemen-Sarti-Nikulin duality.
\end{remark}
\subsubsection{The case of \texorpdfstring{$P_{14}$}{P}-polarized K3 surfaces}
\label{ssec:moduli_space}
We consider the Jacobian elliptic K3 surface $\mathcal{X}$ with one singular fiber of type $I_{2n}^*$ with $n \ge 2$ and a 2-torsion section. Here, we are using the Kodaira classification for singular fibers for Jacobian elliptic fibrations \cite{MR0184257}. A Weierstrass model for such a fibration $\pi_\mathcal{X}\colon  \mathcal{X} \rightarrow \mathbb{P}^1$ -- with fibers in $\mathbb{P}^2 = \mathbb{P}(x, y, z)$ varying over $\mathbb{P}^1=\mathbb{P}(u, v)$ -- is given by
\beq
\label{eqn:alt0}
  \mathcal{X}\colon \quad y^2 z  = x^3 + v  A(u, v) \, x^2 z + v^4 B(u, v) \, x z^2  \,,
\eeq
where $A$ and $B$ are polynomials of degree three and four, respectively. If the Weierstrass model is minimal, the polynomial $A(t,1)$ has a non-vanishing cubic coefficient. The fibration admits the section $\sigma\colon [x:y:z] = [0:1:0]$ and the 2-torsion section $[x:y:z] = [0:0:1]$, and has the discriminant
\beq
\label{eqn:Delta_X}
\Delta_\mathcal{X} =  v^{10} B(u, v)^2 \, \Big(A(u, v)^2-4 \,v^2 B(u, v)\Big) \,.
\eeq
On the elliptic fibration~(\ref{eqn:alt0}) the translation by 2-torsion acts  fiberwise as
\beq
\label{eqn:VGS_involution}
 \jmath_\mathcal{X} \colon \quad \Big[ x : y : z \Big] \mapsto \Big[ v^4 B(u, v) \, x z \  : - v^4 B(u, v) \, y z \ : \ x^2 \Big] 
\eeq 
for $[x:y:z] \not= [0:1:0], [0:0:1]$, and by swapping $[0:1:0] \leftrightarrow [0:0:1]$. 
\par The minimal resolution of the quotient surface $\mathcal{Y}= \widehat{\mathcal{X}/\langle \jmath_\mathcal{X}  \rangle}$ admits the induced elliptic fibration $\pi_\mathcal{Y}\colon  \mathcal{Y} \rightarrow \mathbb{P}^1$ given by
\beq
\label{eqn:alt_SI}
  \mathcal{Y}\colon \quad y^2 z  = x^3 -2 vA(u, v) \, x^2 z + v^2 \Big(A(u, v)^2 - 4 v^2 B(u, v)\Big) x z^2  \,,
\eeq
with the discriminant
\beq
\Delta_\mathcal{Y} =  16 \, v^{6} B(u, v)\, \Big(A(u, v)^2-4 \,v^2 B(u, v)\Big)^2 \,.
\eeq
We make the following:
\begin{remark}
\label{rem:gauge_fixing0}
By rescaling $(x, y, z) \to (\Lambda^2 x, \Lambda^3 y, z)$ and changing $u \mapsto a u+b v$, we can assume that $A(t,1)$ and the sextic $S(t) = A(t,1)^2 - 4 B(t, 1)$ in Equation~(\ref{eqn:alt_SI}) are monic polynomials of degree three and six, respectively, whose sub-leading coefficient proportional to $t^2$ (resp.~\!$t^5$) vanishes. 
\end{remark}
In the following, we will assume that the polynomials $A$ and $B$ are as follows:
\beq
\label{eqn:AB}
 A(u, v) = u^3 + a_1 uv^2 + a_0 v^3\,, \qquad B(u, v) = b_4 u^4 + b_3 u^3v + b_2 u^2v^2 + b_1 uv^3 + b_0 v^4\,.
\eeq
We have the following:
\begin{lemma}
\label{lem:alternate_fibrations_extensions}
General K3 surfaces $\mathcal{X}$ and $\mathcal{Y}$ admit Jacobian elliptic fibrations with singular fibers $I_4^* + 4 I_2 + 6 I_1$ and $I_2^* + 4 I_1 + 6 I_2$, respectively, and Mordell-Weil group $\mathbb{Z}/2\mathbb{Z}$. The singular fibers are $I_6^* + 3 I_2 + 6 I_1$ and $I_3^* + 3 I_1 + 6 I_2$ if and only if $b_4=0$, and singular fibers $I_8^* + 2 I_2 + 6 I_1$ and $I_4^* + 2 I_1 + 6 I_2$ if and only if $b_3=b_4=0$, and the remaining parameters are general. 
\end{lemma}
\begin{proof}
The statements are checked directly using the Weierstrass models in Equation~(\ref{eqn:alt0}) and~(\ref{eqn:alt_SI}). As for the K3 surface $\mathcal{Y}$, by construction the Mordell-Weil group of $\mathcal{Y}$ must contain the subgroup $\mathbb{Z}/2\mathbb{Z}$. It cannot have any additional sections of infinite order. Comparing with the list in \cite{MR1813537} shows that the Mordell-Weil group is indeed $\mathbb{Z}/2\mathbb{Z}$.
\end{proof}
Similar to the lattices $P_{14} \subset P_{15} \subset P_{16}$,  let us also consider the lattices given as
\beq
\label{eqn:Rlattices}
 R_{14} = H \oplus D_4(-1)^{\oplus 3}  \  \subset \ R_{15} = H \oplus D_5(-1) \oplus D_4(-1)^{\oplus 2}  \ \subset \   R_{16} =  H \oplus D_6(-1) \oplus D_4(-1)^{\oplus 2}.
\eeq
We have the following:
\begin{proposition}
\label{prop:fibrations_SI}
General K3 surfaces $\mathcal{X}$ and $\mathcal{Y}$ have a N\'eron-Severi lattice isomorphic to $P_{14}$ and $R_{14}$, respectively. The N\'eron-Severi lattices of $\mathcal{X}$ and $\mathcal{Y}$ are $P_{15}$ and $R_{15}$, respectively, if $b_4=0$, and $P_{16}$ and $R_{16}$, respectively, if $b_3=b_4=0$.
\end{proposition}
\begin{proof}
The statement follows directly from Lemma~\ref{lem:alternate_fibrations_extensions} and standard lattice-theoretic arguments.
\end{proof}
We introduce parameters
\beq
\label{eqn:invariants_alternate_fibration}
 \Big[  J_4 :  J'_4 : J_6:  J'_6 :  J_8:  J_{10}:  J_{12}   \Big] = \left[ - \frac{a_1}{3} :  b_4 : -\frac{a_2}{2} : -b_3 : b_2 : - b_1 : b_0 \right] \,,
\eeq 
and obtain the following:
\begin{theorem}
\label{thm:moduli_space_alt1}
$\mathscr{M}_P$ in Equation~(\ref{eqn:moduli_space_P}) is a coarse moduli space of $P_{14}$-polarized K3 surfaces. Here, a K3 surface $\mathcal{X} \in \mathscr{M}_P$ is the minimal resolution of Equation~(\ref{eqn:alt0}). Moreover, the coarse moduli space of $P_{15}$-polarized K3 surfaces is the subspace $J'_4=0$; the  coarse moduli space of $P_{16}$-polarized K3 surfaces is the subspace $J'_4=J'_6=0$. 
\end{theorem}
\begin{proof}
Because of Proposition~\ref{cor:EFS2}, every $P_{14}$-polarized K3 surface, up to isomorphism, admits a unique alternate fibration that can be brought into the form of Equation~(\ref{eqn:alt0}). Thus, we obtain the moduli space as the following open variety:
 \beq
\left\{ 
\begin{array}{c} 
 \Big\lbrack a_1 : a_2 : b_4 : b_3 : b_2 : b_1 : b_0 \Big \rbrack \\
 \in \mathbb{WP}_{(4,6,4,6,8,10,12)}
 \end{array}
 \left\vert 
\begin{array}{l}
{ \scriptstyle(b_4, b_5, b_6, b_7, b_8) \neq 0 , }\\
{ \scriptstyle \not \exists \, r, s \, \in \, \mathbb{C}: \ (a_1, a_2) = ( -3 r^2, -2 r^3) \; \text{and}} \\
{ \scriptstyle (b_4, b_3, b_2, b_1, b_0) = (s, 4 r s, 6r^2 s, 4r^3 s, r^4 s)}
\end{array}  \right. \right\} \,.
\eeq
Moreover, one can tell precisely when two members of the family in Equation~(\ref{eqn:alt0}) are isomorphic. The normalization of the coefficients in Equation~(\ref{eqn:AB}) fixes the coordinates $[u:v] \in \mathbb{P}^1$ completely; see Remark~\ref{rem:gauge_fixing0}.  Thus,  two members  are isomorphic if and only if their coefficient sets are related by the transformation 
\beq
 \Big( a_1, b_4, a_0, b_3 ,b_2, b_1, b_0 \Big) \  \mapsto \ \Big( \Lambda^4 a_1,  \Lambda^4 b_4, \Lambda^6 a_0, \Lambda^6 b_3, \Lambda^8 b_2, \Lambda^{10} b_1, \Lambda^{12} b_0 \Big) \,,
\eeq
with $\Lambda \in \mathbb{C}^\times$. The reason is that such a rescaling, when combined with the transformation $(u, v, x, y, z) \mapsto (\Lambda^2 u, v, \Lambda^{6} x, \Lambda^{9} y, z)$, gives rise to a holomorphic isomorphism of Equation~(\ref{eqn:alt0}). Conversely, an equivalence class of invariants determines a well defined K3 surface as long as the Weierstrass model is irreducible and minimal.
\par Bringing Equation~(\ref{eqn:alt0}) into the standard Weierstrass normal form, we obtain
\beq
\label{eqn:WE_AB}
 y^2 z = x^3 - 3 v^2 \big(A(u, v)^2 - 3  v^2 B(u, v) \big) x z^2 + v^3 A(u, v)  \big(2 A(u, v)^2 - 9 v^2 B(u, v) \big) z^3\,.
\eeq 
For $B \equiv 0$ the Weierstrass model becomes $y^2 z =(x+ 2 v Az)(x - v A z)^2$. Thus, for the Weierstrass model in Equation~(\ref{eqn:alt0}) to determine a K3 surface  $B$ must not vanish identically. If $B \neq 0$ and if there is no polynomial $c \in \mathbb{C}[u, v]$ so that $c^2$ divides $a$ and $c^4$ divides $b$, then the minimal resolution of Equation~(\ref{eqn:alt0}) is a K3 surface.  The latter occurs if and only if there are $r, b_4 \in \mathbb{C}$ such that $(a_1, a_2) = ( -3 r^2, -2 r^3)$ and
and $(b_5, b_6, b_7, b_8) = (4 r b_4, 6r^2b_4, 4r^3 b_4, r^4b_4)$. Then,  $(u+r v)^2$ divides $A$ and $(u+r v)^4$ divides $B$. Because of Proposition~\ref{prop:fibrations_SI}, Equation~(\ref{eqn:alt0}) becomes a Jacobian elliptic fibration on a general $P_{15}$-polarized K3 surface $\mathcal{Y}$ if $b_4=0$.  The last statement follows from Proposition~\ref{prop:fibrations_SI} and by comparison with results already proved in \cite{MR4015343}.
\end{proof}
\begin{remark}
\label{rem:P}
The proof uses the fact that every $P_{14}$-polarized K3 surface, up to isomorphism, admits a unique alternate fibration~(\ref{eqn:alt0}). Thus, there is a canonical lattice embedding  $H \oplus N \hookrightarrow P_{14}$, and  every $P_{14}$-polarized K3 surface carries a unique underlying $H \oplus N$-polarization. 
\end{remark}
\begin{remark}
For $J'_4 = J'_6=0$ the remaining are the generators of $A(\mathpzc{D}_4, \Gamma_4)$ in Equation~(\ref{eqn:generators}); see Proposition~\ref{prop:coincide}.
\end{remark}
\subsubsection{The case of \texorpdfstring{$P'_{14}$}{P'}-polarized K3 surfaces}
\label{sssec:PPmoduli}
Our approach from Section~\ref{ssec:moduli_space} can also be used to construct a moduli space for K3 surfaces of Picard number 14 for which the types of singular fibers of the alternate fibration do not change under the action of a van Geemen-Sarti involution. A Weierstrass model for such a Jacobian elliptic fibration $\pi_{\mathcal{X}'}\colon  \mathcal{X}' \rightarrow \mathbb{P}^1$ is given by
\beq
\label{eqn:alt_E7}
  y^2 z  = x^3 + v^2  C(u, v) \, x^2 z + v^3 D(u, v) \, x z^2  \,,
\eeq
where $C$ and $D$ are polynomials of degree two and five, respectively. If the Weierstrass model is minimal, the polynomial $D(t,1)$ has a non-vanishing quintic coefficient. The fibration obviously admits the section $\sigma\colon [x:y:z] = [0:1:0]$ and the 2-torsion section $[x:y:z] = [0:0:1]$, and it has the discriminant
\beq
\label{eqn:Delta_alt_E7}
\Delta_{\mathcal{X}'} = v^9 D(u, v)^2 \Big(v \, C(u, v)^2  - 4 \, D(u, v)  \Big) \,.
\eeq
As explained before, on the Jacobian elliptic fibration~(\ref{eqn:alt_E7}) the fiberwise translation by the 2-torsion section acts as a van Geemen-Sarti involution which we will denote by $\jmath_{\mathcal{X}'}$. The minimal resolution of the quotient surface $\mathcal{X}'/\langle \jmath_{\mathcal{X}'} \rangle$ is a K3 surface $\mathcal{Y}'$ admitting an induced Jacobian elliptic fibration $\pi_{\mathcal{Y}'}\colon  \mathcal{Y}' \rightarrow \mathbb{P}^1$. After rescaling, the induced fibration becomes
\beq
\label{eqn:alt_SI_E7}
  \mathcal{Y}'\colon \quad y^2 z  = x^3 -2 v^2 C(u, v) \, x^2 z + v^3  \Big(v \, C(u, v)^2  - 4 \, D(u, v) \Big) \, x z^2  \,,
\eeq
and it has the discriminant
\beq
\Delta_{\mathcal{Y}'} =   16  v^9 D(u, v) \Big(v \, C(u, v)^2  - 4 \, D(u, v)  \Big)^2 \,.
\eeq
Thus, the surfaces $\mathcal{X}'$ and $\mathcal{Y}'$ are both Jacobian elliptic K3 surfaces with a Mordell-Weil group $\mathbb{Z}/2\mathbb{Z}$ and singular fibers $III^* + 5 I_2 + 5 I_1$. We make the following:
\begin{remark}
\label{rem:gauge_fixing1}
By rescaling $(x, y, z) \to (\Lambda^2 x, \Lambda^3 y, z)$ and changing $u \mapsto a u+b v$, we can assume that $D(t,1)$ is a monic polynomial of degree five, whose sub-leading coefficient proportional to $t^4$ vanishes. 
\end{remark}
In the following, we will assume that the polynomials $C$ and $D$ are as follows:
\beq
\label{eqn:CD}
 C(u, v) =  c_2 u^2 + c_1 uv + c_0 v^2\,, \quad D(u, v) = u^5 + d_3 u^3v^2 + d_2 u^2v^3 + d_1 uv^4 + d_0 v^5\,.
\eeq
We have the following:
\begin{corollary}
\label{prop:fibrations_SI_CD}
General K3 surfaces $\mathcal{X}'$ and $\mathcal{Y}'$ have the N\'eron-Severi lattice isomorphic to $P'_{14}$.
\end{corollary}
\begin{proof}
The proof follows directly from the basic lattice theoretical facts in the proof of Proposition~\ref{cor:EFS3}.
\end{proof}
\par We introduce parameters $\{ \mathcal{J}_{2k} \}$, given by
\beq
\label{modulispace_lattice2}
  \Big( \mathcal{J}_2, \mathcal{J}_6, \mathcal{J}_8, \mathcal{J}_{10}, \mathcal{J}_{12}, \mathcal{J}_{16}, \mathcal{J}_{20}  \Big) =  \Big(c_2, c_1, d_3, c_0,  d_2, d_1, d_0 \Big) \,,
\eeq
whose subscripts will reflect their weights under the scaling. We have the following:

\begin{theorem}
\label{thm:moduli_space_seldual}
$\mathscr{M}_{P'}$ in Equation~(\ref{eqn:moduli_space_PP}) is a coarse moduli space of $P'_{14}$-polarized K3 surfaces. Here, a K3 surface $\mathcal{X}' \in \mathscr{M}_{P'}$ is the minimal resolution of Equation~(\ref{eqn:alt_SI_E7}).
\end{theorem}
\begin{proof}
Because of Proposition~\ref{cor:EFS3}, every $P'_{14}$-polarized K3 surface, up to isomorphism, admits a unique alternate fibration that can be brought into the form of Equation~(\ref{eqn:alt_E7}). One can then tell precisely when two members of the family in Equation~(\ref{eqn:alt_E7}) are isomorphic. The normalization of the coefficients in Equation~(\ref{eqn:CD}) fixes the coordinates $[u:v] \in \mathbb{P}^1$ completely; see Remark~\ref{rem:gauge_fixing1}.  Thus, two members  are isomorphic if and only if their coefficient sets are related by the transformation 
\beq
 \Big(c_2, c_1, d_3, c_0,  d_2, d_1, d_0 \Big) \  \mapsto \  \Big(  \Lambda^2 c_2, \Lambda^6 c_1, \Lambda^8 d_3,  \Lambda^{10}c_0, \Lambda^{12} d_2, \Lambda^{16} d_1, \Lambda^{20} d_0 \Big) ,
\eeq
with $\Lambda \in \mathbb{C}^\times$. The reason is that such a rescaling, when combined with the transformation  $(u, v, x, y, z) \mapsto (\Lambda^4 u, v, \Lambda^{10} x, \Lambda^{15} y, z)$, gives rise to a holomorphic isomorphism of Equation~(\ref{eqn:alt_E7}).  Conversely, an equivalence class of invariants in Equation~(\ref{modulispace_lattice2}) determines a well defined K3 surface as long as the Weierstrass model is irreducible and minimal.
\par Bringing Equation~(\ref{eqn:alt_E7}) into a standard Weierstrass normal form, we obtain
\beq
\label{eqn:WE_CD}
 y^2 z = x^3 - \frac{1}{3} v^3 \big( v C(u, v)^2 - 3  D(u, v) \big) \, x z^2 + \frac{1}{27} v^5 C(u, v)  \big(2 v C(u, v)^2 - 9 D(u, v) \big) z^3\,.
\eeq
Because the polynomial $D(t,1)$ is monic, we cannot have $D \equiv 0$ or $v C(u, v)^2 - 4 D(u, v) \equiv 0$. Thus, in Equation~(\ref{eqn:WE_CD}) the right hand side cannot factor into a product of two terms where one is a non-trivial square. However,  the Weierstrass model becomes non-minimal if and only if  there are $r, b_4 \in \mathbb{C}$ such that
\beq
 \Big(c_2, c_1, d_3, c_0,  d_2, d_1, d_0 \Big) \  \mapsto \  \Big(  s^2,  2 r s^2, 10 r^2, r^2s^2,  -20 r^3,  -15r^4, -4 r^5\Big) .
\eeq
Then for the polynomial $c = u + r v \in \mathbb{C}[u, v]$ the polynomials $c^2$ divides $C$ and $c^4$ divides $D$. 
\end{proof}
\begin{remark}
\label{rem:PP}
The proof uses the fact that every $P'_{14}$-polarized K3 surface, up to isomorphism, admits a unique alternate fibration~(\ref{eqn:alt_E7}). Thus, there is a canonical lattice embedding  $H \oplus N \hookrightarrow P'_{14}$, and every $P'_{14}$-polarized K3 surface carries a unique underlying $H \oplus N$-polarization. 
\end{remark}
In contrast to the $P_{14}$-polarized case, the sub-variety $\mathscr{M}_{P'}$ is left invariant by action of the van Geemen-Sarti-Nikulin duality. The dual of a given $P'_{14}$-polarized K3 surface is again a $P'_{14}$-polarized surface; see Corollary~\ref{prop:fibrations_SI_CD}. This involution, denoted by
\beqn
 \imath_{\mathrm{vgsn}}' \colon \quad  \mathscr{M}_{P'}  \ \rightarrow \  \mathscr{M}_{P'} \,, \qquad \text{with} \quad  \imath_{\mathrm{vgsn}} ' \circ \imath_{\mathrm{vgsn}} ' = \mathrm{id} \,,
\eeqn 
can be constructed explicitly. We have the following:
\begin{proposition}
\label{prop:involutionCD}
The van Geemen-Sarti-Nikulin duality acts on the moduli space $\mathscr{M}_{P'}$ in Equation~(\ref{eqn:moduli_space_PP}) as the involution $\imath_{\mathrm{vgsn}}'\colon \mathscr{M}_{P'}  \to \mathscr{M}_{P'}$ given by Equation~(\ref{eqn:involutionCD22}).
\end{proposition}
\begin{proof}
After rescaling Equation~(\ref{eqn:alt_SI_E7}), the induced fibration on $\mathcal{Y}'$  can be written as
\beq
\label{eqn:alt_SI_E7_2a}
  \mathcal{Y}'\colon \quad \tilde{y}^2 \tilde{z}  = \tilde{x}^3 - v^2 C(u, v) \, \tilde{x}^2 \tilde{z} + v^3  \Big( - D(u, v) +  \frac{v}{4} \, C(u, v)^2 \Big) \, \tilde{x} \tilde{z}^2  \,.
\eeq
If we also set $[u : v]=[-\tilde{u}+c_2^2 \tilde{v}/20 : \tilde{v}]$, then Equation~(\ref{eqn:alt_SI_E7_2a}) becomes
\beq
\label{eqn:alt_SI_E7_b}
  \mathcal{Y}'\colon \quad  \tilde{y}^2 \tilde{z}  = \tilde{x}^3+ \tilde{v}^2 \tilde{C}(\tilde{u}, \tilde{v}) \, \tilde{x}^2 \tilde{z} + \tilde{v}^3  \tilde{D}(\tilde{u}, \tilde{v}) \, \tilde{x} \tilde{z}^2  \,,
\eeq
where $\tilde{C}(\tilde{u}, \tilde{v}) = \tilde{c}_2 \tilde{u}^2 + \tilde{c}_1 \tilde{u}\tilde{v} + \tilde{c}_0 \tilde{v}^2$ and $\tilde{D}(\tilde{u}, \tilde{v}) = \tilde{u}^5 + \tilde{d}_3 \tilde{u}^3 \tilde{v}^2 + \tilde{d}_2 \tilde{u}^2 \tilde{v}^3 + \tilde{d}_1 \tilde{u}\tilde{v}^4 + \tilde{d}_0 \tilde{v}^5$ are related to the polynomials in Equation~(\ref{eqn:CD}) by the equations
\beq
\label{eqn:CD2}
  \tilde{C}\Big(\tilde{u}, \tilde{v}\Big)  = -C\left(-\tilde{u}+\frac{c_2^2}{20} \tilde{v} , \tilde{v} \right)\,,\quad
  \tilde{D}\Big(\tilde{u}, \tilde{v}\Big)  = -D\left(-\tilde{u}+\frac{c_2^2}{20}  \tilde{v} , \tilde{v} \right) +  \frac{\tilde{v}}{4}  C\left(-\tilde{u}+\frac{c_2^2}{20} \tilde{v} , \tilde{v} \right)\,.
\eeq
The van Geemen-Sarti-Nikulin duality maps $\mathcal{X}'$ to  $\mathcal{Y}'$ and vice versa. Hence, the duality acts by interchanging $(C,D)$ and $(\tilde{C}, \tilde{D})$ or, equivalently, by the action of an involution $\imath_{\mathrm{vgsn}} '$ on the defining parameter sets of the K3 surfaces $\mathcal{X}'$ and  $\mathcal{Y}'$, i.e., 
\beq 
\label{eqn:involutionCD}
 \imath_{\mathrm{vgsn}}'\colon \quad \Big(c_2, c_1, c_0, d_3, d_2, d_1 , d_0\Big) \ \mapsto \  \Big(\tilde{c}_2, \tilde{c}_1, \tilde{c}_0, \tilde{d}_3, \tilde{d}_2, \tilde{d}_1, \tilde{d}_0\Big) \,,
\eeq
with
\beq
\label{eqn:involutionCD2}
 \left( \begin{array}{c} \tilde{c}_2 \\[2pt] \tilde{c}_1 \\[2pt] \tilde{c}_0 \\[2pt] \tilde{d}_3 \\[2pt] \tilde{d}_2 \\[2pt] \tilde{d}_1 \\[2pt] \tilde{d}_0  \\[6pt] \end{array}\right) 
 =  \left( {\scriptscriptstyle\begin{array}{l} -c_2 \\[2pt]  \phantom{-} c_1 + \frac{1}{10} c_2^3 \\[2pt]  -c_0 - \frac{1}{20} c_1 c_2^2 - \frac{1}{400} c_2^5 \\[2pt] 
 \phantom{-} d_3 - \frac{1}{2} c_1 c_2 - \frac{1}{40} c_2^4 \\[2pt] 
 - d_2 - \frac{3}{20} c_2^2 d_3 + \frac{1}{4} c_1^2 + \frac{1}{2} c_0 c_2 + \frac{3}{40} c_1 c_2^3 + \frac{1}{400} c_2^6\\[2pt] 
 \phantom{-}d_1 + \frac{1}{10} c_2^2 d_2 +\frac{3}{400} c_2^4 d_3 - \frac{1}{2} c_0 c_1 - \frac{1}{20} c_0 c_2^3 - \frac{1}{40} c_1^2 c_2^2 - \frac{3}{800} c_2^5 c_1 - \frac{3}{3200} c_2^8\\[2pt]
 - d_0  -\frac{1}{20} c_2^2 d_1 - \frac{1}{400} c_2^4 d_2 - \frac{1}{8000} c_2^6 d_3  \\[2pt]
  \multicolumn{1}{r}{+ \frac{1}{4} c_0^2+ \frac{1}{40} c_0 c_1 c_2^2+ \frac{1}{1600} c_1^2 c_2^4 + \frac{1}{800} c_0 c_2^5 + \frac{1}{16000} c_1 c_2^7 + \frac{1}{800000} c_2^{10}}\end{array}}\right) 
\eeq
such that $(\imath_{\mathrm{vgsn}} ')^2=\mathrm{id}$. The latter is checked by a straightforward computation. The involution can then be written in terms of the variables of Equation~(\ref{modulispace_lattice2}).
\end{proof}
We have the following:
\begin{corollary}
\label{cor:selfdual_locus}
The self-dual locus in $\mathscr{M}_{P'}$ is given by
\beqn
  \left\lbrace \  \Big[  \mathcal{J}_2 :  \mathcal{J}_6 : \mathcal{J}_8:  \mathcal{J}_{10} :  \mathcal{J}_{12}:  \mathcal{J}_{16}:  \mathcal{J}_{20}
   \Big]  \in \mathscr{M}_{P'} \ \Big\vert  \  \Big(\mathcal{J}_2, \ \mathcal{J}_{10}, \ \mathcal{J}_6^2 - 8 \mathcal{J}_{12}, \  \mathcal{J}_{20} \Big)  = 0 \,  \right\rbrace \,.
 \eeqn  
A general element of the self-dual locus is a Jacobian elliptic K3 surface with singular fibers $III^* + III + 4 I_2 + 4 I_1$ and  Mordell-Weil group $\mathbb{Z}/2\mathbb{Z}$.
\end{corollary}
\subsubsection{The case of \texorpdfstring{$H \oplus D_8(-1) \oplus E_8(-1)$}{H + D8 + E8}-polarized K3 surfaces}
One can ask whether there are any other cases of Jacobian elliptic K3 surfaces which are self-dual with respect to the van Geemen-Sarti-Nikulin duality.  To that end, let us consider a Jacobian elliptic fibration on a K3 surface with singular fibers $k III^* + n I_2 + n I_1$ with $k,n \in \mathbb{N}$ and  $9k + 3n =24$. There are three cases: $(k,n)=(0,8)$ is the original case of Picard number 10 examined by van Geemen and Sarti \cite{MR2274533}; the case $(k,n)=(1,5)$ gives rise to the $P'_{14}$-polarized K3 surfaces. Finally, there is the case $(k,n)=(2,2)$ which we include here for completeness. A Weierstrass model for a Jacobian elliptic fibration $\pi_{\mathcal{X}'}\colon  \mathcal{X}' \rightarrow \mathbb{P}^1$ in the case $(k,n)=(2,2)$ is given by
\beq
\label{eqn:alt_2E7}
  \mathcal{X}'\colon \quad y^2 z  = x^3 +  c_0 u^2v^2 x^2 z + u^3v^3  D(u, v) x z^2  \,,
\eeq
where $D$ is a homogeneous polynomial of degree two and $c_0 \in \mathbb{C}^\times$.   If the Weierstrass model is minimal, the polynomial $D(t,1)$ has a non-vanishing quadratic coefficient. The discriminant is
\beq
\label{eqn:Delta_alt_2E7}
\Delta_{\mathcal{X}'} = u^9 v^9  \, D(u, v)^2  \Big( c_0^2 u v  - 4 D(u, v)  \Big) \,.
\eeq
The van Geemen-Sarti-Nikulin duality yields a K3 surface $\mathcal{Y}'$ with an induced Jacobian elliptic fibration $\pi_{\mathcal{Y}'}\colon  \mathcal{Y}' \rightarrow \mathbb{P}^1$ given by
\beq
\label{eqn:alt_SI_2E7}
  \mathcal{Y}'\colon \quad y^2 z  = x^3 -2 c_0 u^2 v^2 x^2 z + u^3 v^3  \Big( c_0^2 uv  - 4 D(u, v) \Big) x z^2  \,.
\eeq
It has the discriminant
\beq
\Delta_{\mathcal{Y}'} =  16 u^9 v^9 \, D(u, v) \,  \Big( c_0^2 u v  - 4 D(u, v)  \Big) ^2 \,.
\eeq
We have the following:
\begin{lemma}
\label{lem:alternate_fibrations_2E7}
General K3 surfaces $\mathcal{X}'$ and $\mathcal{Y}'$ admit Jacobian elliptic fibrations with singular fibers $2III^* + 2 I_2 + 2 I_1$ and Mordell-Weil group $\mathbb{Z}/2\mathbb{Z}$.
\end{lemma}
\begin{proof}
The statements are checked directly using Equation~(\ref{eqn:alt_2E7}) and~(\ref{eqn:alt_SI_2E7}). 
\end{proof}
We make the following:
\begin{remark}
\label{rem:gauge_fixing2}
By rescaling we can assume that $D(t,1)$ is a monic polynomial of degree two, and we set
\beq
\label{eqn:EF}
 D(u, v) =  u^2 + d_1 uv + d_0 v^2\,.
\eeq
Since we already moved the singular fibers of type $III^*$ to $u=0$ and $v=0$, respectively,  we have fixed the coordinates $[u:v] \in \mathbb{P}^1$ completely.
\end{remark}
This implies the following:
\begin{theorem}
\label{thm:moduli_space_seldual_18}
The 2-dimensional open analytic space given by
\beq
\label{modulispace_lattice3}
  \Big \{ \ 
 \Big[ c_0 : d_1 : d_0 \Big]   \in \mathbb{WP}_{(2,4,8)}  \ \Big| \
d_0 \neq 0  \ \Big \}  \,,
\eeq
is a coarse moduli space of $H \oplus D_8(-1) \oplus E_8(-1)$-polarized K3 surfaces. The van Geemen-Sarti-Nikulin duality acts on the moduli space by $( c_0, d_1, d_0)  \mapsto ( -c_0, d_1 + c_0^2/4,  d_0)$.  A general element of the self-dual locus is given by $c_0=0$ and a Jacobian elliptic surface with singular fibers $2 III^* + 2 III$ and Mordell-Weil group $\mathbb{Z}/2\mathbb{Z}$.
\end{theorem}
\begin{proof}
The proof is analogous to the proof of Theorem~\ref{thm:moduli_space_seldual}. 
\end{proof}
\section{Projective models for certain K3 surfaces}
\label{sec:projective_models}
In this section we construct birational projective models for the K3 surfaces with N\'eron-Severi lattices $P_{14}$, $P'_{14}$, and $P''_{14}$ and determine all inequivalent Jacobian elliptic fibrations and explicit Weierstrass models on a general member in each case.
\subsection{Projective model for \texorpdfstring{$P_{14}$}{P}-polarized K3 surfaces}
\label{sec:geometry}
In \cites{MR2369941,MR2279280} it was proved that  a complex algebraic K3 surface $\mathcal{X}$ with Picard lattice $H \oplus E_8(-1) \oplus E_8(-1)$ admits a birational model isomorphic to the quartic surface in $\mathbb{P}^3=\mathbb{P}(\mathbf{X}, \mathbf{Y}, \mathbf{Z}, \mathbf{W})$ with equation
\beqn
 0= \  \mathbf{Y}^2 \mathbf{Z} \mathbf{W}-4 \mathbf{X}^3 \mathbf{Z}+3 \alpha \mathbf{X} \mathbf{Z} \mathbf{W}^2+ \beta \mathbf{Z} \mathbf{W}^3 
-   \frac{1}{2} \big( \mathbf{Z}^2  \mathbf{W}^2 +  \mathbf{W}^4 \big) .
\eeqn
The 2-parameter family was first introduced by Inose in \cite{MR578868} and is called \emph{Inose quartic}. Other examples of equations relating the elliptic fibrations of K3 surfaces with 2-elementary N\'eron-Severi lattice and quartic hypersurfaces were provided in \cites{MR4130832, MR3882710}. We will consider a multi-parameter generalization of the Inose quartic.
\par Let the projective surface $\mathcal{Q}(\alpha, \beta, \gamma, \delta , \varepsilon, \zeta, \eta, \iota, \kappa, \lambda)$ in $\mathbb{P}^3=\mathbb{P}(\mathbf{X}, \mathbf{Y}, \mathbf{Z}, \mathbf{W})$ be defined for a coefficient set $(\alpha, \beta, \gamma, \delta , \varepsilon, \zeta, \eta, \iota, \kappa, \lambda) \in \mathbb{C}^{10}$ by Equation~(\ref{quartic1}). We denote by $\mathcal{X}(\alpha, \beta, \gamma, \delta , \varepsilon, \zeta, \eta, \iota, \kappa, \lambda)$ the smooth complex surface obtained as the minimal resolution of  $\mathcal{Q}(\alpha, \beta, \gamma, \delta , \varepsilon, \zeta, \eta, \iota, \kappa, \lambda)$.  If there is no danger of confusion, we will simply write $\mathcal{X}$ and $\mathcal{Q}$. One easily checks that the quartic surface $\mathcal{Q}$ has two special singularities at the following points:
\beq
  \mathrm{P}_1 = [0: 1: 0: 0] \,, \qquad  \mathrm{P}_2 = [0: 0: 1: 0] \,.
\eeq  
For a general tuple $(\alpha, \beta, \gamma, \delta , \varepsilon, \zeta, \eta, \iota, \kappa,\lambda)$,  the points $\mathrm{P}_1$ and $\mathrm{P}_2$ are the only singularities of Equation~(\ref{quartic1}) and are rational double points. One easily verifies that in this case the singularity at $\mathrm{P}_1$ is a rational double point of type $A_7$, and  $\mathrm{P}_2$ is of type $A_3$.  In the following, we will assume that the parameters of Equation~(\ref{quartic1}) satisfy
\beq
\label{eqn:general_params}
\begin{gathered}
 (\gamma, \delta), (\varepsilon, \zeta), (\eta, \iota), (\kappa, \lambda)\neq (0,0) \,,\\
 \text{and} \quad \not \exists r \in \mathbb{C}: \  (\alpha, \beta) = (r^2, r^3) \  \text{and} \  [\gamma: \delta] = [\varepsilon: \zeta] = [\eta: \iota] = [\kappa: \lambda] = [1:-r] \,.
\end{gathered}
\eeq
We have the following:
\begin{lemma}
Assuming Equation~(\ref{eqn:general_params}), the surface $\mathcal{X}$ obtained as the minimal resolution of $\mathcal{Q}$ is a smooth K3 surface.
\end{lemma}
\begin{proof}
Equation~(\ref{eqn:general_params}) ensures that the singularities of  $\mathcal{Q}(\alpha, \beta, \gamma, \delta , \varepsilon, \zeta, \eta, \iota, \kappa, \lambda)$ are rational double points. This fact, in connection with the degree of Equation~(\ref{quartic1}) being four, guarantees that the minimal resolution is a K3 surface.
\end{proof}
\par We have the following symmetries:
\begin{lemma}
\label{symmetries1}
Let $(\alpha, \beta, \gamma, \delta , \varepsilon, \zeta, \eta, \iota, \kappa, \lambda) \in \mathbb{C}^{10}$ as before. Then, one has the following isomorphisms of K3 surfaces:
\begin{enumerate}
\item  $\mathcal{X} (\alpha,\beta, \gamma, \delta, \varepsilon, \zeta, \eta,  \iota, \kappa, \lambda) \ \simeq \
\mathcal{X}(\alpha,  \beta,  \varepsilon,  \zeta, \gamma,  \delta , \eta,  \iota, \kappa, \lambda )$,
\item  $\mathcal{X}(\alpha,\beta, \gamma, \delta, \varepsilon, \zeta, \eta,  \iota, \kappa, \lambda) \ \simeq \
\mathcal{X}(\alpha,\beta,  \eta,  \iota,  \varepsilon, \zeta, \gamma, \delta, \kappa, \lambda)$,
\item $\mathcal{X}(\alpha,\beta, \gamma, \delta, \varepsilon, \zeta, \eta,  \iota, \kappa, \lambda) \ \simeq \
\mathcal{X}(\alpha,\beta, \gamma, \delta, \kappa, \lambda, \eta,  \iota, \varepsilon, \zeta)  $,
\item $\mathcal{X}(\alpha,\beta, \gamma, \delta, \varepsilon, \zeta, \eta,  \iota, \kappa, \lambda)\ \simeq \
\mathcal{X}(\Lambda^4 \alpha,  \Lambda^6 \beta,  \Lambda^{10} \gamma,  \Lambda^{12} \delta,  \Lambda^{-2} \varepsilon,  \zeta , \Lambda^{-2} \eta, \iota, \Lambda^{-2} \kappa, \lambda ) $, \\
for $\Lambda \in \mathbb{C}^\times$.
\end{enumerate}
\end{lemma}
\begin{proof}
The birational involution $\mathbb{P}^3 \dashrightarrow   \mathbb{P}^3$ given by
\beqn
\begin{split}
  [\mathbf{X}: \mathbf{Y}: \mathbf{Z}: \mathbf{W}] \ \mapsto & \ [   \mathbf{X}\mathbf{Z} \big(2 \eta \mathbf{X} -\iota \mathbf{W} \big): 
    \mathbf{Y}\mathbf{Z} \big(2\eta \mathbf{X} - \iota \mathbf{W} \big): \\
  & \ \ \mathbf{W}^2  \big(2\kappa \mathbf{X} - \lambda \mathbf{W} \big):  \mathbf{Z}\mathbf{W}  \big(2 \eta \mathbf{X} - \iota \mathbf{W}\big)   ] \,,
\end{split}
\eeqn
extends to an isomorphism between the two K3 surfaces from statement (a). Parts (b) and (c) are obvious from Equation~(\ref{quartic1}). For $\Lambda \in \mathbb{C}^\times$ the projective automorphism, given by
\beqn
\label{tmor}
\mathbb{P}^3  \ \to \  \mathbb{P}^3, \ \ \ [\mathbf{X}: \mathbf{Y}: \mathbf{Z}: \mathbf{W}] \ \mapsto \ [\ \Lambda^8\mathbf{X}: \ \Lambda^9\mathbf{Y}: \mathbf{Z}: \ \Lambda^6\mathbf{W} \ ] \,,
\eeqn
extends to an isomorphism realizing part (d).
\end{proof}
\par We also have the following:
\begin{proposition}
\label{NikulinInvolution}
Let $(\alpha, \beta, \gamma, \delta, \varepsilon, \zeta, \eta, \iota, \kappa, \lambda) \in \mathbb{C}^{10}$ as before. The van Geemen-Sarti involution on $\mathcal{X}$ is induced by the projective automorphism
\beq
\label{eq:NikulinInvolutions}
\begin{split}
 \Psi\colon\quad  \mathbb{P}^3   \ \to \ & \  \mathbb{P}^3, \\
  [\mathbf{X}: \mathbf{Y}: \mathbf{Z}: \mathbf{W}]  \mapsto & \ [   (2 \gamma \mathbf{X} - \delta \mathbf{W}) (2\eta \mathbf{X} - \iota\mathbf{W}) \mathbf{X}\mathbf{Z} 
  :   - (2 \gamma \mathbf{X} - \delta \mathbf{W}) (2\eta \mathbf{X} - \iota \mathbf{W} )  \mathbf{Y}\mathbf{Z}   : \\
  & \ \ (2 \varepsilon \mathbf{X} - \zeta \mathbf{W}) (2 \kappa \mathbf{X} - \lambda \mathbf{W}) \mathbf{W}^2  :  (2 \gamma \mathbf{X} - \delta \mathbf{W})  (2\eta \mathbf{X}- \iota\mathbf{W} )  \mathbf{W} \mathbf{Z}  ] \,.
\end{split}
\eeq
\end{proposition}
\begin{proof}
One checks that $\Psi$ constitutes an involution of the projective quartic surface $\mathcal{Q} \subset  \mathbb{P}^3(\mathbf{X}, \mathbf{Y}, \mathbf{Z}, \mathbf{W})$.  If we use the affine chart $\mathbf{W}=1$ then the unique holomorphic 2-form is given by $d\mathbf{X} \wedge d\mathbf{Y} / \partial_\mathbf{Z} F( \mathbf{X},\mathbf{Y}, \mathbf{Z})$ where $F( \mathbf{X},\mathbf{Y}, \mathbf{Z})$ is the left side of Equation~(\ref{quartic1}). One then checks that $\Psi$ in Equation~(\ref{eq:NikulinInvolutions}) constitutes a symplectic involution after using $F( \mathbf{X},\mathbf{Y}, \mathbf{Z})=0$. 
\end{proof}
One has lines on $\mathcal{Q}(\alpha, \beta, \gamma, \delta , \varepsilon, \zeta, \eta, \iota, \kappa, \lambda)$ in Equation~(\ref{quartic1}), denoted by $L_1 $, $ L_2 $, $L_3$, $L_4$, $L_5$:
\beq
\label{eqn:lines}
 \begin{array}{llcll}
 L_1\colon& \; \mathbf{X}=\mathbf{W}=0 \,, &\qquad L_2\colon &\; \mathbf{Z}=\mathbf{W}=0\,, \\[0.2em]
 L_3\colon& \; 2\varepsilon \mathbf{X}-\zeta \mathbf{W} = \mathbf{Z} = 0  \,, &\qquad L_4\colon & \; 2\mathbf{X} +\gamma\eta\mathbf{Z} = \mathbf{W} = 0 \,,\\[0.2em]
 L_5\colon& \; 2 \kappa \mathbf{X} - \lambda \mathbf{W}=\mathbf{Z}=0 \,.
\end{array} 
\eeq
For $\gamma \varepsilon \zeta \eta \kappa\lambda \neq 0$, the lines are distinct and concurrent, meeting at $\mathrm{P}_1$.  By $L_i(u, v)$ with $[u:v] \in \mathbb{P}^1$ we denote the pencil of hyperplanes containing the line $L_i$ for $i=1, \dots, 5$.

We have the following:
\begin{theorem}
\label{thm1}
Assuming Equation~(\ref{eqn:general_params}), the minimal resolution of the quartic in Equation~(\ref{quartic1}) is a K3 surface $\mathcal{X}$ endowed with a canonical $P_{14}$-polarization. Conversely, every $P_{14}$-polarized K3 surface has a birational projective model given by Equation~(\ref{quartic1}). In particular, the Jacobian elliptic fibrations in Lemma~\ref{prop:ADE} are attained as follows:
\beqn
\label{list:fib00}
\begin{array}{c|c|c|c|l}
\# 	&  \text{singular fibers} 	& \operatorname{MW} 		& \text{root lattice} 	& \text{pencil} \\
\hline
1	& I_4^* + 4 I_2 + 6 I_1	&  \mathbb{Z}/2\mathbb{Z}	& D_{8} + A_1^{\oplus 4} 	& \text{residual surface intersection} \\ 
	&					&						&				& \text{of $L_1(u, v)=0$ and $\mathcal{Q}$}\\
\hline
2	& 2 I_2^* + 8 I_1		&  \{ \mathbb{I} \} 			& D_6^{\oplus 2} 	& \text{residual surface intersection} \\ 
	&					&						&				& \text{of $L_2(u, v)=0$ and $\mathcal{Q}$}\\
\hline
 3	&III^* + I_0^* + I_2 + 7 I_1	&  \{ \mathbb{I} \} 			& E_7 + D_4 + A_1	& \text{residual surface intersection} \\ 
	&					&						&				& \text{of $L_i(u, v)=0\; (i=3,5)$ and $\mathcal{Q}$}\\
\hdashline
3' 	& II^* + 4 I_2 + 6 I_1		& \{ \mathbb{I} \} 			& E_8 + A_1^{\oplus 4}& \text{residual surface intersection} \\ 
	&					&						&				& \text{of $\tilde{C}_3(u, v)=0 \, (\deg =2)$  and $\mathcal{Q}$}\\
\hline
4	& I_6^* + 2 I_2 + 8 I_1	&  \{ \mathbb{I} \} 			& D_{10} + A_1^{\oplus 2} & \text{residual surface intersection} \\ 
	&					&						&				& \text{of $L_4(u, v)=0$ and $\mathcal{Q}$}\\
\end{array}
\eeqn
Fibrations in cases $(2), (3), (4)$ and $(3')$ are also induced by the intersection of the quartic surface $\mathcal{Q}$ with pencils $C_i(u, v)$ of degree $d_i$ such that $(i, d_i)=(2,3), (3,3), (4,4)$ and $C'_3(u, v)$ of degree $d'_3=3$.
\medskip

\noindent The Jacobian elliptic fibrations on a general $P_{15}$-polarized K3 surface in Lemma~\ref{prop:ADE} are attained by setting $(\kappa, \lambda)=(0,1)$ above. They are as follows:
\beqn
\label{list:fib0}
\begin{array}{c|c|c|c|l}
\# 	&  \text{singular fibers} 	& \operatorname{MW} & \text{root lattice} & \text{pencil} \\
\hline
1	& I_6^* + 3 I_2 + 6 I_1	&  \mathbb{Z}/2\mathbb{Z}	& D_{10} + A_1^{\oplus 3} & \text{residual surface intersection} \\ 
	&					&						&					& \text{of $L_1(u, v)=0$ and  $\mathcal{Q}$}\\
\hline
2	& III^* + I_2^* + 7 I_1	&  \{ \mathbb{I} \} 	& E_7 + D_6 	& \text{residual surface intersection} \\ 
	&					&				&			& \text{of $L_2(u, v)=0$ and  $\mathcal{Q}$}\\
\hline
3	& II^* + I_0^* + I_2 + 6 I_1	&  \{ \mathbb{I} \} 	& E_8 + D_4 + A_1	& \text{residual surface intersection} \\ 
	&					&				&				& \text{of $L_3(u, v)=0$ and  $\mathcal{Q}$}\\
\hline
4	& I_8^* + I_2 + 8 I_1		&  \{ \mathbb{I} \} 	& D_{12} + A_1		& \text{residual surface intersection} \\ 
	&					&				&				& \text{of $L_4(u, v)=0$ and  $\mathcal{Q}$}
\end{array}
\eeqn
\end{theorem}
The situation for Picard number 16 was already discussed in \cite{CHM19}.
\begin{remark}
The fibrations in Theorem~\ref{thm1} are labeled $(1), (2), (3), (3'), (4)$ to make the notation consistent with the one that appeared for higher Picard ranks in \cites{CHM19,MR4160930}.
\end{remark}
\begin{proof}
We will construct explicit Weierstrass models for the fibrations $(1)$-$(4)$ in Sections~\ref{sssec:WEalt}-\ref{sssec:WEmax}. Using fibration~$(3')$ it follows immediately that a K3 surface $\mathcal{X}$ is endowed with a canonical  $P_{14}$-polarization.  The given substitution for fibration~$(1)$ leads to a Weierstrass model in the form of Equation~(\ref{eqn:alt0}) if we set
\beq
\label{eqn:pf_comparison}
\begin{split}
 A(t) & = t^3 + a_1 t + a_0 \ = \  t^3 - 3 \alpha t - 2 \beta \,, \\
 B(t) & = b_4 t^4 + b_3 t^3 + b_2 t^2 + b_1 t + b_0 \ = \  \big(\gamma t - \delta\big) \big(\varepsilon t - \zeta\big)  \big(\eta t  - \iota\big)  \big(\kappa t -\lambda\big) \,.
\end{split}
\eeq 
Equation~(\ref{eqn:general_params}) ensures that the singularities of  $\mathcal{Q}$ are rational double points.  For fibration~$(1)$ the given conditions are equivalent to corresponding Weierstrass model being irreducible and minimal; see proof of Theorem~\ref{thm:moduli_space_alt1}.
\par Conversely, Proposition~\ref{cor:EFS2}  proves that every general $P_{14}$-polarized K3 surface admits a unique alternate fibration. It follows from Equations~(\ref{eqn:substitution_alt}) that from an alternate fibration a quartic can be constructed if we write the polynomials $A$ and $B$ according to Equation~(\ref{eqn:pf_comparison}).  Thus, every $P_{14}$-polarized K3 surface, up to isomorphism, is in fact realized as the resolution of the quartic in Equation~(\ref{quartic1}). We normalized the elliptic fibrations so that for $(\kappa,\lambda)=(0,1)$ they remain well defined and specialize to the corresponding elliptic fibrations in Picard rank 15 except for fibration~$(3')$. 
\par We now complete the proof by constructing Weierstrass models for the Jacobian elliptic fibrations and the associated pencils on the quartic normal form explicitly:
\subsubsection{Fibration~\texorpdfstring{$(1)$}{(1)}}
\label{sssec:WEalt}
The alternate fibration is induced by intersecting the quartic surface $\mathcal{Q}$ with the pencil
\beq
\label{eqn:pencil_alternate}
L_1(u, v)=u \mathbf{W} - 2 v \mathbf{X} =0
\eeq
for $[u:v] \in \mathbb{P}^1$. Making the substitutions 
\beq
\label{eqn:substitution_alt}
\mathbf{X}= u v x\,, \quad  \mathbf{Y}=  \sqrt{2} y \,, \quad \mathbf{Z} = 2 v^4 (\varepsilon u-\zeta v) (\kappa u- \lambda v) z\,, \quad \mathbf{W} = 2 v^2 x\,,
\eeq
into Equation~(\ref{quartic1}), compatible with  $L_1(u, v)=0$, determines the Jacobian elliptic fibration  $\pi\colon  \mathcal{X} \rightarrow \mathbb{P}^1$ with fiber $\mathcal{X}_{[u:v]}$ given by 
\beq
\label{eqn:alt}
\mathcal{X}_{[u:v]}\colon \quad y^2 z  = x \Big(x^2 + v \, A(u, v) \, x z +v^4 B(u, v) \, z^2 \Big) \,.
\eeq
The fibrations admits the section $\sigma: [x:y:z] = [0:1:0]$ and the 2-torsion section  $[x:y:z] = [0:0:1]$. Here, the discriminant is
\beq
\label{eqn:discr_alt}
\Delta(u, v) =  v^{10} B(u, v)^2 \, \Big(A(u, v)^2-4 v^2 B(u, v)\Big) \,,
\eeq
and
\beq
\label{eqn:polynomials0}
A(u, v) = u^3-3\alpha  uv^2- 2\beta v^3 ,\quad
B(u, v) = (\gamma u-\delta v)(\varepsilon u -\zeta v) (\eta u -\iota v)  (\kappa u - \lambda v)  \,.
\eeq
\subsubsection{Fibration~\texorpdfstring{$(2)$}{(2)}}
An elliptic fibration with section, called the \emph{standard fibration}, is induced by intersecting the quartic surface $\mathcal{Q}$ with the pencil
\beq
\label{eqn:pencil_standard1}
L_2(u, v)=u \mathbf{W} - v \mathbf{Z} =0
\eeq
for $[u:v] \in \mathbb{P}^1$. Making the substitutions  
\beq
\label{eqn:substitution_std}
 \mathbf{X} = u v x\,, \quad  \mathbf{Y}=  \sqrt{2} y\,, \quad \mathbf{Z} = 2 u^4 v^2 z\,, \quad \mathbf{W} = 2 u^3 v^3 z \,,
\eeq
in Equation~(\ref{quartic1}), compatible with $L_2(u, v)=0$, yields the Jacobian elliptic fibration $\pi\colon \mathcal{X} \rightarrow \mathbb{P}^1$ with fiber $\mathcal{X}_{[u:v]}$ given by 
\beq
\label{eqn:std}
\mathcal{X}_{[u:v]}\colon \quad y^2 z  = x^3 +   e(u, v) \, x^2 z + f(u, v) \, x z^2 + g(u, v) \, z^3 \,.
\eeq
The fibrations admits the section $\sigma: [x:y:z] = [0:1:0]$. Here, the discriminant is
\beq
\Delta(u, v) =   f^2 \big( e^2 - 4 f\big)  - 2 eg \big( 2 e^2 -9 f \big) - 27 g^2 =  u^8 v^8  p(u, v) \,,
\eeq
and
\beq
\begin{split}
e(u, v) & =  uv \big( \gamma \eta u^2 + \varepsilon \kappa v^2 \big) \,,\\[0.2em]
f(u, v)  & = - u^3 v^3  \Big((\gamma\iota+\delta\eta)  u^2 + 3  \alpha  uv + (\varepsilon\lambda + \zeta\kappa) v^2\Big) \,, \\
g(u, v)  & =  u^5 v^5 \Big(\delta\iota  u^2 - 2  \beta  u v  + \zeta\lambda  v^2\Big) \,,
\end{split}
\eeq
and $p(u, v) =  \gamma^2 \eta^2 (\gamma\iota- \delta\eta)^2 u^8 + \dots +  \varepsilon^2 \kappa^2 (\varepsilon\lambda- \zeta\kappa)^2 v^8$ is a homogeneous polynomial of degree eight. 
\par When applying the Nikulin involution in Proposition~\ref{NikulinInvolution} to the  pencil $L_2(u, v)$, we obtain a pencil, denoted by $C_2(u, v)=0$ with $[u:v] \in \mathbb{P}^1$. A computation yields
\beq
\label{eqn:pencil_standard2}
 C_2(u, v)= v \mathbf{W} \big( 2 \varepsilon \mathbf{X}- \zeta \mathbf{W} \big)   \big( 2 \kappa \mathbf{X}- \lambda\mathbf{W} \big)- u \mathbf{Z} \big( 2 \gamma \mathbf{X} - \delta \mathbf{W}\big) \big(2\eta \mathbf{X}-\iota\mathbf{W}\big) =0 \,,
\eeq
such that the fibration is also obtained by intersecting the quartic $\mathcal{Q}$ with the pencil $C_2(u, v)=0$. 
\subsubsection{Fibration~\texorpdfstring{$(3)$}{(3)}}
\label{sssec:bfd}
By blowing up the base according to $x = v^4 x'$, Equation~(\ref{eqn:alt}) becomes a double cover of the Hirzebruch surface $\mathbb{F}_0=\mathbb{P}(u, v) \times\mathbb{P}(x', z)$ branched along a curve of bi-degree $(4,4)$. Every such cover has two natural elliptic fibrations corresponding to the two rulings of the quadric $\mathbb{F}_0$ coming from the two projections $\pi_i: \mathbb{F}_0 \to \mathbb{P}^1$ for $i=1,2$. The fibration $\pi_1$ is isomorphic to the alternate fibration. The second elliptic fibration arises from the projection $\pi_2$ and is called the \emph{base-fiber dual fibration} -- a label that has appeared in the physics literature. This second elliptic fibration with section is induced by intersecting the quartic surface $\mathcal{Q}$ with the pencil
\beq
\label{eqn:pencil_bfd1}
 L_3(u, v)=u \mathbf{Z} - v (2 \varepsilon \mathbf{X} - \zeta \mathbf{W})=0
\eeq
for $[u:v] \in \mathbb{P}^1$. Making the substitutions 
\beq
\label{eqn:substitution_bfd}
\begin{array}{lll}
\mathbf{X}=  u v x \,, &&  \mathbf{Y} = \sqrt{2}y \,, \\ [0.4em]
\mathbf{Z}= 2 \big(\varepsilon x+ \zeta (u +\gamma\varepsilon\eta v) u v^2 z\big) v^2\,, && \mathbf{W} = 2 \big(u +\gamma\varepsilon\eta v\big) u^2 v^3 z \,,
\end{array}
\eeq
into Equation~(\ref{quartic1}), compatible with $L_3(u, v)=0$, determines a Jacobian elliptic fibration  $\pi\colon \mathcal{X} \to \mathbb{P}^1$ with fiber $\mathcal{X}_{[u:v]}$ given by
\beq
\label{eqn:bfd}
\mathcal{X}_{[u:v]}\colon \quad y^2 z = x^3 + e(u, v) \, x^2 z + f(u, v) \, x z^2 + g(u, v) \, z^3 \,.
\eeq
The fibration admits the section $\sigma: [x:y:z] = [0:1:0]$. Here, the discriminant is
\beq
\Delta(u, v) =  u^6 v^9  \big(u + \gamma\varepsilon \eta v\big)^2 p(u, v) \,,
\eeq
and
\beqn
\begin{split}
&e(u, v)  = - u v^3 (\gamma \varepsilon \iota + \gamma\zeta \eta + \delta \varepsilon \eta) \,,\\
&f(u, v) = u^2 v^3  \big(u +\gamma\varepsilon\eta v\big)  \big( \kappa u^2 - 3 \alpha u v + (  \gamma\zeta\iota  + \delta \varepsilon\iota + \delta \zeta \eta) v^2 \big) \,,\\
&g(u, v)  = -u^3 v^5   \big(u +\gamma\varepsilon\eta v\big)^2 \big( \lambda u^2 + 2 \beta u v + \delta \zeta \iota v^2 \big) \,,
\end{split}
\eeqn
and $p(u, v) =   (\gamma\iota-\delta\eta)^2(\varepsilon\iota-\zeta\eta)^2 (\gamma\zeta-\delta\varepsilon)^2 v^7 + \dots - 4 \kappa^3 u^7$ is a homogeneous polynomial of degree seven.
\par Applying the Nikulin involution in Proposition~\ref{NikulinInvolution} to the  pencil of planes $L_3(u, v)$ we obtain a pencil, denoted by $C_3(u, v)=0$ with $[u:v] \in \mathbb{P}^1$. A computation yields
\beq
\label{eqn:pencil_bfd2}
 C_3(u, v) = v \mathbf{Z} \big( 2 \gamma \mathbf{X}- \delta \mathbf{W} \big)- \big(2\eta \mathbf{X} - \mathbf{W}\big) u \mathbf{W}^2  \big(2\kappa \mathbf{X} - \mathbf{W}\big) \,,
\eeq
such that the fibration is also obtained by intersecting the quartic $\mathcal{Q}$ with the pencil $C_3(u, v)=0$. A fibration with the same singular fibers but for different parameters can be obtained in the same fashion using the line $L_5$ instead of  $L_3$; in this case, the moduli $(\varepsilon, \zeta) \leftrightarrow (\kappa, \lambda)$ are swapped according to the symmetries in Lemma~\ref{symmetries1}.
\subsubsection{Fibration~\texorpdfstring{$(3')$}{(3')}}
A pencil of quadratic surfaces, denoted by $\tilde{C}_3(u, v)=0$ with $[u:v] \in \mathbb{P}^1$ is given by
\beq
\label{eqn:pencil_extra1}
\begin{split}
 \tilde{C}_3(u, v) & = \varepsilon \big(\kappa u - \lambda v\big) \big(2 \varepsilon \mathbf{X} - \zeta \mathbf{W}\big) \big( 2 \kappa \mathbf{X} - \lambda \mathbf{W} + \gamma \kappa \eta \mathbf{Z}\big) \\
 & - \kappa \big(\varepsilon u - \zeta v\big) \big(2 \kappa \mathbf{X} - \lambda \mathbf{W}\big) \big( 2 \varepsilon \mathbf{X} - \zeta \mathbf{W} + \gamma \varepsilon \eta \mathbf{Z}\big) \,.
\end{split} 
\eeq
Making the substitutions 
\beq
\label{eqn:substitution_bfd2}
\begin{array}{ll}
\,\mathbf{X} & =  \gamma^2 \varepsilon^2 \eta^2 \kappa^2  v \big(\gamma u-\delta v\big)\big(\eta u-\iota v\big)  q_1(x, z, u, v) z\,, \\ [0.4em]
\,\mathbf{Y} & = \sqrt{2} \gamma \epsilon \eta \kappa \big(\gamma u-\delta v\big)\big(\eta u-\iota v\big)  y z \,, \\  [0.4em]
\,\mathbf{Z} & = 2 q_2(x, z, u, v) \, q_3(x, z, u, v) \,, \\ [0.6em]
\mathbf{W} & = 2 \gamma^2 \varepsilon^2 \eta^2 \kappa^2 v^2  \big(\gamma u-\delta v\big)\big(\eta u-\iota v\big) x z  \,,
\end{array}
\eeq
in Equation~(\ref{quartic1}), compatible with $\tilde{C}_3(u, v)=0$, and using the polynomials
\beq
\begin{array}{ll}
q_1(x, z, u, v) & = u x - \gamma \varepsilon \eta \kappa v  \big(\gamma u - \delta v\big) \big(\varepsilon u - \zeta v\big)   \big(\kappa u - \lambda v\big)  \big(\eta u - \iota v\big)  z \,, \\ [0.4em]
q_2(x, z, u, v) & = x- \gamma \varepsilon \eta \kappa^2 v \big(\gamma u - \delta v\big) \big(\varepsilon u - \zeta v\big)  \big(\eta u - \iota v\big)  z \,, \\ [0.4em]
q_3(x, z, u, v) & = x- \gamma \varepsilon^2 \eta \kappa v \big(\gamma u - \delta v\big) \big(\kappa u - \lambda v\big)  \big(\eta u - \iota v\big)  z \,, 
\end{array}
\eeq
determines a Jacobian elliptic fibration  $\pi\colon \mathcal{X} \to \mathbb{P}^1$ with fiber $\mathcal{X}_{[u:v]}$ given by 
\beq
\label{eqn:bfd2}
\mathcal{X}_{[u:v]}\colon \quad y^2 z = x^3 + e(u, v) \, x^2 z + f(u, v) \, x z^2 + g(u, v) \, z^3 \,.
\eeq
The fibration admits the section $\sigma: [x:y:z] = [0:1:0]$. Here, the discriminant is
\beq
\Delta(u, v) =  v^{10} \big(\gamma u - \delta v\big)^2 \big(\varepsilon u - \zeta v\big)^2   \big(\kappa u - \lambda v\big)^2  \big(\eta u - \iota v\big)^2  p(u, v) \,,
\eeq
and
\beqn
\begin{array}{lrl}
e(u, v)  &= &  -\gamma \varepsilon \eta \kappa v \Big( 3 \gamma \varepsilon \eta \kappa u^3 - 3  ( \gamma \zeta \eta \kappa + \delta \varepsilon \eta \kappa+ \gamma \varepsilon \eta \lambda + \gamma \varepsilon \iota \kappa) u^2v \\ [0.2em]
&& +  (3 \alpha \gamma \varepsilon \eta \kappa + 2 \delta \zeta \eta \kappa + 2 \gamma \zeta \eta \lambda + 2 \gamma \zeta \iota \kappa 
+2 \delta \varepsilon \eta \lambda+ 2   \delta \varepsilon\iota \kappa + \gamma \varepsilon \iota \lambda) uv^2 \\ [0.2em]
&& + (2 \beta \gamma \varepsilon \eta \kappa-\delta \zeta \eta \lambda- \delta \zeta \iota \kappa- \gamma \zeta \iota \lambda- \varepsilon \delta \iota \lambda) v^3 \Big) \,,\\ [0.4em]
f(u, v)  &= & \gamma^2 \varepsilon^2 \eta^2 \kappa^2 v^2  \big(\gamma u - \delta v\big) \big(\varepsilon u - \zeta v\big)   \big(\kappa u - \lambda v\big)  \big(\eta u - \iota v\big) \\[0.2em]
 &\times & \Big( 3 \gamma \varepsilon \eta \kappa u^2  - 3 ( \gamma \zeta \eta \kappa + \delta \varepsilon \eta \kappa+ \gamma \varepsilon \eta \lambda + \gamma \varepsilon \iota \kappa) uv\\ 
& &+ (\gamma^2 \varepsilon^2 \eta^2 \kappa^2 + 3 \alpha \gamma \varepsilon \eta \kappa + \delta \zeta \eta \kappa + \gamma \zeta \eta \lambda + \gamma \zeta \iota \kappa 
+\delta  \varepsilon \eta \lambda+  \delta \varepsilon \iota \kappa + \gamma \varepsilon \iota \lambda) v^2 \Big)\,, \\ [0.4em]
g(u, v) &= & -\gamma^3 \varepsilon^3 \eta^3 \kappa^3 v^3  \big(\gamma u - \delta v\big)^2 \big(\varepsilon u - \zeta v\big)^2   \big(\kappa u - \lambda v\big)^2  \big(\eta u - \iota v\big)^2 \\ [0.2em]
& \times &\big( \gamma \varepsilon \eta \kappa u- (\gamma \zeta \eta \kappa+ \delta \varepsilon \eta \kappa+ \gamma \varepsilon \eta \lambda+ \gamma \varepsilon \iota \kappa) v\big),
\end{array}
\eeqn
and $p(u, v) =  -27 (\gamma\varepsilon\eta\kappa)^{12} u^6 + \dots$ is a homogeneous polynomial of degree six.
\par Applying the Nikulin involution in Proposition~\ref{NikulinInvolution} to $\tilde{C}_3(u, v)$ we obtain a pencil, denoted by $C'_3(u, v)=0$ with $[u:v] \in \mathbb{P}^1$.  A computation yields
\beq
\label{eqn:pencil_extra2}
\begin{array}{c}
 C'_3(u, v)= - u \gamma \varepsilon \eta \kappa \mathbf{W}^3  \\
 + v \Big( 2 \gamma \varepsilon \eta \kappa \mathbf{W}^2 \mathbf{X} + \delta \iota \mathbf{W}^2 \mathbf{Z} - 2 (\gamma \iota + \delta \eta) \mathbf{W}\mathbf{X}\mathbf{Z} 
 + 4 \gamma \eta \mathbf{X}^2 \mathbf{Z} \Big)  \,,
\end{array} 
\eeq
such that the fibration is also obtained by intersecting the quartic $\mathcal{Q}$ with the pencil $C'_3(u, v)=0$. 
\subsubsection{Fibration~\texorpdfstring{$(4)$}{(4)}}
\label{sssec:WEmax}
An elliptic fibration with section, called the \emph{maximal fibration}, is induced by intersecting the quartic surface $\mathcal{Q}$ with the pencil
\beq
\label{eqn:pencil_max1}
 L_4(u, v)=u \mathbf{W} - 2v (2\mathbf{X} +\gamma\eta \mathbf{Z}) =0
\eeq
for $[u:v] \in \mathbb{P}^1$. Making the substitutions 
\beq
\label{eqn:substitution_max}
\begin{split}
&\mathbf{X}  = u^2 v x\,, \quad  \mathbf{Y}=  \sqrt{2} u y \,, \quad \mathbf{Z} =  2u v^4 \big( \varepsilon u\ -\zeta v \big )  \big( \kappa u\ - \lambda v \big )  z\,, \\
&\mathbf{W}  = 2 v^2 \Big( u x - \gamma \eta ( \varepsilon u\ -\zeta v \big )  \big( \kappa u - \lambda v \big ) v^3 z \Big) \,,
\end{split}
\eeq
into Equation~(\ref{quartic1}), compatible with $L_4(u, v)=0$, determines a Jacobian elliptic fibration   $\pi\colon \mathcal{X} \to \mathbb{P}^1$ with fiber $\mathcal{X}_{[u:v]}$ given by
\beq
\label{eqn:max}
\mathcal{X}_{[u:v]}\colon \quad y^2 z = x^3 + e(u, v) \, x^2 z + f(u, v) \, x z^2 + g(u, v) \, z^3 \,.
\eeq
The fibration admits the section $\sigma\colon [x:y:z] = [0:1:0]$. Here, the discriminant is
\beq
\begin{split}
\Delta(u, v) & = v^{12} \big( \varepsilon u -  \zeta v\big)^2 \big(\kappa  u - \lambda  v\big)^2 p(u, v)\,,
\end{split}
\eeq
and
\beq
\begin{split}
e(u, v)  & = \frac{v}{u} \Big( u^4 -(3 \alpha -\gamma \varepsilon \eta\kappa) u^2 v^2 - 2 \big(  \beta + \gamma\varepsilon\eta\lambda+ \gamma \zeta \eta \kappa\big) u v^3 +3 \gamma \zeta \eta \lambda v^4 \Big)\,, \\
f(u, v)  & = -\frac{v^5}{u^2}  \big( \varepsilon u -  \zeta v\big) \big( \kappa  u - \lambda v\big) \Big( \big(\gamma\iota + \delta\eta \big) u^3+  \big(3 \alpha\gamma\eta -  \delta\iota\big) u^2 v \\
& +  \gamma\eta \big(4\beta + \gamma \varepsilon \eta\lambda + \gamma \zeta \eta \kappa \big) uv^2 - 3 \gamma^2\zeta\eta^2\lambda v^3 \Big)\,,\\
g(u, v) &=  \frac{\gamma\eta v^{9}}{u^3}  \big( \varepsilon u -  \zeta v\big)^2 \big(\kappa  u - \lambda v\big)^2 \Big(\delta \iota u^2 - 2 \beta \gamma \eta uv +  \gamma^2 \zeta \eta^2 \lambda v^2\Big) \,,
\end{split}
\eeq
and $p(u, v) =  (\gamma \iota - \delta \eta)^2  u^8 + O(v)$ is a homogeneous polynomial of degree eight. Upon eliminating the term proportional to $x^2 z$ in Equation~(\ref{eqn:max}) by a shift, we obtain a Weierstrass model such that the coefficients of $x z^2$ and $z^3$ are homogeneous polynomials, and all denominators cancel.
\par Applying the Nikulin involution in Proposition~\ref{NikulinInvolution} to the  pencil of planes $L_4(u, v)$ we obtain a pencil, denoted by $C_4(u, v)=0$ with $[u:v] \in \mathbb{P}^1$.  A computation yields
\beq
\label{eqn:pencil_max2}
\begin{array}{c}
 C_4(u, v)= u \mathbf{W} \mathbf{Z} \big( 2 \gamma \mathbf{X}- \delta \mathbf{W} \big) \big( 2\eta \mathbf{X} -  \mathbf{W} \big)  -v   \Big(\gamma \zeta \eta \mathbf{W}^4 \\[0.4em]
- 2 \gamma \eta (\varepsilon + \zeta \kappa) \mathbf{W}^3  \mathbf{X} + 4 \gamma \varepsilon \eta \kappa \mathbf{W}^2  \mathbf{X}^2 + 2 \delta  \mathbf{W}^2 \mathbf{X}  \mathbf{Z}   - 4 (\gamma+\delta\eta)  \mathbf{W}  \mathbf{X}^2  \mathbf{Z} + 8 \gamma \eta  \mathbf{X}^3  \mathbf{Z}\Big)  \,,
\end{array} 
\eeq 
such that the fibration is also obtained by intersecting the quartic $\mathcal{Q}$ with the pencil $C_4(u, v)=0$. 
\end{proof}
\subsection{Projective model for \texorpdfstring{$P'_{14}$}{P'}-polarized K3 surfaces}
We also consider the projective surface $\mathcal{Q}'(f_0,f_1,f_2, g_0, g_1, h_0,h_1,h_2)$ in $\mathbb{P}^3=\mathbb{P}(\mathbf{X}, \mathbf{Y}, \mathbf{Z}, \mathbf{W})$ with a coefficient set $(f_0, f_1, f_2, g_0, g_1, h_0, h_1, h_2) \in \mathbb{C}^8$, defined by 
\beq
\label{mainquartic2}
\begin{split}
0= \mathbf{Y}^2 \mathbf{Z} \mathbf{W}-4 \mathbf{X}^3 \mathbf{Z}- 2 \Big(  \mathbf{W}^2 + f_2  \mathbf{W} \mathbf{Z} + h_2 \mathbf{Z}^2 \Big) \, \mathbf{X}^2 \phantom{\,.} \\
-  \Big(   f_1 \mathbf{W} \mathbf{Z} +   g_1 \mathbf{W}^2 + h_1 \mathbf{Z}^2 \Big)  \,  \mathbf{X} \mathbf{Z}
- \frac{1}{2}  \Big(   f_0 \mathbf{W} \mathbf{Z} +   g_0 \mathbf{W}^2 + h_0 \mathbf{Z}^2 \Big)  \, \mathbf{Z}^2\,.
\end{split}
\eeq
The projective automorphism
\beq
\begin{array}{rccl}
 \phi_1\colon& [\mathbf{X}: \mathbf{Y}: \mathbf{Z}: \mathbf{W}] & \mapsto &  [\Lambda_1 \mathbf{X} : \ \Lambda_1^2 \mathbf{Y}: \  \Lambda_1^{-3} \mathbf{Z}:\  \Lambda_1^{-1} \mathbf{W} ]  \,,\end{array} 
\eeq
changes a given parameter set of the quartic for $\Lambda_1 \in \mathbb{C}^\times$ according to
\beq
\begin{array}{l}
 \Big(f_2, f_1, f_0, g_1, g_0, h_2, h_1, h_0 \Big) \mapsto \\
 \qquad \Big( f_2 \Lambda_1^2, \ f_1 \Lambda_1^6, \ f_0 \Lambda_1^{10}, \ g_1 \Lambda_1^4, \ g_0 \Lambda_1^8, \ h_2 \Lambda_1^4, \ h_1 \Lambda_1^8, \ h_0 \Lambda_1^{12}\Big) \,.
 \end{array}
\eeq
One can also use a linear substitution $\mathbf{X} \mapsto \mathbf{X} + \Lambda_2 \mathbf{Z}$ for $\Lambda_2 \in \mathbb{C}$. The induced projective automorphism $\phi_2$ transforms Equation~(\ref{mainquartic2}) into an equation of the same type, but with transformed moduli given by
\beq
\label{eqn:shift}
 \left( \begin{array}{c} f_2 \\ f_1 \\ f_0  \\ g_1 \\g _0 \\ h_2 \\ h_1 \\ h_0\end{array} \right) 
\ \mapsto \  \left( \begin{array}{l} f_2 \\ f_1 - 4 f_2 \Lambda_2 \\ f_0 - 2 f_1 \Lambda_2 +4 f_2 \Lambda_2^2  \\ g_1 - 4  \Lambda_2 \\ g_0 -2 g_1 \Lambda_2 + 4 \Lambda_2^2\\
 h_2 - 6 \Lambda_2 \\ h_1 -4 h_2 \Lambda_2 + 12 \Lambda_2^2 \\ h_0 - 2 h_1 \Lambda_2 + 4 h_2 \Lambda_2^2 - 8 \Lambda_2^3\end{array} \right) \,.
\eeq
\par Equation~(\ref{mainquartic2}) defines a family of quartic hypersurfaces whose minimal resolution is a K3 surface $\mathcal{X}'$ of Picard rank 14.  We have the following:
\begin{proposition}
\label{NikulinInvolution_SD}
Let $(f_0, f_1, f_2, g_0, g_1, h_0, h_1, h_2) \in \mathbb{C}^8$ as before. The van Geemen-Sarti involution on $\mathcal{X}'$ is induced by the projective automorphism
\beq
\label{eq:NikulinInvolutions_SD}
\begin{split}
 \Psi\colon\quad  \mathbb{P}^3   \ \to \ & \  \mathbb{P}^3, \\
  [\mathbf{X}: \mathbf{Y}: \mathbf{Z}: \mathbf{W}]  \mapsto & \ [ Q\big(2 \mathbf{X}, \mathbf{Z}\big) \, \mathbf{X}\mathbf{W}   : \  - Q\big(2 \mathbf{X}, \mathbf{Z}\big) \, \mathbf{Y}\mathbf{W}   : \\
  & \ \ Q\big(2 \mathbf{X}, \mathbf{Z}\big) \, \mathbf{Z} \mathbf{W}  : \ 8 H\big(2 \mathbf{X}, \mathbf{Z}\big) \, \mathbf{Z}  ] \,,
\end{split}
\eeq
with $Q(u, v)=u^2 + g_1 uv + g_0v^2$ and $H(u, v)=u^3 + h_2 u^2 v + h_1 uv^2+ h_0v^3$.
\end{proposition}
\begin{proof}
The proof is analogous to the proof of Proposition~\ref{NikulinInvolution}.
\end{proof}
\par We use the automorphism $\phi_2$ to eliminate one parameter from the parameter set $(f_0, f_1, f_2, g_0, g_1, h_0, h_1, h_2)$ and  obtain seven coordinates on a weighted projective space associated with the equivalence relation induced by the action of $\phi_1$. It turns out that a convenient choice is given by $h_2+ g_1 =0$; this will become clear presently, as we employ the results from Section~\ref{sssec:PPmoduli}. The constraint, $h_2+ g_1 =0$, is invariant under the action of $\phi_1$, and is achieved by setting $10\Lambda_2=h_2 + g_1$ in $\phi_2$ in Equation~(\ref{eqn:shift}). Thus, it is enough to consider the quartic surface $\mathcal{Q}'(f_2, f_1, f_0, g_0, h_2, h_1, h_0)$ given by Equation~(\ref{quartic2}). Using the parameters $\{ \mathcal{J}_{2k} \}$ defined in Equation~(\ref{eqn:CurlyJinvariants}), we assume
\beq
\label{eqn:generic_parameters_PP}
 \not \exists r, s \, \in \, \mathbb{C}:  ( \mathcal{J}_2 ,  \mathcal{J}_6, \mathcal{J}_8,  \mathcal{J}_{10},  \mathcal{J}_{12},  \mathcal{J}_{16},  \mathcal{J}_{20} )
=  (s^2, 2 r s^2, 10r^2, s^2r^2, -20r^3, -15r^4, - 4 r^5).
\eeq
On checks that under the action of $\phi_1$ the parameters transform according to
\beqn
 ( \mathcal{J}_2, \mathcal{J}_6, \mathcal{J}_8, \mathcal{J}_{10}, \mathcal{J}_{12}, \mathcal{J}_{16}, \mathcal{J}_{20} )   \ \to \ 
(\Lambda^{2} \mathcal{J}_2, \, \Lambda^6 \mathcal{J}_6, \, \Lambda^8 \mathcal{J}_8, \, \Lambda^{10} \mathcal{J}_{10}, \, \Lambda^{12} \mathcal{J}_{12}, \, \Lambda^{16} \mathcal{J}_{16}, \, \Lambda^{20} \mathcal{J}_{20}).
\eeqn
We have the following:
\begin{theorem}
\label{thm_parity}
Assuming Equation~(\ref{eqn:generic_parameters_PP}), the surface obtained as the minimal resolution of the quartic in Equation~(\ref{quartic2}) is a K3 surface $\mathcal{X}'$ endowed with a canonical  $P'_{14}$-polarization. Conversely, every $P'_{14}$-polarized K3 surface has a birational projective model given by Equation~(\ref{quartic2}). In particular, the Jacobian elliptic fibrations in Lemma~\ref{prop:ADE} are attained as follows:
\beqn
\begin{array}{c|c|c|c|l}
\# 	&  \text{singular fibers} 	& \operatorname{MW} & \text{root lattice} & \text{substitution} \; [ \mathbf{X} : \mathbf{Y}: \mathbf{Z} : \mathbf{W} ] = \\
\hline
1 	& III^* + 5 I_2 + 5 I_1	& \mathbb{Z}/2\mathbb{Z}	& E_7 \oplus A_1^{\oplus 5}	&  [ u v x : \sqrt{2} y : 2v^2 z : 2v^3 H(u, v)z ] \\ 
\hline
2	& I_4^* + I_0^* + 8 I_1	&  \{ \mathbb{I} \} 		& D_8 \oplus D_4			&  [ 2 u v x:  y : 8u^4v^2 z : 32 u^3 v^3 z]  
\end{array}
\eeqn
Here, we have set $H(u, v)=u^3 + h_2 u^2 v + h_1 uv^2+ h_0v^3$.
\end{theorem}
\begin{proof}
One constructs the explicit Weierstrass models using the substitutions provided in the statement.  Using fibration~$(2)$ it follows immediately that a K3 surface $\mathcal{X}'$ is endowed with a canonical  $P'_{14}$-polarization.  The given substitution for fibration~$(1)$ leads to the Weierstrass model
\beq
\label{eqn:pf_alt_E7}
  \mathcal{X}'\colon \quad y^2 z  = x^3 + v^2  F(u, v) \, x^2 z + v^3 H(u, v) \,G(u, v) x z^2  \,,
\eeq
where $F(u, v)=f_2  u^2  + f_1 uv+ f_0v^2$ and $G(u, v)=u^2  + g_1 uv+ g_0v^2$ with $g_1=-h_2$. This is precisely the alternate fibration in Equation~(\ref{eqn:alt_E7}) from Section~\ref{sssec:PPmoduli}: we have $C(u, v)=F(u, v)$, and $H(u, v) \, G(u, v) = D(u, v)$ if and only if  $h_2+ g_1 =0$, and the respective parameters are related by $(c_2, c_1, c_0) =(f_2, f_1, f_0)$ and
\beq
\begin{split}
 d_3 = g_0 + h_1 - h_2^2 \,,\quad   d_2 = g_0 h_2 - h_1 h_2 + h_0 \,, \quad d_1 = g_0 h_1 -  h_0 h_2 \,,\quad d_0 = g_0 h_0 \,.
\end{split} 
\eeq
These relations follow immediately from Equation~(\ref{modulispace_lattice2}) and Equation~(\ref{eqn:CurlyJinvariants}). For fibration~$(1)$ the condition in Equation~(\ref{eqn:generic_parameters_PP}) is equivalent to the corresponding Weierstrass model being irreducible and minimal; see proof of Theorem~\ref{thm:moduli_space_seldual}.
\par Conversely, Proposition~\ref{cor:EFS3}  proves that every $P'_{14}$-polarized K3 surface admits a unique alternate fibration, and it follows from an alternate fibration a quartic can be constructed using Equation~(\ref{eqn:pf_alt_E7}). Thus, every $P'_{14}$-polarized K3 surface, up to isomorphism, is realized as the resolution of the quartic in Equation~(\ref{quartic2}). 
\end{proof}
\subsection{Projective model for \texorpdfstring{$P''_{14}$}{P''}-polarized K3 surfaces}
\label{ssec:EllFibVinb}
We also prove the analogue of Theorem~\ref{thm1} for the family of K3 surfaces defined by Vinberg in \cite{MR2718942}. It isomorphic to the family of $P''_{13}$-polarized K3 surfaces. Since $A, B, C \neq 0$ in \cite{MR2718942}*{Eqn.~\!(13)}, we rescale the coordinates to achieve $A=-1, B=4, C=1$. Let $( f_{1,2}, f_{2,2}, f_{1,3}, f_{2,3}, f_{3,3}, g_0, g_1, g_3) \in \mathbb{C}^{8}$ be general parameters.  Consider the projective surface $\mathcal{Q}''(f_{1,2}, \dots, g_3)$ in $\mathbb{P}^3 = \mathbb{P}(\mathbf{x}_0, \mathbf{x}_1, \mathbf{x}_2, \mathbf{x}_3)$  defined by the homogeneous quartic equation
\beq
\label{eqn:Vinberg_body}
  \mathbf{x}_0^2 \mathbf{x}_2 \mathbf{x}_3 - 4 \mathbf{x}_1^3 \mathbf{x}_3 - \mathbf{x}_2^4 - \mathbf{x}_1 \mathbf{x}_3^2 \, g(\mathbf{x}_0, \mathbf{x}_1, \mathbf{x}_3) - \mathbf{x}_2 \mathbf{x}_3 \, f(\mathbf{x}_1, \mathbf{x}_2, \mathbf{x}_3) = 0 \,,
\eeq
with
\beq
  g = g_0 \mathbf{x}_0 + g_1 \mathbf{x}_1 + g_3 \mathbf{x}_3 \,, \qquad 
  f =  f_{12} \mathbf{x}_1 \mathbf{x}_2 + f_{22} \mathbf{x}_2^2  + f_{13} \mathbf{x}_1 \mathbf{x}_3  + f_{23} \mathbf{x}_2 \mathbf{x}_3 + f_{33} \mathbf{x}_3^2  \,.
\eeq
Setting $g_0=0$ we obtain Equation~(\ref{eqn:Vinberg_body1}). We then have the following:
\begin{theorem}
\label{thm_Vinberg}
Assume that $( f_{1,3}, f_{2,3}, f_{3,3}, g_0, g_1, g_3 ) \neq 0$.  The minimal resolution of Equation~(\ref{eqn:Vinberg_body}) is a K3 surface $\mathcal{X}''$ endowed with a canonical  $P''_{13}$-polarization. Conversely, every $P''_{13}$-polarized K3 surface has a birational projective model given by Equation~(\ref{eqn:Vinberg_body}). In particular, the Jacobian elliptic fibrations in Lemma~\ref{prop:ADE} are attained as follows:
\beqn
\begin{array}{c|c|c|c|l}
\# 	&  \text{singular fibers} 	& \operatorname{MW} & \text{root lattice} & \text{substitution} \; [\mathbf{x}_0 : \mathbf{x}_1: \mathbf{x}_2 : \mathbf{x}_3 ] = \\
\hline
1 	& II^* + I_4 + 10 I_1		& \{ \mathbb{I} \} 	& E_8 + A_3		&  [ y + g_0 v^2/2 \, :  u v x : 4 u^3 v^3 z : 4 u^2 v^4 z] \\ 
\hline
2	& I_7^* + 11 I_1		&  \{ \mathbb{I} \} 	& D_{11} 			&  [ \sqrt{2}y + g_0 u v^5 z :  u v x : 2v^2 x : 4v^6 z]  
\end{array}
\eeqn
\noindent The Jacobian elliptic fibrations on a general $P''_{14}$-polarized K3 surface in Lemma~\ref{prop:ADE} are attained by setting $g_0=0$ above.  They are as follows:
\beqn
\begin{array}{c|c|c|c|l}
\# 	&  \text{singular fibers} 	& \operatorname{MW} & \text{root lattice} & \text{substitution} \; [\mathbf{x}_0 : \mathbf{x}_1: \mathbf{x}_2 : \mathbf{x}_3 ] = \\
\hline
1 	& II^* + I_0^* + 8 I_1		& \{ \mathbb{I} \} 	& E_8 + D_4		&  [ y :  u v x : 4 u^3 v^3 z : 4 u^2 v^4 z] \\ 
\hline
2	& I_8^* + 10 I_1		&  \{ \mathbb{I} \} 	& D_{12} 			&  [ \sqrt{2}y :  u v x : 2v^2 x : 4v^6 z]  
\end{array}
\eeqn
Moreover, for $g_0=g_3=0$ the polarizing lattice extends to $H \oplus E_8(-1) \oplus D_5(-1)$.
\end{theorem}
\begin{proof}
One constructs the explicit Weierstrass models using the substitutions provided in the statement. Using fibration~$(1)$ it follows immediately that a K3 surface $\mathcal{X}''$ is endowed with a canonical  $H \oplus E_8(-1) \oplus A_3(-1)$-polarization. We proved in Proposition~\ref{cor:EFS5} that there are only two inequivalent Jacobian elliptic fibrations on K3 surfaces with a N\'eron-Severi lattice isomorphic to $H \oplus E_8(-1) \oplus A_3(-1)$. We realized both as explicit Weierstrass models. The Vinberg quartic determines a K3 surface if and only if the given substitution for fibration~$(1)$ determines an irreducible, minimal Weierstrass model. One checks using fibration~$(1)$ that this is the case if and only if
\beq
 \Big( f_{1,3}, f_{2,3}, f_{3,3}, g_0, g_1, g_3 \Big) \neq 0 \,.
\eeq
Conversely, it was proved in \cite{MR2718942} that every $H \oplus E_8(-1) \oplus A_3(-1)$-polarized K3 surface, up to isomorphism, is realized as the minimal resolution of a quartic in Equation~(\ref{eqn:Vinberg_body}). Lastly, it was proven in \cite{MR2718942} that the extension in Equation~(\ref{eqn:lattices}) to $n=6$ occurs along the locus $g_0=0$, and to $n=5$ along $g_0=g_3=0$. Similarly, the extension to $n=4$ occurs when $f_{33}=g_0=g_3=0$. One checks that fibration~$(1)$ has singular fibers $II^* + I_1^* + 7 I_1$ and $II^* + I_2^* + 6 I_1$ for $g_0=g_3=0$ and $g_0=g_3=f_{33}=0$, respectively. 
\end{proof}
We also have:
\begin{proposition}
\label{prop:coincide}
Using the notation above, for $(\eta, \iota)=(\kappa, \lambda)=(0,1)$ and $g_0=g_3=f_{33}=0$ and 
\beq
\label{eqn:coeffs_invs}
\begin{array}{c}
 f_{1,2}=- 3 \alpha , \quad  f_{2,2} =- \beta ,\quad  g_1 = \gamma \varepsilon , \quad f_{1,3} = - \frac{\gamma \zeta + \delta \varepsilon}{2} , \quad f_{2,3} =\frac{\delta \zeta}{4},
\end{array}
\eeq
the quartics in Equation~(\ref{eqn:Vinberg_body}) and Equation~(\ref{quartic1}) are birationally equivalent. Moreover, the remaining invariants are generators of $A(\mathcal{D}_4, \Gamma_4)$ in Equation~(\ref{eqn:generators}).
\end{proposition}
\begin{proof}
The birational map $\mathbb{P}^3  \dashrightarrow  \mathbb{P}^3$ given by
\beqn
\begin{split}
  [\mathbf{X}: \mathbf{Y}: \mathbf{Z}: \mathbf{W}]  \mapsto & \ [  2 \mathbf{x}_1\mathbf{x}_2 \ : \ 2 \mathbf{x}_0 \mathbf{x}_2 \ : \ - \mathbf{x}_3 \big(2 \varepsilon \mathbf{x}_1 - \zeta \mathbf{x}_2\big) \ : \ 2 \mathbf{x}_2^2  ] \,,
\end{split}
\eeqn
with the birational inverse $\mathbb{P}^3  \dashrightarrow   \mathbb{P}^3$ given by
\beqn
\label{eqn:rescaling_Vinberg}
\begin{split}
  [\mathbf{x}_0: \mathbf{x}_1: \mathbf{x}_2: \mathbf{x}_3]  \mapsto & \ [  \big(2 \varepsilon \mathbf{X} - \zeta \mathbf{W}\big) \mathbf{Y} \ : \  \big(2 \varepsilon \mathbf{X} - \zeta \mathbf{W}\big) \mathbf{X}
   \ : \  \big(2 \varepsilon \mathbf{X} - \zeta \mathbf{W}\big) \mathbf{W} \ : \ - 2 \mathbf{Z} \mathbf{W} ] \,,
\end{split}
\eeqn
realizes the equivalence. We already proved in \cite{MR4015343}  that in the case $(\eta, \iota)=(\kappa, \lambda)=(0,1)$ the non-vanishing invariants coincide with the generators of $A(\mathpzc{D}_4, \Gamma_4)$ in Equation~(\ref{eqn:generators}) defined by Vinberg.
\end{proof}
\par For $\Lambda \in \mathbb{C}^\times$ the projective automorphism, given by
\beqn
\label{tmor_Vinberg}
\mathbb{P}^3 \ \to \ \mathbb{P}^3, \qquad [\mathbf{x}_0: \mathbf{x}_1: \mathbf{x}_2: \mathbf{x}_3] \ \mapsto \ [\ \Lambda \mathbf{x}_0: \ \mathbf{x}_1 :\ \Lambda^{-2} \mathbf{x}_2: \ \Lambda^{-8}\mathbf{x}_3 \ ] \,,
\eeqn
extends to an isomorphism of K3 surfaces that rescales the coefficients according to
\beq
\label{eqn:scaling_Vinberg}
\begin{array}{l}
 \Big( f_{1,2}, f_{2,2}, g_1, f_{1,3}, f_{2,3}, g_0^2, g_3, f_{3,3}\Big)  \mapsto \hfill \\[0.2em]
\qquad  \qquad   \Big( \Lambda^4 f_{1,2},  \Lambda^6 f_{2,2},  \Lambda^8 g_1,  \Lambda^{10} f_{1,3},  \Lambda^{12} f_{2,3},  \Lambda^{14} g_0^2,  \Lambda^{16} g_3,  \Lambda^{18} f_{3,3}\Big) \,.
\end{array}  
\eeq
Moreover, one can tell precisely when two members of the family in Equation~(\ref{eqn:Vinberg_body}) are isomorphic. Using an appropriate normal form for fibration~$(1)$ in Theorem~\ref{thm_Vinberg} and an analogous argument as in Sections~\ref{ssec:moduli_space} and~\ref{sssec:PPmoduli},  it follows that two members  are isomorphic if and only if their coefficient sets are related by Equation~(\ref{eqn:scaling_Vinberg}).  We use Equation~(\ref{eqn:inv_PPP}), set  $\mathscr{J}_{14} = g_0^2, \mathscr{J}_{16}=g_3, \mathscr{J}_{18}=f_{3,3}$,  and obtain the following:
\begin{theorem}
\label{thm:Vinberg_moduli}
The seven-dimensional open analytic space
\beq
  \Big \{  \ 
\Big[   \mathscr{J}_4 :  \dots : \mathscr{J}_{18}  \Big]   \in  \mathbb{WP}_{(4,6,8,10,12,14, 16,18)} \ \Big| \
( \mathscr{J}_8 ,  \mathscr{J}_{10},  \mathscr{J}_{12},  \mathscr{J}_{14}, \mathscr{J}_{16}, \mathscr{J}_{18}) \neq 0 \
\Big \} 
\eeq
forms a coarse moduli space of $P''_{13}$-polarized K3 surfaces. Moreover, the coarse moduli space of $P''_{14}$-polarized K3 surfaces is the subspace $\mathscr{J}_{14}=0$, the coarse moduli space of $H \oplus E_8(-1) \oplus D_5(-1)$-polarized K3 surfaces  is the subspace $\mathscr{J}_{14}=\mathscr{J}_{16}=0$.
\end{theorem}
\begin{proof}
One checks that the Weierstrass models in Theorem~\ref{thm_Vinberg} only depend on $g_0^2$. It follows that $\mathscr{J}_{14} = g_0^2$ is a coordinate on the moduli space. As proved by Vinberg  in \cite{MR2718942} the invariants $\mathscr{J}_4, \dots, \mathscr{J}_{18}$, up to the rescaling given by Equation~(\ref{eqn:coeffs_invs}), freely generate the coordinate ring of the moduli space. The remainder of the statement follows directly from Theorem~\ref{thm_Vinberg} or was already proven in \cite{MR2718942}.
\end{proof}
\section{Dual graphs of rational curves}
\label{sec:dual_graphs}
In this section we determine the dual graphs of smooth rational curves and their intersection properties on the K3 surfaces $\mathcal{X}$, $\mathcal{X}'$, and $\mathcal{X}''$ with the N\'eron-Severi lattices $P_{14}$, $P'_{14}$, and $P''_{14}$, respectively. We define the \emph{dual graph} of smooth rational curves to be the simplicial complex whose set of vertices is a set of smooth rational curves on a K3 surface such that the vertices $\Sigma, \Sigma'$ are joined by an $m$-fold edge if and only if $\Sigma \circ \Sigma' = m$. For Picard number bigger than or equal to $15$, the possible dual graphs of all smooth rational curves on K3 surfaces with finite automorphism groups were determined in  \cite{MR556762}*{Sec.~\!4}.  
\subsection{The dual graph for \texorpdfstring{$P_{14}$}{P}-polarized K3 surfaces}
\label{ssec:dual_graph_P}
We will now determine the dual graph of smooth rational curves for the K3 surfaces $\mathcal{X}$ in Theorem~\ref{thm1} with N\'eron-Severi lattice $P_{14}$. The case of the K3 surfaces $\mathcal{X}$ of Picard rank 15 and the embedding of the reducible fibers into the dual graph of smooth rational curves are determined in Appendix~\ref{sec:rank15}. Results for the case of Picard rank 16 were obtained by the authors in \cite{CHM19}.
\par We consider the following complete intersections that can be easily checked to lie on the quartic $\mathcal{Q}$ in Equation~(\ref{quartic1}) and lift to rational curves on the smooth K3 surface:
\smallskip
\beqn
\begin{split}
R_1\colon &\left\lbrace 
  \begin{array}{rcl}
   0 & = &  2\varepsilon \mathbf{X}-\zeta \mathbf{W}\\
   0 & = & \big(3 \alpha \varepsilon ^2 \zeta + 2 \beta \varepsilon ^3 - \zeta^3 \big) \mathbf{W}^2 -  \varepsilon \big(\varepsilon \iota-\eta\zeta\big) \big( \delta \varepsilon-\gamma \zeta \big) \mathbf{Z} \mathbf{W} 
   + 2 \varepsilon^3 \mathbf{Y}^2,
  \end{array} \right. \\[0.5em]
R_2 \colon&\left\lbrace 
  \begin{array}{rcl}
   0 & = &  2\gamma \mathbf{X}-\delta \mathbf{W} \\
   0 & = & \gamma \big(\gamma\lambda-\delta\kappa\big) \big( \gamma \zeta - \delta \varepsilon \big) \mathbf{W}^3  - \big(3 \alpha \gamma ^2 \delta + 2 \beta \gamma ^3 - \delta^3 \big) \mathbf{Z} \mathbf{W}^2 
- 2 \gamma^3 \mathbf{Y}^2 \mathbf{Z},
  \end{array} \right. \\[0.5em]
R_{3}\colon&\left\lbrace 
  \begin{array}{rcl}
   0 & = &  2\eta \mathbf{X} - \iota \mathbf{W}  \\
   0 & = & \eta \big(\eta\lambda-\iota\kappa) \big( \varepsilon \iota - \zeta \eta \big) \mathbf{W}^3 
   -   \big( \iota^3 - 3 \alpha \eta^2\iota - 2 \beta \eta^3 \big) \mathbf{Z} \mathbf{W}^2 
+ 2 \eta^3 \mathbf{Y}^2 \mathbf{Z},
  \end{array} \right.
\end{split}  
\eeqn
and 
\beqn
\begin{split}
R_{4}\colon &\left\lbrace 
  \begin{array}{rcl}
   0 & = & 2 \varepsilon \mathbf{X} - \zeta \mathbf{W} + \gamma \varepsilon \eta \mathbf{Z}\\
   0 & = & (  \delta \zeta \iota -  2\beta \gamma \varepsilon \eta + \gamma^2 \varepsilon^2 \eta^2 \lambda)  \mathbf{W}^2 
   + 4 (\delta \varepsilon \eta + \gamma \zeta \eta + \gamma \varepsilon\iota)  \mathbf{X}^2 - 2 \gamma \varepsilon \eta  \mathbf{Y}^2 \\
   &-& 2 (3 \alpha \gamma \varepsilon\eta + \gamma \zeta\iota +  \delta \varepsilon\iota + \delta \zeta \eta
   + \gamma^2\varepsilon^2\eta^2\kappa)  \mathbf{X} \mathbf{W},
  \end{array} \right. \\[0.5em]
R_5\colon &\left\lbrace 
  \begin{array}{rcl}
   0 & = & \gamma \varepsilon \eta\lambda  \mathbf{W}^3 - \delta \iota \mathbf{W}^2 \mathbf{Z} -2\gamma\varepsilon\eta\kappa \mathbf{W}^2 \mathbf{X} + 2 (\gamma\iota+\delta\eta) \mathbf{X}\mathbf{Z}\mathbf{W} - 4 \gamma \eta \mathbf{X}^2 \mathbf{Z}\\
   0 & = & (  \delta \zeta \iota -  2\beta \gamma \varepsilon \eta + \gamma^2 \varepsilon^2 \eta^2 \lambda)  \mathbf{W}^2 
   + 4 (\delta \varepsilon \eta + \gamma \zeta \eta + \gamma \varepsilon\iota)  \mathbf{X}^2 - 2 \gamma \varepsilon \eta  \mathbf{Y}^2 \\
   &-& 2 (3 \alpha \gamma \varepsilon\eta + \gamma \zeta\iota +  \delta \varepsilon\iota + \delta \zeta \eta
   + \gamma^2\varepsilon^2\eta^2\kappa)  \mathbf{X} \mathbf{W},
  \end{array} \right. 
\end{split}  
\eeqn
and
\beqn
\begin{split}
R_{6}\colon&\left\lbrace 
  \begin{array}{rcl}
   0 & = & 2 \kappa \mathbf{X} - \lambda \mathbf{W} + \gamma \kappa \eta \mathbf{Z}\\
   0 & = & (  \delta \iota \lambda -  2\beta \gamma \eta \kappa + \gamma^2 \zeta \eta^2 \kappa^2 )  \mathbf{W}^2 
    +4 (\delta \eta \kappa  + \gamma \eta \lambda + \gamma \iota \kappa)  \mathbf{X}^2- 2 \gamma \eta \kappa \mathbf{Y}^2 \\
   &- &2 (3 \alpha \gamma \eta \kappa  + \gamma \iota \lambda +  \delta\iota \kappa + \delta \eta \lambda
   + \gamma^2\varepsilon\eta^2\kappa^2)  \mathbf{X} \mathbf{W},
  \end{array} \right. \\[0.5em]
R_{7}\colon &\left\lbrace 
  \begin{array}{rcl}
   0 & = & \gamma \zeta \eta \kappa  \mathbf{W}^3 - \delta \iota \mathbf{W}^2 \mathbf{Z} -2\gamma\varepsilon\eta\kappa \mathbf{W}^2 \mathbf{X} + 2 (\gamma\iota+\delta\eta) \mathbf{X}\mathbf{Z}\mathbf{W} - 4 \gamma \eta \mathbf{X}^2 \mathbf{Z} \\
   0 & = & (  \delta \iota \lambda -  2\beta \gamma  \eta \kappa + \gamma^2 \zeta \eta^2 \kappa^2 )  \mathbf{W}^2 
    +4 (\delta \eta \kappa  + \gamma \eta \lambda + \gamma \iota \kappa)  \mathbf{X}^2- 2 \gamma \eta \kappa \mathbf{Y}^2 \\
   &- &2 (3 \alpha \gamma \eta \kappa  + \gamma \iota \lambda +  \delta \iota \kappa+ \delta \eta \lambda
   + \gamma^2\varepsilon\eta^2\kappa^2)  \mathbf{X} \mathbf{W},
  \end{array} \right. 
\end{split}  
\eeqn
and 
\beqn
\begin{split}
R_{8}\colon &\left\lbrace 
  \begin{array}{rcl}
   0 & = &  2\kappa \mathbf{X} - \lambda \mathbf{W}  \\
   0 & = &\big(3 \alpha \kappa^2 \lambda + 2 \beta \kappa^3 - \lambda^3 \big) \mathbf{W}^2 -  \kappa \big(\eta\lambda-\iota\kappa\big) \big( \gamma \lambda-\delta \kappa \big) \mathbf{Z} \mathbf{W} 
   + 2 \kappa^3 \mathbf{Y}^2,
  \end{array} \right.
\end{split}  
\eeqn
and 
\beqn
\begin{split}
R_{9}\colon &\left\lbrace 
  \begin{array}{rcl}
   0 & = &  \zeta \lambda \mathbf{W}^2  - \delta \varepsilon\eta\kappa \mathbf{W} \mathbf{Z} - 2(\varepsilon\lambda+\zeta\kappa) \mathbf{W} \mathbf{X} + 2 \gamma\varepsilon\eta\kappa  \mathbf{X} \mathbf{Z}
   + 4 \varepsilon \kappa  \mathbf{X}^2  \\
    0 & = & (  \zeta \iota \lambda -  2\beta \varepsilon  \eta \kappa + \delta \varepsilon^2  \eta^2 \kappa^2 )  \mathbf{W}^2 
    + 4 (\zeta \eta \kappa  + \varepsilon \eta \lambda + \varepsilon \iota \kappa)  \mathbf{X}^2- 2 \varepsilon \eta \kappa \mathbf{Y}^2 \\
   &-&2 (3 \alpha \varepsilon \eta \kappa  + \varepsilon \iota \lambda +  \zeta \iota \kappa+ \zeta \eta \lambda
   + \gamma\varepsilon^2\eta^2\kappa^2)  \mathbf{X} \mathbf{W},  \end{array} \right. \\[0.5em]
R_{10}\colon &\left\lbrace 
  \begin{array}{rcl}
   0 & = &  \varepsilon \eta\kappa \mathbf{W}^2 - \iota \mathbf{W} \mathbf{Z} + 2 \eta \mathbf{X} \mathbf{Z}   \\
    0 & = & (  \zeta \iota \lambda -  2\beta \varepsilon  \eta \kappa + \delta \varepsilon^2  \eta^2 \kappa^2 )  \mathbf{W}^2 
    + 4 (\zeta \eta \kappa  + \varepsilon \eta \lambda + \varepsilon \iota \kappa)  \mathbf{X}^2- 2 \varepsilon \eta \kappa \mathbf{Y}^2 \\
   &-&2 (3 \alpha \varepsilon \eta \kappa  + \varepsilon \iota \lambda +  \zeta \iota \kappa+ \zeta \eta \lambda
   + \gamma\varepsilon^2\eta^2\kappa^2)  \mathbf{X} \mathbf{W},  \end{array} \right. 
\end{split}  
\eeqn
and
\beqn
\begin{split}
R_{11}\colon &\left\lbrace 
  \begin{array}{rcl}
   0 & = &  \zeta \lambda \mathbf{W}^2  - \gamma \varepsilon\iota\kappa \mathbf{W} \mathbf{Z} - 2(\varepsilon\lambda+\zeta\kappa) \mathbf{W} \mathbf{X} + 2 \gamma\varepsilon\eta\kappa  \mathbf{X} \mathbf{Z}
   + 4 \varepsilon \kappa  \mathbf{X}^2  \\
    0 & = & (  \delta \zeta \lambda -  2\beta \gamma \varepsilon  \kappa + \gamma^2 \varepsilon^2  \iota \kappa^2 )  \mathbf{W}^2 
    + 4 (\gamma \zeta \kappa  + \gamma \varepsilon \lambda + \delta \varepsilon \kappa)  \mathbf{X}^2- 2 \gamma \varepsilon \kappa \mathbf{Y}^2 \\
   &-&2 (3 \alpha \gamma \varepsilon \kappa  + \delta \varepsilon \lambda +  \delta \zeta \kappa+ \gamma \zeta \lambda
   + \gamma^2\varepsilon^2\eta\kappa^2)  \mathbf{X} \mathbf{W}.  \end{array} \right. \\[0.5em]
R_{12}\colon &\left\lbrace 
  \begin{array}{rcl}
0 & = &  \varepsilon \eta\kappa \mathbf{W}^2 - \delta \mathbf{W} \mathbf{Z} + 2 \gamma \mathbf{X} \mathbf{Z}   \\
    0 & = & (  \delta \zeta \lambda -  2\beta \gamma \varepsilon  \kappa + \gamma^2 \varepsilon^2  \iota \kappa^2 )  \mathbf{W}^2 
    + 4 (\gamma \zeta \kappa  + \gamma \varepsilon \lambda + \delta \varepsilon \kappa)  \mathbf{X}^2- 2 \gamma \varepsilon \kappa \mathbf{Y}^2 \\
   &-&2 (3 \alpha \gamma \varepsilon \kappa  + \delta \varepsilon \lambda +  \delta \zeta \kappa+ \gamma \zeta \lambda
   + \gamma^2\varepsilon^2\eta\kappa^2)  \mathbf{X} \mathbf{W}.  \end{array} \right. 
\end{split}  
\eeqn
\par We also remind the reader that lines $L_1 $, $ L_2 $, $L_3$, $L_4$, $L_5$ on the quartic surface were defined in Equation~(\ref{eqn:lines}). When resolving the quartic surface~(\ref{quartic1}),  the above curves lift to smooth rational curves on $\mathcal{X}(\alpha, \beta, \gamma, \delta , \varepsilon, \zeta, \eta, \iota, \kappa, \lambda)$, which by a slight  abuse of notation we shall denote by the same symbol.  One easily verifies that for general parameters the singularity at $\mathrm{P}_1$ is a rational double point of type $A_7$, and  $\mathrm{P}_2$ is of type $A_3$. The two sets $\{ a_1, a_2, \dots, a_7 \}$ and $\{b_1, b_2, b_3\}$ will denote the curves appearing from resolving the rational double  point singularities at $\mathrm{P}_1$ and $\mathrm{P}_2$, respectively. 
\begin{figure}
\hspace*{-0.5cm}
\begin{subfigure}{.5\textwidth}  
  	\centering
  	\scalebox{\MyScalePicLarge}{%
    		\begin{tikzpicture}
       		\input{pic14_1.tex}
    		\end{tikzpicture}}
  	\caption{with double lines and simple lines}
  	\label{fig:pic14a}
\end{subfigure}%
\hspace*{0.3cm}
\begin{subfigure}{.5\textwidth}  
  	\centering
  	\scalebox{\MyScalePicLarge}{%
    		\begin{tikzpicture}
       		\input{pic14_2.tex}
    		\end{tikzpicture}}
  	\caption{with 6-fold lines (thick), 4-fold lines (thin)}
  	\label{fig:pic14b}
\end{subfigure}%
\caption{Rational curves on $\mathcal{X}$ with $\mathrm{NS}(\mathcal{X})=P_{14}$}
\label{fig:pic14}
\end{figure}
We have the following:
\begin{theorem}
\label{thm:polarization14}
Assuming Equation~(\ref{eqn:general_params}), for a K3 surface $\mathcal{X}$ with N\'eron-Severi lattice $P_{14}$ in Theorem~\ref{thm1} the dual graph of all smooth rational curves is given by Figure~\ref{fig:pic14} where single and double edges are shown in Figure~\ref{fig:pic14a} and six-fold and four-fold edges are shown in Figure~\ref{fig:pic14b}.
\end{theorem}
\begin{remark}
The embeddings of the reducible fibers for each elliptic fibration in Theorem~\ref{thm1} into the graph given by Figure~\ref{fig:pic14} will be constructed in Sections~\ref{ssec:alt14}-\ref{ssec:max14} for Picard number 14, in Sections~\ref{ssec:alt}-\ref{ssec:max} for Picard number 15.
\end{remark}
\begin{proof}
Assuming Equation~(\ref{eqn:general_params}), the above equations determine projective curves $R_1, R_4, R_6, R_8$, and $R_2, R_3$ of degrees two and three, respectively.  The conic $R_1$ is a smooth rational curve tangent to $L_1$ at $\mathrm{P}_2$. The cubics $R_2$, $R_3$ pass through the points $\mathrm{P}_1, \mathrm{P}_2$. The cubic $R_2$ has a double point at $\mathrm{P}_2$, passes through $\mathrm{P}_1$ and is irreducible. For the pairs of curves $\{R_4, R_5\}$, $\{R_6, R_7\}$, $\{R_9, R_{10}\}$, $\{R_{11}, R_{12}\}$, their respective second equations coincide and determine $\mathbf{Y}^2$. Thus, six intersection points of $R_4$ and $R_5$ are given by the solutions of
\beq
\begin{split}
 \big( \delta \zeta \iota - \gamma^2 \varepsilon^2 \eta^2 \lambda \big) \mathbf{W}^3 
 -2 \big( \gamma \zeta \iota + \delta \varepsilon \iota + \delta\zeta \eta - 2 \gamma^2 \varepsilon^2\eta^2\kappa \big) \mathbf{X}\mathbf{W}^2 \\
 + 4 \big( \gamma \varepsilon \iota + \gamma \zeta \eta + \delta \varepsilon \eta \big) \mathbf{X}^2\mathbf{W} - 8 \gamma \varepsilon \eta \mathbf{X}^3 =0 \,, \quad
\end{split} 
\eeq 
and $2 \varepsilon \mathbf{X} - \zeta \mathbf{W} - \gamma \varepsilon \eta \mathbf{Z} =0$ and $\mathbf{Y}^2 = \dots$. An analogous computation allows to compute the six intersection points of $\{ R_6, R_7\}$, $\{ R_9, R_{10}\}$, $\{ R_{11}, R_{12}\}$.  Similarly, one shows that each pair out of $\{R_4, R_7\}$,   $\{R_4, R_{10}\}$, $\{R_4, R_{12}\}$,  $\{R_5, R_6\}$, $\{R_5, R_9\}$,  $\{R_5, R_{11}\}$, $\{R_6, R_{10}\}$,  $\{R_6, R_{12}\}$, $\{R_7, R_9\}$, $\{R_7, R_{11}\}$, $\{R_9, R_{12}\}$, $\{R_{10}, R_{11}\}$, has four intersection points. These six-fold and four-fold edges are shown in Figure~\ref{fig:pic14b}.
 Using Equation~(\ref{quartic1}), one then derives by standard methods the intersection numbers of the rational curves on the minimal resolution. The analysis is simplified by taking  the various elliptic fibrations in Theorem~\ref{thm1} into account and determining what rational curves are contained in certain reducible fibers.  The results then determine Figure~\ref{fig:pic14} uniquely. A comparison of the divisor classes for the constructed rational curves and their intersection numbers with the lattice-theoretic results obtained in \cite{Roulleau22} shows that they constitute all rational curves. 
\end{proof}
\par We have the following:
\begin{proposition}
\label{prop:polarization14}
The polarization of a general K3 surface $\mathcal{X}(\alpha, \beta, \gamma, \delta , \varepsilon, \zeta, \eta, \iota, \kappa, \lambda)$ is given by the divisor
\beq
\label{linepolariz14}
\mathcal{H} = 2 L_2 + L_3 + L_5 + 3 a_1 + 4 a_2 + 5 a_3 + 4 a_4 + 3 a_5 + 2 a_6 + a_7  \,,
\eeq
with $\mathcal{H}^2=4$. 
\end{proposition}
\begin{proof}
Using the reducible fibers provided for each fibration in Sections~\ref{ssec:alt14}-\ref{ssec:max14}, there are several equivalent ways to express the smooth fiber class for a given fibration. In this way, we obtain the linear relations between the divisors $R_1, \dots, R_{12}$, $L_1, \dots, L_5$, and $a_1, \dots, a_7, b_1, b_2, b_3$. From these relations, we obtain the divisor classes of $R_1, \dots, R_{12}$ and $L_4$ as linear combinations with integer coefficients of the remaining classes.
\par Looking at the standard fibration in Figure~\ref{fig:pic14std1}, we observe that the nodes $a_6$ and $a_4$ are the extra nodes of the two extended Dynkin diagrams of $\widetilde{D}_6$.  It follows that $L_1, \dots,  L_5,$ $a_1, \dots, a_3, a_5, a_7$, $b_1, b_2, b_3,$ and the fiber class of the standard fibration form a basis in $\operatorname{NS}(\mathcal{X})$. Thus, the polarizing divisor can be written as a linear combination
\beq
\label{eqn:lin_comb14}
\mathcal{H} = f \, F_{\text{std}}  + \sum_{i=1}^5 l_i \, L_i +  \sum_{i=1}^3 \beta_i \, b_i  
+ \sum_{i=1}^3 \alpha_i  \, a_i + \alpha_5 \, a_5 + \alpha_7 \, a_7  \,.
\eeq
We use $\mathcal{H} \circ a_i = \mathcal{H} \circ b_j = 0$ for $i=1, \dots,7$ and $j=1, 2, 3$, and  $\mathcal{H} \circ L_k=1$ for $k=1, \dots, 5$. We obtain a linear system of equations for the coefficients in Equation~(\ref{eqn:lin_comb14}) whose unique solution is given by Equation~(\ref{linepolariz14}). We then check that $\mathcal{H}^2=4$ and  $\mathcal{H} \circ  F =3$, where $F$ is the smooth fiber class of any elliptic fibration that is obtained as the intersection  of the quartic $\mathcal{Q}$ with a line $L_i$ for $i=1, \dots, 5$; see Sections~\ref{ssec:alt14}-\ref{ssec:max14}.
\end{proof}
We now construct the embeddings of the reducible fibers  into Figure~\ref{fig:pic14} for each elliptic fibration of Theorem~\ref{thm1}:
\subsubsection{The alternate fibration}
\label{ssec:alt14}
\begin{figure}
	\centering
  	\scalebox{\MyScalePicBig}{%
    		\begin{tikzpicture}
       		\input{pic14alt.tex}
    		\end{tikzpicture}}
	\caption{The alternate fibration on $\mathcal{X}$}
	\label{fig:pic14alt}
\end{figure}
There is one way of embedding the corresponding reducible fibers of case~(1) in Theorem~\ref{thm1} into the graph given by Figure~\ref{fig:pic14}. The configuration is invariant when applying the Nikulin involution in Proposition~\ref{NikulinInvolution} and shown in Figure~\ref{fig:pic14alt}. We have
\beq
\begin{split}
{\color{blue}\widetilde{A}_1} =   \langle b_1, R_3  \rangle \,, \qquad 
{\color{green}\widetilde{A}_1}=   \langle  R_1,  L_3  \rangle \,, \qquad
{\color{magenta}\widetilde{A}_1} =   \langle b_3 ,  R_2  \rangle \,, \\
{\color{yellow}\widetilde{A}_1} =   \langle R_8, L_5  \rangle \,, \qquad
{\color{orange}\widetilde{D}_{8}}=  \langle  a_2, L_2, a_3,  a_4, a_5,  a_6, a_7, L_4, L_1 \rangle \,.
\end{split}
\eeq
Thus, the smooth fiber class is given by
\beq
\label{eqn:F_alt_14}
\begin{split}
 F_{\text{alt}} 
 & =   L_1 + L_2 +  L_4  + a_2+ 2 a_3 + 2 a_4 + 2 a_5 + 2 a_6 + 2 a_7  \\
 & =  R_1 + L_3 \ =  R_2 + b_3 \  =  R_3 + b_1 \ = R_8 + L_5\,,
\end{split} 
\eeq
and the classes of a section and 2-torsion section are ${\color{red} b_2}$ and  ${\color{red} a_1}$, respectively. Using the polarizing divisor $\mathcal{H}$ in Equation~(\ref{linepolariz14}), one checks that
\beq
 \mathcal{H} -  F_{\text{alt}} - L_1 \equiv  a_1 + \dots + a_7 + b_1 + 2 b_2 + b_3  \,.
\eeq
This is consistent with the fact that this fibration is obtained by intersecting the quartic $\mathcal{Q}$ with the pencil of planes $L_1(u, v)=0$ in Equation~(\ref{eqn:pencil_alternate}), invariant  under the action of the Nikulin involution in Proposition~\ref{NikulinInvolution};  in the graph the action is represented by a horizontal flip that also exchanges the two red nodes ${\color{red} b_2}$ and  ${\color{red} a_1}$ representing the section and the 2-torsion section.
\subsubsection{The standard fibration}
\label{ssec:std14}
There are two ways of embedding the corresponding reducible fibers of case~(2) in Theorem~\ref{thm1} into the graph given by Figure~\ref{fig:pic14}. They are depicted in Figure~\ref{fig:pic14std}. In the case of Figure~\ref{fig:pic14std1}, we have
\beq
{\color{green}\widetilde{D}_6}=   \langle L_3,   L_5, a_1,  a_2, a_3, L_2, a_4 \rangle \,, \qquad
{\color{blue}\widetilde{D}_6} =   \langle b_1, b_3, b_2, L_1, a_7,  L_4, a_6 \rangle \, .
\eeq
Thus, the smooth fiber class is given by
\beq
\label{eqn:F1_std_14}
\begin{split}
 F_{\text{std}} 
 & =  L_2 +L_3 + L_5 + 2 a_1 + 2 a_2 + 2 a_3 + a_4    \\
 & =  2 L_1 + L_4 + a_6 + 2 a_7 + b_1 + 2 b_2 + b_3  \,,
\end{split} 
\eeq
and the class of a section is ${\color{red} a_5}$.  Using the polarizing divisor $\mathcal{H}$ in Equation~(\ref{linepolariz14}), one checks that
\beq
 \mathcal{H} -  F_{\text{std}} - L_2 \equiv   a_1 + 2 a_2 + 3 a_3 + 3 a_4 + 3 a_5 + 2 a_6  + a_7\,.
\eeq
This is consistent with the fact that this fibration is obtained by intersecting the quartic $\mathcal{Q}$ with the pencil of planes $L_2(u, v)=0$ in Equation~(\ref{eqn:pencil_standard1}).
\begin{figure}
\begin{subfigure}{.5\textwidth}  
  	\centering
  	\scalebox{\MyScalePicBig}{%
    		\begin{tikzpicture}
       		\input{pic14std1.tex}
    		\end{tikzpicture}}
  	\caption{\phantom{A}}
  	\label{fig:pic14std1}
\end{subfigure}%
\begin{subfigure}{.5\textwidth}  
  	\centering
  	\scalebox{\MyScalePicBig}{%
    		\begin{tikzpicture}
       		\input{pic14std2.tex}
    		\end{tikzpicture}}
  	\caption{\phantom{B}}
  	\label{fig:pic14std2}
\end{subfigure}%
\caption{The standard fibration on $\mathcal{X}$}
\label{fig:pic14std}
\end{figure}
\par Applying the Nikulin involution in Proposition~\ref{NikulinInvolution}, we obtain the fiber configuration in Figure~\ref{fig:pic14std2} with
\beq
{\color{green}\widetilde{D}_6}=   \langle R_1,  R_8, b_2,  L_1, a_7, L_4, a_6 \rangle \,, \qquad
{\color{blue}\widetilde{D}_6} =   \langle R_2,   R_3,  a_1, a_2, a_3, L_2, a_4 \rangle \, .
\eeq
The smooth fiber class is now given by
\beq
\label{eqn:F2_std_14}
\begin{split}
 \check{F}_{\text{std}} & =  R_1+ R_8 + 2  L_1 + L_4  + 2 a_7 + a_6  + 2 b_2  \\
 & =   R_2 +  R_3 +  L_2  +  2 a_1 + 2 a_2 + 2 a_3 + a_4  \,,
\end{split} 
\eeq
 and the class of the section is ${\color{red} a_5}$. Using the polarizing divisor $\mathcal{H}$ in Equation~(\ref{linepolariz14}), one checks that
\beq
\begin{split}
&  3 \mathcal{H} -  \check{F}_{\text{std}} - 2L_1  - L_2  - L_3 - L_5 \\
\equiv  \; & 3 a_1 + 4 a_2 + 5 a_3 + 5 a_4 + 5 a_5 + 4 a_6 + 3 a_7 + 3 b_1 + 4 b_2 + 3 b_3\,.
\end{split}
\eeq
This is consistent with the fact that this fibration is also obtained by intersecting the quartic $\mathcal{Q}$ with the pencil of cubic surfaces $C_2(u, v)=0$ in Equation~(\ref{eqn:pencil_standard2}). One checks that $C_2(u, v)=0$ contains $L_1, L_2, L_3, L_5$, and is tangent to $L_1$. 
\subsubsection{The base-fiber dual fibration}
\label{ssec:bfd14}
There are several ways of embedding the corresponding reducible fibers of case~(3) in Theorem~\ref{thm1} into the graph given by Figure~\ref{fig:pic14}. They are depicted in Figure~\ref{fig:pic14bfd}. In the case of Figure~\ref{fig:pic14bfd1}, we have
\beq
\begin{split}
{\color{green}\widetilde{E}_7}=   \langle  L_5, a_1, a_2, a_3, \underline{L_2}, a_4, a_5, a_6  \rangle \,, \qquad
{\color{blue}\widetilde{D}_4} =   \langle R_1, b_1,  b_2 , b_3, L_1 \rangle \,, \quad
{\color{orange}\widetilde{A}_1} =   \langle R_4 ,  L_4 \rangle  \,.
\end{split}
\eeq
Thus, the smooth fiber class is given by
\beq
\label{eqn:F1_bfd_14}
\begin{split}
 F_{\text{bfd}} 
 & = 2 L_2 + L_5 + 2 a_1 + 3 a_2 + 4 a_3 +  3 a_4 + 2 a_5 + a_6 \\
 &=  R_1  + L_1 +  b_1 + 2 b_2 +  b_3  \ =  R_4 +  L_ 4\,,
\end{split} 
\eeq
and the class of a section is ${\color{red} a_7}$. Using the polarizing divisor $\mathcal{H}$ in Equation~(\ref{linepolariz14}), one checks that
\beq
 \mathcal{H} -  F_{\text{bfd}} - L_3 \equiv  a_1 + \dots + a_7\,.
\eeq
This is consistent with the fact that this fibration is obtained by intersecting the quartic $\mathcal{Q}$ with the pencil of planes $L_3(u, v)=0$ in Equation~(\ref{eqn:pencil_bfd1}).
\begin{figure}
\begin{subfigure}{.5\textwidth}  
  	\centering
  	\scalebox{\MyScalePicBig}{%
    		\begin{tikzpicture}
       		\input{pic14bfdL3_1.tex}
    		\end{tikzpicture}}
    	\caption{\phantom{A}}
    	\label{fig:pic14bfd1}
\end{subfigure}%
\begin{subfigure}{.5\textwidth}
  	\centering
  	\scalebox{\MyScalePicBig}{%
    		\begin{tikzpicture}
       		\input{pic14bfdL3_2.tex}
    		\end{tikzpicture}}
    	\caption{\phantom{B}}
    	\label{fig:pic14bfd2}
\end{subfigure}%
\caption{The base-fiber dual fibration on $\mathcal{X}$ (using $L_3$)}
\label{fig:pic14bfd}
\end{figure}
\par Applying the Nikulin involution in Proposition~\ref{NikulinInvolution}, we obtain the fiber configuration in Figure~\ref{fig:pic14bfd2} with
\beq
{\color{green}\widetilde{E}_7} 	= \langle R_8,  b_2,  L_1,  a_7,  \underline{L_4}, a_6, a_5, a_4  \rangle\,, \qquad
{\color{blue}\widetilde{D}_4}	= \langle R_2,  R_3,  a_1, L_3, a_2   \rangle\,, \quad
{\color{orange}\widetilde{A}_1}	= \langle R_5 , L_2 \rangle.
\eeq
The smooth fiber class is given by
\beq
\label{eqn:F2_bfd_14}
\begin{split}
 \check{F}_{\text{bfd}} &=  R_8 + 3 L_1 + 2 L_4 + a_4 + 2 a_5 + 3 a_6 + 4 a_7 + 2 b_2  \\
 & =  R_2 + R_3 + L_3 +2 a_1 + a_2 \ =  R_5 + L_2 \,,
\end{split} 
\eeq
and the class of the section is ${\color{red} a_3}$. Using the polarizing divisor $\mathcal{H}$ in Equation~(\ref{linepolariz14}), one checks that
\beq
\begin{split}
& 3 \mathcal{H} -  \check{F}_{\text{bfd}} - 2L_1  - 2L_2  - L_5 \\
 \equiv \; & 3 a_1 + 5 a_2 + 7 a_3 +6 a_4+ 5 a_5 + 4 a_6 + 3 a_7 + 3 b_1 + 4 b_2 + 3 b_3\,.
\end{split} 
\eeq
This is consistent with the fact that this fibration is also obtained by intersecting the quartic $\mathcal{Q}$ with the pencil of cubic surfaces $C_3(u, v)=0$ in Equation~(\ref{eqn:pencil_bfd2}). One checks that $C_3(u, v)=0$ contains $L_1, L_2, L_5$ and  is tangent to $L_1, L_2$.
\par As explained in  Section~\ref{sssec:bfd} a fibration with the same singular fibers, but different moduli is obtained by swapping the roles of the lines $L_3 \leftrightarrow L_5$. In the case of Figure~\ref{fig:pic14bfd1a}, we have
\beq
\begin{split}
{\color{green}\widetilde{E}_7}=   \langle  L_3, a_1, a_2, a_3, \underline{L_2}, a_4, a_5, a_6  \rangle \,, \qquad
{\color{blue}\widetilde{D}_4} =   \langle R_8, b_1,  b_2 , b_3, L_1 \rangle \,, \quad
{\color{orange}\widetilde{A}_1} =   \langle R_6 ,  L_4 \rangle  \,.
\end{split}
\eeq
Applying the Nikulin involution in Proposition~\ref{NikulinInvolution}, we obtain the fiber configuration in Figure~\ref{fig:pic14bfd2a} with
\beq
{\color{green}\widetilde{E}_7} 	= \langle R_1,  b_2,  L_1,  a_7,  \underline{L_4}, a_6, a_5, a_4  \rangle\,, \qquad
{\color{blue}\widetilde{D}_4}	= \langle R_2,  R_3,  a_1, L_5, a_2   \rangle\,, \quad
{\color{orange}\widetilde{A}_1}	= \langle R_7 , L_2 \rangle.
\eeq
\begin{figure}
\begin{subfigure}{.5\textwidth}  
  	\centering
  	\scalebox{\MyScalePicBig}{%
    		\begin{tikzpicture}
       		\input{pic14bfdL5_1.tex}
    		\end{tikzpicture}}
    	\caption{\phantom{A}}
   	\label{fig:pic14bfd1a}
\end{subfigure}%
\begin{subfigure}{.5\textwidth}
  	\centering
  	\scalebox{\MyScalePicBig}{%
    		\begin{tikzpicture}
       		\input{pic14bfdL5_2.tex}
    		\end{tikzpicture}}
    	\caption{\phantom{B}}
   	\label{fig:pic14bfd2a}
\end{subfigure}%
\caption{The base-fiber dual fibration on $\mathcal{X}$ (using $L_5$)}
\label{fig:pic14bfd_a}
\end{figure}
\subsubsection{The base-fiber dual fibration -- case~\texorpdfstring{$(3')$}{(3')}}
\label{ssec:bfd14b}
There are two ways of embedding the corresponding reducible fibers of case~(3$'$) in Theorem~\ref{thm1} into the graph given by Figure~\ref{fig:pic14}. They are depicted in Figure~\ref{fig:pic14bfd_b}. In the case of Figure~\ref{fig:pic14bfd1b}, we have
\beq
\begin{split}
{\color{green}\widetilde{E}_8}=   \langle a_1,  a_2,  \underline{L_2}, a_3,  a_4, a_5, a_6,  a_7, L_1 \rangle \,, \qquad
{\color{blue}\widetilde{A}_1} =   \langle R_4, R_8 \rangle \, , \\
{\color{orange}\widetilde{A}_1} =   \langle R_6 ,  R_1 \rangle \,, \qquad
{\color{magenta}\widetilde{A}_1} =   \langle R_9 ,  b_1 \rangle \,, \qquad
{\color{yellow}\widetilde{A}_1} =   \langle R_{11} ,  b_3 \rangle \,.
\end{split}
\eeq
Thus, the smooth fiber class is given by
\beq
\label{eqn:F1_bfd_14b}
\begin{split}
 F'_{\text{bfd}}
 & = L_1 +3 L_2 +  2 a_1 + 4 a_2 + 6 a_3 +  5 a_4 + 4 a_5 + 2 a_6 + 2 a_7\\
 &=  R_1 +  R_6 \ = R_4  + R_8  \ =   R_9 +  b_1 \ =  R_{11} +  b_3\,,
\end{split} 
\eeq
and the class of a section is ${\color{red} b_2}$.  Using the polarizing divisor $\mathcal{H}$ in Equation~(\ref{linepolariz14}), one checks that
\beq
2 \mathcal{H} -  F'_{\text{bfd}} - L_1 -  L_3 - L_4 - L_5\equiv  2 a_1 + \dots  + 2a_7 + b_1 + 2 b_2 + b_3 \,.
\eeq
This is consistent with the fact that this fibration is obtained by intersecting the quartic $\mathcal{Q}$ with the pencil $\tilde{C}_3(u, v)=0$ in Equation~(\ref{eqn:pencil_extra1}). One checks that $\tilde{C}_3(u, v)=0$ contains $L_1, L_3, L_4, L_5$.
\begin{figure}
\begin{subfigure}{.5\textwidth}  
  	\centering
  	\scalebox{\MyScalePicBig}{%
    		\begin{tikzpicture}
       		\input{pic14bfdp1.tex}
    		\end{tikzpicture}}
    	\caption{\phantom{A}}
   	\label{fig:pic14bfd1b}
\end{subfigure}%
\begin{subfigure}{.5\textwidth}
  	\centering
  	\scalebox{\MyScalePicBig}{%
    		\begin{tikzpicture}
       		\input{pic14bfdp2.tex}
    		\end{tikzpicture}}
    	\caption{\phantom{B}}
   	\label{fig:pic14bfd2b}
\end{subfigure}%
\caption{The base-fiber dual fibration on $\mathcal{X}$ -- case~(3$'$) }
\label{fig:pic14bfd_b}
\end{figure}
\par Applying the Nikulin involution in Proposition~\ref{NikulinInvolution}, we obtain the fiber configuration in Figure~\ref{fig:pic14bfd2b} with
\beq
\begin{split}
{\color{green}\widetilde{E}_8} 		= \langle b_2,   L_1,  \underline{L_4}, a_7,  a_6,  a_5, a_4,   a_3,   a_2  \rangle \, , \qquad
{\color{blue}\widetilde{A}_1}		= \langle R_5,  L_5 \rangle \,, \\
{\color{orange}\widetilde{A}_1}		= \langle R_7 , L_3 \rangle \,, \quad
{\color{magenta}\widetilde{A}_1} =   \langle R_{12} ,  R_3 \rangle \,, \qquad
{\color{yellow}\widetilde{A}_1} =   \langle R_{10},  R_2 \rangle \,.
\end{split}
\eeq
The smooth fiber class is given by
\beq
\label{eqn:F2_bfdb}
\begin{split}
 \check{F}'_{\text{bfd}} &= 4 L_1 + 3 L_4 + a_4 + 2 a_5 + 3 a_6 + 4 a_7 + 5 a_8 + 6 a_9 +  2 b_2 \\
 &  =  R_2 + R_{10}  \ =  R_3 + R_{12} \ =  R_5 + L_5 \ =  R_7 + L_3  \,,
\end{split} 
\eeq
and the class of a section is ${\color{red} a_1}$. Using the polarizing divisor $\mathcal{H}$ in Equation~(\ref{linepolariz14}), one checks that
\beq
3 \mathcal{H} -  \check{F}'_{\text{bfd}} - 2 L_1 -  3 L_2  \equiv 3 a_1 + 5 a_2 + 7 a_3 + 6 a_4 + 5 a_5 + 4 a_6 + 3 a_7  +  3 b_1 + 4 b_2 + 3 b_3 \,.
\eeq
This is consistent with the fact that this fibration is also obtained by intersecting the quartic $\mathcal{Q}$ with the pencil $C'_3(u, v)=0$ in Equation~(\ref{eqn:pencil_extra2}). One checks that $C'_3(u, v)=0$ contains $L_1, L_2$ and is also tangent to $L_1, L_2$.
\subsubsection{The maximal fibration}
\label{ssec:max14}
There are two ways of embedding the corresponding reducible fibers of case~(4) in Theorem~\ref{thm1} into the graph given by Figure~\ref{fig:pic14}. They are depicted in Figure~\ref{fig:pic14max}.  In the case of Figure~\ref{fig:pic14max1}, we have
\beq
{\color{orange}\widetilde{D}_{10}} =   \langle b_1, b_3, b_2,  L_1,  a_7, a_6, a_5, a_4, a_3, L_2, a_2 \rangle  \,, \qquad
{\color{blue}\widetilde{A}_{1}} =   \langle  R_4 , L_3 \rangle \,, \quad
{\color{green}\widetilde{A}_{1}} =   \langle  R_6 , L_5 \rangle\,.
\eeq
Thus, the smooth fiber class is given by
\beq
\label{eqn:F1_max_14}
\begin{split}
 F_{\text{max}}  & =  2 L_1 + L_2 + a_2 + 2 a_3 + 2 a_4 + 2 a_5 + 2 a_6 + 2 a_7 + b_1 +2 b_2 + b_3  \\
 & = R_4 +  L_3 \ = R_6 +  L_5\,,
\end{split} 
\eeq
and the class of a section is ${\color{red} a_1}$. Using the polarizing divisor $\mathcal{H}$ in Equation~(\ref{linepolariz14}), one checks that
\beq
  \mathcal{H} -  F_{\text{max}} - L_4 \equiv  a_1 + a_2 + a_3 + a_4 + a_5 + a_6 + a_7    \,.
\eeq
This is consistent with the fact that this fibration is obtained by intersecting the quartic $\mathcal{Q}$ with the pencil of planes $L_4(u, v)=0$ in Equation~(\ref{eqn:pencil_max1}).
\begin{figure}
\begin{subfigure}{.5\textwidth}  
  	\centering
  	\scalebox{\MyScalePicBig}{%
    		\begin{tikzpicture}
       		\input{pic14max1.tex}
    		\end{tikzpicture}}
    	\caption{\phantom{A}}
   	\label{fig:pic14max1}
\end{subfigure}%
\begin{subfigure}{.5\textwidth}
	\centering
  	\scalebox{\MyScalePicBig}{%
    		\begin{tikzpicture}
       		\input{pic14max2.tex}
    		\end{tikzpicture}}
    	\caption{\phantom{B}}
   	\label{fig:pic14max2}
\end{subfigure}%
\caption{The maximal fibration on $\mathcal{X}$}
\label{fig:pic14max}
\end{figure}
\par Applying the Nikulin involution in Proposition~\ref{NikulinInvolution}, we obtain the fiber configuration in Figure~\ref{fig:pic14max2} with
\beq
{\color{orange}\widetilde{D}_{10}} =   \langle R_2,  R_3, a_1, a_2, a_ 3, a_4, a_5, a_6 ,  a_7, L_4, L_1  \rangle \,, \quad
{\color{blue}\widetilde{A}_{1}} =   \langle  R_5 , R_1 \rangle \,, \quad 
{\color{green}\widetilde{A}_{1}} =   \langle  R_7 , R_8 \rangle \,.
\eeq
The smooth fiber class is given by
\beq
\label{eqn:F2_max_14}
\begin{split}
 \check{F}_{\text{max}}
 & =  R_2 + R_3 +  L_1 + L_4 + 2 a_1 + 2 a_2 + 2 a_ 3 + 2 a_5 + 2 a_6 + 2 a_7   \\ 
 & =   R_1 + R_5 \ =  R_7 + R_8\,,
\end{split} 
\eeq
and the class of the section is ${\color{red} b_2}$.  Using the polarizing divisor $\mathcal{H}$ in Equation~(\ref{linepolariz14}), one checks that
\beq
\begin{split}
&  4 \mathcal{H} -   \check{F}_{\text{max}}  - 3L_1  - 3L_2  - L_3 - L_5  \\
\equiv \; & 4 a_1 + 6 a_2 + 8 a_3 +7 a_4 + 6 a_5 + 5 a_6 + 4 a_7 + 4 b_1 + 6 b_2 + 4 b_3 \,.
\end{split}
\eeq
This is consistent with the fact that this fibration is also obtained by intersecting the quartic $\mathcal{Q}$ with the pencil of quartic surfaces $C_4(u, v)=0$ in Equation~(\ref{eqn:pencil_max2}). One checks that $C_4(u, v)=0$ contains $L_1, L_2, L_3, L_5$, is tangent to $L_1, L_2$, and has also a vanishing Hessian along $L_1$.
\subsection{The dual graph for \texorpdfstring{$P'_{14}$}{P'}-polarized K3 surfaces}
Next, we will construct the dual graph of smooth rational curves for the K3 surfaces $\mathcal{X}'$ in Theorem~\ref{thm_parity} with N\'eron-Severi lattice $P'_{14}$. The graph can be constructed by the tools developed in Section~\ref{ssec:dual_graph_P}. We state the following result using the parameters in Equation~(\ref{eqn:CurlyJinvariants}):
\begin{figure}
	\centering
  	\scalebox{\MyScalePicLarge}{%
    		\begin{tikzpicture}
       		\input{pic14sd.tex}
    		\end{tikzpicture}}
  	\caption{Rational curves on $\mathcal{X}'$ with $\mathrm{NS}(\mathcal{X}')=P'_{14}$}
  	\label{fig:pic14sd}
\end{figure}
\begin{theorem}
\label{thm:polarization14sd}
Assuming Equation~(\ref{eqn:generic_parameters_PP}),  for a K3 surface $\mathcal{X}'$ with N\'eron-Severi lattice $P'_{14}$ in Theorem~\ref{thm_parity} the dual graph of all smooth rational curves is given by Figure~\ref{fig:pic14sd}.
\end{theorem}
Analogous to Sections~\ref{ssec:alt14}-\ref{ssec:max14}, one can construct the embeddings of the reducible fibers for each elliptic fibration of Picard rank 14 in Theorem~\ref{thm_parity} into the graph given by Figure~\ref{fig:pic14sd}: for fibration~$(1)$ the graph is Figure~\ref{fig:pic14sd_fib1} where the green nodes represent the reducible fiber of type ${\color{orange}\widetilde{E}_7}$, the blue/yellow/magenta/orange/brown nodes represent the reducible fibers of type ${\color{blue}\widetilde{A}_{1}}$, and the red node represents the class of the section and the 2-torsion section. Notice that the diagram is invariant under the action of the Nikulin involution in Proposition~\ref{NikulinInvolution_SD}; in the graph the action is represented by a horizontal flip that also exchanges the two red nodes representing the section and the 2-torsion section. The same behavior occurred for the alternate fibration on the K3 surface $\mathcal{X}$ and was discussed in Section~\ref{ssec:alt14}. 
\par Similarly, for fibration~$(2)$ the graph is given by Figure~\ref{fig:pic14sd_fib2} where the green nodes represent the reducible fiber of type ${\color{green}\widetilde{D}_{8}}$, the blue nodes  represent the reducible fiber of type ${\color{blue}\widetilde{D}_{4}}$, and the red node represents the class of the section. As in Section~\ref{ssec:std14} we obtain a second embedding with the same singular fibers by applying the Nikulin involution in Proposition~\ref{NikulinInvolution_SD}; the graph for that second configuration with the same singular fibers is represented by a horizontal flip of the first one. The same behavior occurred for the standard  fibration on $\mathcal{X}$ and was discussed in Section~\ref{ssec:std14}. 
\begin{figure}
\begin{subfigure}{.5\textwidth}  
  	\centering
  	\scalebox{\MyScalePicMed}{%
    		\begin{tikzpicture}
       		\input{pic14sd_fib1.tex}
    		\end{tikzpicture}}
  	\caption{\phantom{A}}
  	\label{fig:pic14sd_fib1}
\end{subfigure}%
\begin{subfigure}{.5\textwidth}
  	\centering
  	\scalebox{\MyScalePicMed}{%
    		\begin{tikzpicture}
       		\input{pic14sd_fib2.tex}
    		\end{tikzpicture}}
  	\caption{\phantom{B}}
  	\label{fig:pic14sd_fib2}
\end{subfigure}%
\caption{The two fibrations on $\mathcal{X}'$}
\label{fig:pic14sd_fib}
\end{figure}
\subsection{The dual graph for  \texorpdfstring{$P''_{14}$}{P''}-polarized K3 surfaces}
Finally, we will comment on the dual graph of smooth rational curves for the K3 surfaces $\mathcal{X}''$ in Theorem~\ref{thm_Vinberg} with N\'eron-Severi lattice $P''_{14}$. The graph was already determined in \cite{MR2718942}*{Table~2}, and we recall the following:
\begin{figure}
	\centering
  	\scalebox{\MyScalePicLarge}{%
    		\begin{tikzpicture}
       		\input{pic14vin.tex}
    		\end{tikzpicture}}
  	\caption{Rational curves on $\mathcal{X}''$ with $\mathrm{NS}(\mathcal{X}'')=P''_{13}$}
  	\label{fig:pic14vin}
\end{figure}
\begin{theorem}[Vinberg]
\label{thm:polarization14vin}
Assume that $( f_{1,3}, f_{2,3}, f_{3,3}, g_1, g_3 ) \neq 0$. For a K3 surface $\mathcal{X}''$ in Theorem~\ref{thm_Vinberg} with N\'eron-Severi lattice $P''_{14}$ the dual graph of all smooth rational curves is given by Figure~\ref{fig:pic14vin}.
\end{theorem}
It is easy to construct embeddings of the reducible fibers for each elliptic fibration of Picard rank 14 in Theorem~\ref{thm_Vinberg} into the graph given by Figure~\ref{fig:pic14vin}: for fibration~$(1)$ the graph is Figure~\ref{fig:pic14vin_fib1} where the green nodes represent the reducible fiber of type ${\color{green}\widetilde{E}_8}$, the blue nodes represent the reducible fiber of type ${\color{blue}\widetilde{D}_{4}}$, and the red node represents the class of the section. Similarly, for fibration~$(2)$ the graph is given by Figure~\ref{fig:pic14vin_fib2} where the green nodes represent the reducible fiber of type ${\color{green}\widetilde{D}_{12}}$ and the red node represents the class of the section.
\begin{figure}
\begin{subfigure}{.5\textwidth}  
  	\centering
  	\scalebox{\MyScalePicSmall}{%
    		\begin{tikzpicture}
       		\input{pic14vin_fib1.tex}
    		\end{tikzpicture}}
  	\caption{\phantom{A}}
  	\label{fig:pic14vin_fib1}
\end{subfigure}%
\begin{subfigure}{.5\textwidth}
  	\centering
  	\scalebox{\MyScalePicSmall}{%
    		\begin{tikzpicture}
       		\input{pic14vin_fib2.tex}
    		\end{tikzpicture}}
  	\caption{\phantom{B}}
  	\label{fig:pic14vin_fib2}
\end{subfigure}%
\caption{The two fibrations on $\mathcal{X}''$}
\label{fig:pic14vin_fib}
\end{figure}
\section{The corresponding double sextic K3 surfaces}
\label{sec:Kummer}
In this section we discuss the K3 surfaces $\mathcal{Y}$, obtained from the family of K3 surfaces $\mathcal{X}$ polarized by $P_{14}$ in Section~\ref{ssec:moduli_space}. We start by constructing double sextic surfaces branched on three lines coincident in a point and a cubic.  We then identify them as the K3 surfaces associated with $\mathcal{X}$ under the van Geemen-Sarti-Nikulin duality. 
\subsection{Double covers of the projective plane}
Let $\bar{\mathcal{Y}}$ be the double cover of the projective plane $\mathbb{P}^2 =\mathbb{P}(Z_1, Z_2, Z_3)$ branched along the union of three lines $\ell_1, \ell_2, \ell_3$ coincident in a point and a cubic $\mathcal{C}$. We call such a configuration \emph{generic} if the cubic is smooth and meets the three lines in nine distinct points. In particular, the cubic does not meet the point of coincidence of the three lines. We construct a geometric model as follows: we use a suitable projective transformation to move the line $\ell_3$ to $\ell_3 = \mathrm{V}(Z_3)$. We then mark three distinct points $q_0$, $q_1$, and $q_\infty$ on $\ell_3$ and use a M\"obius transformation to move these points to $[Z_1 : Z_2 : Z_3] = [ 0 : 1 : 0]$, $[1:1:0]$, and $[1: 0: 0]$. Up to scaling, the three lines, coincident in $q_1$, are then given by
\beq
  \ell_1 = \mathrm{V}\big( Z_1 - Z_2  + \mu Z_3 \big) \,, \qquad  \ell_2 = \mathrm{V}\big( Z_1 - Z_2  + \nu Z_3 \big)  \,, \qquad  \ell_3 = \mathrm{V}\big(Z_3\big) \,,
\eeq  
for some parameters $\mu, \nu$ with $\mu \neq \nu$. Let the cubic $\mathcal{C} = \mathrm{V} ( C(Z_1, Z_2, Z_3) )$ intersect the line $\ell_3$ at $q_0$, $q_\infty$, and at the point $[-d_2:c_1:0] \neq [1:1:0]$. Thus, we have 
\beq
\label{eqn:cubic}
 C = e_3 Z_3^3 + \Big( d_0 Z_1 + e_1 Z_2 \Big) Z_3^2 +  \Big( c_0 Z_1^2 + d_1 Z_1 Z_2 + e_2 Z_2^2 \Big) Z_3 + Z_1 Z_2 \Big(c_1 Z_1 + d_2 Z_2\Big)\,,
\eeq 
which can be written as
\beq
\label{eqn:cubic2}
 C=\Big( c_1 Z_2 + c_0 Z_3\Big) Z_1^2 + \Big( d_2 Z_2^2 +d_1 Z_2 Z_3 + d_0 Z_3^2\Big) Z_1 + \Big(e_2 Z_2^2 + e_1 Z_2 Z_3 + e_0 Z_3^2\Big) Z_3,
\eeq  
such that in $\mathbb{WP}_{(1,1,1,3)} =\mathbb{P}(Z_1, Z_2, Z_3, Y)$ the surface $\bar{\mathcal{Y}}$ is given by
\beq
\label{eqn:defining_eqn_3lines+cub_1st}
 Y^2  = \big( Z_1 - Z_2  + \mu Z_3 \big) \big( Z_1 - Z_2  + \nu Z_3 \big) Z_3 \; C(Z_1, Z_2,  Z_3) \,,
\eeq
for parameters $\mu, \nu, c_0, c_1, d_0, d_1, d_2, e_0, e_1, e_2$ such that $c_1 \neq 0, c_1 + d_2 \neq 0$, $\mu \neq \nu$, and $\mathcal{C}$ is a smooth cubic that intersects each line $\ell_1, \ell_2, \ell_3$ in three distinct points. We have the following:
\begin{lemma}
\label{lem:specialization_cubic}
The cubic $\mathcal{C}$ is tangent to the line $\ell_3$ at $q_0$ if and only if $d_2=0$ and the remaining parameters are general. The cubic $\mathcal{C}$ is singular at $q_0$ if and only if $d_2=e_2=0$ and the remaining parameters are general; see Figure~\ref{fig:pic_cub}.
\end{lemma}
\begin{figure}
\hspace*{-0.5cm}
\begin{subfigure}{.35\textwidth}  
  \beqn
  \scalemath{\MyScalePicHuge}{
  \begin{xy}
  	<0cm,0cm>;<1cm,0cm>:
	(-1.0,	-0.5);(+5.0,+3.5)**@{-},
	(+4.0,+3.5);(+4.0,-0.5)**@{-},
	(+4.3,+3.4);(+2.2,-0.5)**@{-},
	(-1.0,+0.08);(+4.5,+0.2)**\crv{(2.0, 0.0)&(3.6, 0.1)&(4.2,0.13)},
	(+4.5,+0.2);(+1.1,+1.2)**\crv{(4.6,0.25)&(4.7,0.4)&(4.5,0.6)&(3.5, 0.78)},
	(+1.1,+1.2);(+4.6,+2.0)**\crv{(0.9,1.25)&(0.85,1.3)&(0.8,1.45)&(3.0,1.8)},
	(+4.0,+2.8)*{\bullet},
	(+4.0,+1.9)*{\bullet},
	(+4.0,+0.655)*{\bullet},
	(-1.05,-0.64)*++!D\hbox{$\ell_1$},
	(+2.05,-0.64)*++!D\hbox{$\ell_2$},
	(+3.75,-0.64)*++!D\hbox{$\ell_3$},
	(+4.25,+2.25)*++!D\hbox{$q_1$},
	(+4.32,+1.31)*++!D\hbox{$q_\infty$},
	(+4.25,+0.505)*++!D\hbox{$q_0$},
	(+1.60,+1.45)*++!D\hbox{$\mathcal{C}$},
  \end{xy}}
  \eeqn
  \caption{generic}
  \label{fig:pic_cub1}
\end{subfigure}%
\hspace*{-0.5cm}
\begin{subfigure}{.35\textwidth}
  \beqn
  \scalemath{\MyScalePicHuge}{
  \begin{xy}
  	<0cm,0cm>;<1cm,0cm>:
	(-1.0,	-0.5);(+5.0,+3.5)**@{-},
	(+4.0,+3.5);(+4.0,-0.5)**@{-},
	(+4.3,+3.4);(+2.2,-0.5)**@{-},
	(-1.0,+0.08);(+3.8,+0.2)**\crv{(1.3, 0.0)&(2.9, 0.10)&(3.2,0.12)&(3.4,0.14)&(3.7,0.16)},
	(+3.8,+0.2);(+1.1,+1.2)**\crv{(3.9,0.25)&(4.05,0.4)&(3.8,0.6)&(2.8, 0.78)},
	(+1.1,+1.2);(+4.6,+2.0)**\crv{(0.9,1.25)&(0.85,1.3)&(0.8,1.45)&(3.0,1.8)},
	(+4.0,+2.8)*{\bullet},
	(+4.0,+1.9)*{\bullet},
	(+4.0,+0.37)*{\bullet},
	(-1.05,-0.64)*++!D\hbox{$\ell_1$},
	(+2.05,-0.64)*++!D\hbox{$\ell_2$},
	(+3.75,-0.64)*++!D\hbox{$\ell_3$},
	(+4.25,+2.25)*++!D\hbox{$q_1$},
	(+4.32,+1.31)*++!D\hbox{$q_\infty$},
	(+4.30,+0.10)*++!D\hbox{$q_0$},
	(+1.60,+1.45)*++!D\hbox{$\mathcal{C}$},
  \end{xy}}
  \eeqn
  \caption{$d_2= 0$}
  \label{fig:pic_cub2}
\end{subfigure}%
\hspace*{-0.5cm}
\begin{subfigure}{.35\textwidth}
  \beqn
  \scalemath{\MyScalePicHuge}{
  \begin{xy}
  	<0cm,0cm>;<1cm,0cm>:
	(-1.0,	-0.5);(+5.0,+3.5)**@{-},
	(+4.0,+3.5);(+4.0,-0.5)**@{-},
	(+4.3,+3.4);(+2.2,-0.5)**@{-},
	(-1.0,+0.08);(+3.8,+0.2)**\crv{(1.3, 0.0)&(2.9, 0.10)&(3.2,0.12)&(3.4,0.14)&(3.7,0.16)},
	(+4.0,+0.38);(+4.0,+0.36)**\crv{(4.1,0.48)&(4.2,0.55)&(4.45,0.35)&(4.15,0.19)},
	(+3.8,+0.2);(+1.1,+1.2)**\crv{(3.9,0.25)&(4.05,0.4)&(3.8,0.6)&(2.8, 0.78)},
	(+1.1,+1.2);(+4.6,+2.0)**\crv{(0.9,1.25)&(0.85,1.3)&(0.8,1.45)&(3.0,1.8)},
	(+4.0,+2.8)*{\bullet},
	(+4.0,+1.9)*{\bullet},
	(+4.0,+0.37)*{\bullet},
	(-1.05,-0.64)*++!D\hbox{$\ell_1$},
	(+2.05,-0.64)*++!D\hbox{$\ell_2$},
	(+3.75,-0.64)*++!D\hbox{$\ell_3$},
	(+4.25,+2.25)*++!D\hbox{$q_1$},
	(+4.32,+1.31)*++!D\hbox{$q_\infty$},
	(+4.25,+0.30)*++!D\hbox{$q_0$},
	(+1.60,+1.45)*++!D\hbox{$\mathcal{C}$},
  \end{xy}}
  \eeqn
 \caption{$d_2=e_2=0$}
 \label{fig:pic_cub3}
\end{subfigure}%
\caption{The different branch loci for the K3 surfaces $\mathcal{Y}$}
\label{fig:pic_cub}
\end{figure}
\noindent We also remark that the cubics  $\mathcal{C}$ and $\mathcal{C} + \Lambda \ell_1 \ell_2 \ell_3$ for $\Lambda \in \mathbb{C}$ have the same intersection points with the lines $\ell_1, \ell_2, \ell_3$. After a suitable shift of coordinates, the parameters of the cubic pencil $\mathcal{C}' = \mathcal{C} + \Lambda \ell_1 \ell_2 \ell_3$ and $\mathcal{C}$ are related by
\beq
\begin{array}{rlrlrlrl}
 c'_1 & = c_1, \quad & c'_0 & = c_0 + \Lambda, \quad & d'_2 & = d_2, \quad &  d'_1 & = d_1 - 2 \Lambda ,\\[0.2em]
d'_0 & = d_0+ (\mu+\nu) \Lambda, &  e'_2 & = e_2 + \Lambda , &e'_1 & = e_1- (\mu+\nu) \Lambda, &e'_0 & = e_0+ \mu \nu \Lambda.
\end{array} 
\eeq 
\par Returning to the cubic $\mathcal{C}$, using an overall rescaling we can assume $c_1=1$ in Equation~(\ref{eqn:cubic}). Next, we apply the transformation
\beq
 \Big( Z_1, \ Z_2, \ Z_3 \Big) \ \mapsto \  \Big( \tilde{Z}_1, \ \tilde{Z}_2, \ \tilde{Z}_3 \Big) = \Big( Z_1 - \frac{d_1}{2} Z_3 , \  Z_2 - \frac{d_1+ 2 \kappa}{2} Z_3 , \ Z_3 \Big) \,,
\eeq
and set
\beq
\label{eqn:params_transfo1}
\begin{split}
 \tilde{\mu} \,& =\mu -\frac{d_1}{2} + \kappa + \kappa d_2\,, \qquad \tilde{\nu} =\nu -\frac{d_1}{2} + \kappa + \kappa d_2\,, 
 \end{split}
\eeq 
and
\beq
\label{eqn:params_transfo2}
\begin{split}
 \tilde{d}_2  & = d_2 \,, \qquad \qquad \tilde{e}_2  = e_2 - \frac{1}{2} d_1 d_2 + \kappa d_2^2\,,\\
 \tilde{e}_1 &= e_1 - \frac{1}{4} d_1^2 + \kappa (d_1 d_2 - 2 e_2) -  \kappa^2 d_2^2 \,,\\
 \tilde{c}_0 & = c_0 - \kappa \,,\qquad \; \tilde{d}_0  = d_0 - c_0 d_1 + 2 \kappa c_0 d_2 - \kappa^2 d_2\,,\\
 \tilde{e}_0 &= e_0 - \frac{d_0 d_1}{2} + \frac{c_0 d_1^2}{4} + \frac{\kappa}{4} \big( d_1^2 + 4 d_0 d_2- 4 c_0 d_1 d_2-4 e_1 \big) + \frac{\kappa^2}{2} \big(2e_2 - d_1d_2 + 2c_0 d_2^2\big)\,.
\end{split}
\eeq 
This transformation leaves $\ell_3$ and $q_0$, $q_1$, and $q_\infty$ invariant, and we obtain
\beq
  \ell_1 = \mathrm{V}\big( \tilde{Z}_1 - \tilde{Z}_2  + \tilde{\mu} \tilde{Z}_3 \big) \,, \qquad  \ell_2 = \mathrm{V}\big( \tilde{Z}_1 - \tilde{Z}_2  + \tilde{\nu} \tilde{Z}_3 \big)  \,, 
  \qquad  \ell_3 = \mathrm{V}\big(\tilde{Z}_3\big) \,,
\eeq 
and
\beq
C  =  \Big( \tilde{Z}_2 + \tilde{c}_0  \tilde{Z}_3\Big) \tilde{Z}_1^2 +\Big(\tilde{d}_2  \tilde{Z}_2^2 + \tilde{d}_0  \tilde{Z}_3^2 \Big) \tilde{Z}_1 
+ \Big(\tilde{e}_2 \tilde{Z}_2^2 + \tilde{e}_1 \tilde{Z}_2 \tilde{Z}_3 + \tilde{e}_0 \tilde{Z}_3^2\Big) \tilde{Z}_3 \,.
\eeq
Since $\kappa$ is a free parameter, we can impose one additional relation for the configuration. A convenient choice turns out to be
\beq
\label{eqn:monic}
\tilde{c}_0 + \tilde{e}_2 =  \left(1+ \frac{\tilde{d}_2}{2}\right)(\tilde{\mu} + \tilde{\nu})   \,.
\eeq  
This choice is achieved by substituting
\beq
 \kappa = \frac{2(\mu+\nu)-(d_2+2)(c_0+e_2)}{(d_2+1)(d_2-2)(d_2+3)} + \frac{d_2(d_2^2+2d_2-4)}{2(d_2+1)(d_2-2)(d_2+3)} \,
\eeq 
into Equations~(\ref{eqn:params_transfo1}) and~(\ref{eqn:params_transfo2}). The only remaining projective action -- leaving the line $\ell_1$ and its marked points $q_0$, $q_1$, and $q_\infty$ invariant -- is generated by rescaling $Z_3$. Under the action $Z_3 \mapsto \Lambda Z_3$ with $\Lambda \in \mathbb{C}^{\times}$, parameters of equivalent configurations are related by
\beq
\label{eqn:scaling_coeffs}
 \Big( \tilde{d}_2, \ \tilde{\mu},  \ \tilde{c}_0, \ \tilde{e}_2, \ \tilde{d}_0, \ \tilde{e}_1,  \, \tilde{e}_0 \Big) \ \mapsto \  
 \Big( \tilde{d}_2, \ \Lambda \tilde{\mu}, \ \Lambda \tilde{c}_0, \ \Lambda \tilde{e}_2, \ \Lambda^2 \tilde{d}_0, \ \Lambda^2 \tilde{e}_1,  \, \Lambda^3 e_0 \Big) \,.
\eeq 
\par In the following, we will drop tildes, always assume $d_2 \neq -1$ (to assure that the cubic does not pass through $q_1=[1:1:0]$, i.e.,  the point of coincidence of the three lines) and assume that $\mu$ and $\nu$ are related by Equation~(\ref{eqn:monic}). These assumptions fix all degrees of freedom except the scaling in Equation~(\ref{eqn:scaling_coeffs}). We have proved the following:
\begin{lemma}
\label{lem:double_cover}
Let $\bar{\mathcal{Y}}$ be the double cover of the projective plane $\mathbb{P}^2 = \mathbb{P}(Z_1, Z_2, Z_3)$ branched on three lines coincident in a point and a general cubic. There are affine parameters $(d_2, \mu, c_0, e_2, d_0, e_1, e_0) \in \mathbb{C}^7$, unique up to the action given by
\beq
\label{eqn:rescaling}
 \Big(  d_2, \ \mu, \ c_0, \ e_2, \ d_0, \  e_1,  \, e_0 \Big) \ \mapsto \  
 \Big(  d_2, \ \Lambda \mu, \ \Lambda c_0, \ \Lambda e_2, \ \Lambda^2  d_0, \ \Lambda^2 e_1,  \, \Lambda^3 e_0 \Big) 
\eeq 
with $\Lambda \in \mathbb{C}^{\times}$, such that $\bar{\mathcal{Y}}$ in $\mathbb{WP}_{(1,1,1,3)} =\mathbb{P}(Z_1, Z_2, Z_3, Y)$ is obtained by
\beq
\label{eqn:defining_eqn_3lines+cub}
\begin{split}
 Y^2 & \,  = \big( Z_1 - Z_2  + \mu Z_3 \big) \big( Z_1 - Z_2  + \nu Z_3 \big) Z_3 \\
  & \times   \left( \Big( Z_2 + c_0 Z_3\Big) Z_1^2 + \Big( d_2 Z_2^2 + d_0 Z_3^2\Big) Z_1 + \Big(e_2 Z_2^2 + e_1 Z_2 Z_3 + e_0 Z_3^2\Big) Z_3 \right) \,,
\end{split} 
\eeq
with $\mu + \nu = (1+d_2/2)(c_0 + e_2)$ and $d_2 \neq -1$.
\end{lemma} 
\subsection{Elliptic fibrations and  moduli}
\label{ssec:clingher_trick}
We denote by $\mathcal{Y}$ the surface obtained as the minimal resolution of $\bar{\mathcal{Y}}$. Since $\mathcal{Y}$ is the resolution of a double sextic surface, it is a K3 surface.  We will now construct two Jacobian elliptic fibrations on $\mathcal{Y}$, one with trivial Mordell-Weil group, the other with Mordell-Weil group $\mathbb{Z}/2\mathbb{Z}$. We refer to these fibrations as the standard fibration and the alternate fibration.
\par The pencil of lines $(Z_1 - Z_2) - t Z_3=0$ for $t\in \mathbb{C}$ through the point $q_1=[1:1:0]$ induces the standard fibration on $\mathcal{Y}$. When substituting $Z_1 = X$, $Z_2=X-(c_1+d_2) (t+\mu)(t+\nu) t$, and $Z_3= (c_1+d_2)(t+\mu)(t+\nu)$ into Equation~(\ref{eqn:defining_eqn_3lines+cub_1st}) we obtain the Weierstrass model
\beq
\label{eqn:defining_fibration}
\begin{split}
 Y^2 & = X^3 - \big(t+\mu\big)\big(t+\nu\big) \Big( (c_1+2 d_2)t - (c_0 + d_1 + e_2)\Big) X^2\\
 &  + \big(c_1+d_2\big) \big(t+\mu\big)^2\big(t+\nu\big)^2 \Big(d_2 t^2 -(d_1 + 2 e_2) t + (d_0 +e_1)\Big) X \\
 & +\big(c_1+d_2\big)^2 \big(t+\mu\big)^3\big(t+\nu\big)^3 \Big( e_2 t^2 - e_1 t + e_0\Big) \,,
\end{split} 
\eeq 
with a discriminant function of the elliptic fibration $\Delta = (t+\mu)^6 (t+\nu)^6 (c_1+d_2)^2 p(t)$, where $p(t) = c_1^2 d_2^2 t^6 + \dots$ is a polynomial of degree six. We have the following:
\begin{lemma}
\label{lem:defining_eqn_3lines+cub}
A general K3 surface $\mathcal{Y}$  admits a Jacobian elliptic fibration with singular fibers $3 I_0^* + 6 I_1$ and trivial Mordell-Weil group. The  fibration has singular fibers $I_1^* + 2 I_0^* + 5 I_1$ if and only if $d_2=0$ and the remaining parameters are general. It has singular fibers $I_2^* + 2 I_0^* + 4 I_1$ if and only if $d_2=e_2=0$. 
\end{lemma}
\begin{proof}
Given the Weierstrass model in Equation~(\ref{eqn:defining_fibration}) the statement is checked by explicit computation.
\end{proof}
Similarly, the pencil of lines $Z_2 + t Z_3=0$ with $t \in \mathbb{C}$ through the point $q_\infty =[1:0:0]$ induces  the alternate fibration on $\mathcal{Y}$. In fact, when substituting $Z_1 = \nu + t + (\mu-\nu) Q_1/(Q_1 -X)$, $Z_2=t$, $Z_3=-1$ into Equation~(\ref{eqn:defining_eqn_3lines+cub}) one obtains the corresponding Weierstrass model. We have the following:
\begin{lemma}
\label{lem:alternate_eqn_3lines+cub}
A general K3 surface $\mathcal{Y}$  admits a Jacobian elliptic fibration with singular fibers $I_2^* + 6 I_2 + 4 I_1$ and Mordell-Weil group $\mathbb{Z}/2\mathbb{Z}$. The  fibration has singular fibers $I_3^* + 6 I_2 + 3 I_1$ if and only if $d_2=0$, and singular fibers $I_4^* + 6 I_2 + 2 I_1$ if and only if $d_2=e_2=0$,  and the remaining parameters are general. 
\end{lemma}
\begin{proof}
Given the Weierstrass model the statement is checked by explicit computation.
\end{proof}
\par It follows from Lemma~\ref{lem:defining_eqn_3lines+cub} that the K3 surface $\mathcal{Y}$ is polarized by the lattice $R_{14}$, and the polarizing lattice can extend to $R_{15}$ and $R_{16}$, respectively, where the lattices were defined in Equation~(\ref{eqn:Rlattices}). The complete classification of the elliptic fibrations for $R_{14}$-polarized K3 surfaces was given in  \cite{MR3201823}*{Sec.~4.7}; similarly, the classification for the elliptic fibrations on $R_{16}$-polarized K3 surfaces was given in \cite{MR2254405}.  The standard and alternate fibration above are contained in these classifications.
\par Equation~(\ref{eqn:polynomials0}) expresses the polynomials $A$ and $B$, that define the alternate fibrations on $\mathcal{X}$ and $\mathcal{Y}$ in Equation~(\ref{eqn:alt0}) and   Equation~(\ref{eqn:alt_SI}), respectively, in terms of the parameters of the Inose-type quartic $\mathcal{Q}$ in Equation~(\ref{quartic1}) as follows:
\beq
\label{eqn:APpolys}
\begin{split}
 A(t) & = t^3 + a_1 t + a_0 \ = \  t^3 - 3 \alpha t - 2 \beta \,, \\
 B(t) & = b_4 t^4 + b_3 t^3 + b_2 t^2 + b_1 t + b_0 \ = \  \big(\gamma t - \delta\big) \big(\varepsilon t - \zeta\big)  \big(\eta t - \iota\big)  \big( \kappa t -\lambda\big) \,.
\end{split}
\eeq 
On the other hand, Equation~(\ref{eqn:invariants_alternate_fibration}) expresses the coordinates of the moduli space $\mathscr{M}_P$ in Equation~(\ref{eqn:moduli_space_P}) in terms of the coefficients of the polynomials $A$ and $B$. Combing these results immediately yields Equations~(\ref{eqn:invariants_intro_1})-(\ref{eqn:invariants_intro_4}).  The permutations of the roots of $B(t)$ are generated by the symmetries in Lemma~\ref{symmetries1},(1)-(3). Finally, one checks that the action in Lemma~\ref{symmetries1},(4) corresponds to the $\mathbb{C}^\times$-action on the weighted projective space in which $\mathscr{M}_P$ sits.
\par We have the following:
\begin{theorem}
\label{thm:end1c}
The van Geemen-Sarti-Nikulin dual of a general $P_{14}$-polarized K3 surface $\mathcal{X}$ has a birational model as a double sextic branched on three lines coincident in a point and a general cubic.  Conversely, every K3 surface $\mathcal{Y}$ obtained as the minimal resolution of such double sextic is the van Geemen-Sarti-Nikulin dual of a $P_{14}$-polarized K3 surface $\mathcal{X}$. The cubic is tangent to one line if $J'_4=0$, and the cubic is singular at the intersection point with that line if $J'_4=J'_6=0$. 
\end{theorem}
\begin{proof}
Proposition~\ref{prop:fibrations_SI} shows that the van Geemen-Sarti-Nikulin dual surface $\mathcal{Y}$ of a general $P_{14}$-polarized K3 surface $\mathcal{X}$ has N\'eron-Severi lattice $R_{14}$. Because generically $\mathrm{NS}(\mathcal{Y}) = R_{14}$, the polarizing lattice is 2-elementary with invariants $(\rho, \ell, \delta) = (14, 6, 1)$ and there is a non-symplectic involution $\imath_\mathcal{Y}$ acting trivially $\mathrm{NS}(\mathcal{Y})$ \cites{MR544937,MR633160,MR633160b,MR752938,MR3165023}. The quotient $\mathcal{Y}/\imath_\mathcal{Y}$ is a rational surface and the branch locus consists of a the image of the fixed locus of the involution. Applying Nikulin's result  \cite{MR633160b}*{Thm.4.2.2}, we obtain that the fixed locus consists of a curve of genus $g=(22-\rho-\ell)/2=1$ and $k=(\rho-\ell)/2=4$. Thus, the general $R_{14}$-polarized K3 surface $\mathcal{Y}$ has a model as a double cover of a blow up of $\mathbb{P}^2$ branched on a genus-1 curve and 4 rational curves, which are necessarily 3 lines and an exceptional divisor of a singular point of the branch locus. This branch point has multiplicity more than 2, otherwise the exceptional divisor is not in the branch locus; see \cite{MR3201823}*{Sec.~3 and Prop.~3.1}.  
\par Conversely, given K3 surface $\mathcal{Y}$ obtained as the minimal resolution of the double cover branched branched over three distinct concurrent lines and a general cubic curve, Lemmas~\ref{lem:alternate_eqn_3lines+cub} and~\ref{lem:alternate_fibrations_extensions} and Proposition~\ref{prop:fibrations_SI} show that $\mathcal{Y}$ admits an alternate fibration whose van Geemen-Sarti-Nikulin dual surface $\mathcal{X}$ is a $P_{14}$-polarized K3 surface. 
\par Specialization of the branch locus in Lemma~\ref{lem:specialization_cubic}  are given by conditions $d_2=0$ and $d_2=e_2=0$, respectively. They match the specializations of the alternate fibration given in Lemma~\ref{lem:alternate_eqn_3lines+cub}. In turn, the same specializations of the alternate fibrations  occur in Lemma~\ref{lem:alternate_fibrations_extensions}, and the corresponding extensions of the lattice polarization are given in Proposition~\ref{prop:fibrations_SI}. Equation~(\ref{eqn:invariants_alternate_fibration}) then shows that conditions $d_2=0$ and $d_2=e_2=0$ are equivalent to $J'_4=0$ and $J'_4=J'_6=0$, respectively.
\end{proof}
Many details about an $R_{14}$-polarization and $R_{16}$-polarization have been known:  in  \cite{MR3201823}*{Sec.~4.7} it was shown that the K3 surface with 2-elementary N\'eron-Severi lattice $L$ with $\rho_L =14$, $\ell_L=6$ admits a model as double cover of $\mathbb{P}^2$ branched on three lines meeting in a point and a cubic. Similarly, the case of the $R_{16}$-polarized K3 surface is already known; see \cite{MR3201823}*{Example 3.7~(iii)}. Moreover, in \cite{MR3201823}*{Table on p.~521} it was shown that the quotient of an $R_{16}$-polarized K3 surface by a van Geemen-Sarti involution is a $P_{16}$-polarized  K3 surface. 
\begin{appendix}
\section{The graph of rational curves for Picard number 15}
\label{sec:rank15}
In this section we determine the graph of rational curves on the K3 surface $\mathcal{X}$ for Picard number 15, that is, for $(\kappa, \lambda)=(0,1)$. In this case the $P_{14}$-polarization is enhanced to a $P_{15}$-polarization. One verifies that the singularity at $\mathrm{P}_1$ is a rational double point of type $A_9$, and the singularity at $\mathrm{P}_2$ is still of type $A_3$. For $(\kappa, \lambda)=(0,1)$, the two sets $\{a_1, a_2, \dots, a_9\}$ and $\{b_1, b_2, b_3\}$ are the curves appearing from resolving the rational double  point singularities at $\mathrm{P}_1$ and $\mathrm{P}_2$, respectively.  The curves $L_{5}, R_6, \dots, R_{12}$ introduced above become redundant for $(\kappa, \lambda)=(0,1)$. 
\begin{figure}
  	\centering
  	\scalebox{\MyScalePicLarge}{%
    		\begin{tikzpicture}
       		\input{pic15.tex}
    		\end{tikzpicture}}
  	\caption{Rational curves on $\mathcal{X}$ with $\mathrm{NS}(\mathcal{X}) =P_{15}$}
  	\label{fig:pic15}
\end{figure}
We have the following:
\begin{theorem}
\label{thm:polarization15}
Assume Equation~(\ref{eqn:general_params}) and $(\kappa, \lambda)=(0,1)$. Then, the K3 surface $\mathcal{X}(\alpha, \beta, \gamma, \delta , \varepsilon, \zeta, \eta, \iota, 0, 1)$ is endowed with a canonical $P_{15}$-polarization with the dual graph of all smooth rational curves given by Figure~\ref{fig:pic15}.
\end{theorem}
\begin{proof}
From any of the elliptic fibrations in Theorem~\ref{thm1} it follows that the Picard rank is 15, and $\mathcal{X}$ admits an $P_{15}$-polarization. The graph of all smooth rational curves on a K3 surface endowed with a canonical $P_{15}$-polarization was constructed in \cite{MR1029967}*{Sec.~4.5} and is shown in Figure~\ref{fig:pic15}.  Thus, to prove the theorem we only have to match the curves on $\mathcal{X}(\alpha, \beta, \gamma, \delta , \varepsilon, \zeta, \eta, \iota, 0, 1)$ and their intersection properties with the ones in Figure~\ref{fig:pic15}. The graph can then be constructed in the same way as in the proof of Theorem~\ref{thm:polarization14} and shown to coincide with Figure~\ref{fig:pic15}. Notice that the nodes $R_4$ and  $R_5$ are connected by a six-fold edge. It was proven in \cite{MR556762} that Figure~\ref{fig:pic15} contains all smooth rational curves on a general K3 surface with $P_{15}$-polarization.
\end{proof}
\begin{remark}
Figure~\ref{fig:pic15} first appeared in \cite{MR556762}*{Rem.~4.5.2} and \cite{MR1029967}*{Fig.~4}.
\end{remark}
\par We have the following:
\begin{proposition}
The polarization of the K3 surface $\mathcal{X}(\alpha, \beta, \gamma, \delta , \varepsilon, \zeta, \eta, \iota, 0, 1)$ is given by the divisor
\beq
\label{linepolariz}
\mathcal{H} = 3 L_2 + L_3 + 3 a_1 + 5 a_2 + 7 a_3 + 6 a_4 + 5 a_5 + 4 a_6 + 3 a_7 + 2 a_8 + a_9\,,
\eeq
with $\mathcal{H}^2=4$.
\end{proposition}
\begin{proof}
The proof is analogous to the proof of Proposition~\ref{prop:polarization14}.
\end{proof}
We now construct the embeddings of the reducible fibers into the graph given by Figure~\ref{fig:pic15} for each elliptic fibration in Theorem~\ref{thm1}:
\subsection{The alternate fibration}
\label{ssec:alt}
\begin{figure}
  	\centering
  	\scalebox{\MyScalePicMed}{%
    		\begin{tikzpicture}
       		\input{pic15alt.tex}
    		\end{tikzpicture}}
	\caption{The alternate fibration on $\mathcal{X}$ for Picard number 15}
	\label{fig:pic15alt}
\end{figure}
There is one way of embedding the corresponding reducible fibers of case~(1) in Theorem~\ref{thm1} into the graph given by Figure~\ref{fig:pic15}. The configuration is invariant when applying the Nikulin involution in Proposition~\ref{NikulinInvolution} and shown in Figure~\ref{fig:pic15alt}. We have
\beq
\begin{split}
{\color{blue}\widetilde{A}_1} =   \langle b_1, R_3  \rangle \,, \qquad 
{\color{green}\widetilde{A}_1}=   \langle  R_1,  L_3  \rangle \,, \qquad
{\color{magenta}\widetilde{A}_1} =   \langle b_3 ,  R_2  \rangle \,, \\
{\color{orange}\widetilde{D}_{10}}=  \langle a_2, L_2,  a_3,  a_4, a_5,  a_6, a_7, a_8, a_9,  L_4, L_1 \rangle \,.
\end{split}
\eeq
Thus, the smooth fiber class is given by
\beq
\label{eqn:F_alt}
\begin{split}
 F_{\text{alt}} 
 & =   L_1 +  L_2 + L_4 + a_2+ 2 a_3 + 2 a_4 + 2 a_5 + 2 a_6 + 2 a_7 + 2 a_8 + 2 a_9 \\
 & = R_1 + L_3\  =  R_2+ b_3  \  = R_3 + b_1\,,
\end{split} 
\eeq
and the classes of a section and 2-torsion section are ${\color{red} b_2}$ and  ${\color{red} a_1}$, respectively. Using the polarizing divisor $\mathcal{H}$ in Equation~(\ref{linepolariz}), one checks that
\beq
 \mathcal{H} -  F_{\text{alt}} - L_1 \equiv  a_1 + \dots + a_9 + b_1 + 2 b_2 + b_3  \,.
\eeq
This is consistent with the fact that this fibration is obtained by intersecting the quartic $\mathcal{Q}(\alpha, \beta, \gamma, \delta , \varepsilon, \zeta, \eta, \iota, 0, 1)$ with the pencil of planes $L_1(u, v)=0$ in Equation~(\ref{eqn:pencil_alternate}) which is invariant under the Nikulin involution.
\subsection{The standard fibration}
\label{ssec:std}
There are two ways of embedding the corresponding reducible fibers of case~(2) in Theorem~\ref{thm1} into the graph given by Figure~\ref{fig:pic15}. They are depicted in Figure~\ref{fig:pic15std}. In the case of Figure~\ref{fig:pic15std1}, we have
\beq
{\color{green}\widetilde{E}_7}=   \langle L_3,  a_1,  a_2, a_3, \underline{L_2}, a_4,  a_5, a_6 \rangle \,, \qquad
{\color{blue}\widetilde{D}_6} =   \langle b_3, b_1, b_2, L_1, a_9, L_4,  a_8 \rangle \, .
\eeq
Thus, the smooth fiber class is given by
\beq
\label{eqn:F1_std}
\begin{split}
 F_{\text{std}} 
 & = 2 L_2 + L_3 + 2 a_1 + 3 a_2 + 4 a_3 +  3 a_4 + 2  a_5 + a_6   \\
 & = 2 L_1 + L_4 + a_8 + 2 a_9 + b_1 + 2 b_2 + b_3  \,,
\end{split} 
\eeq
and the class of a section is ${\color{red} a_7}$.  Using the polarizing divisor $\mathcal{H}$ in Equation~(\ref{linepolariz}), one checks that
\beq
 \mathcal{H} -  F_{\text{std}} - L_2 \equiv  a_1 + 2 a_2 + 3 a_3 + 3 a_4 + 3 a_5 + 3 a_6  + 3 a_7 + 2 a_8  + a_9\,.
\eeq
This is consistent with the fact that this fibration is obtained by intersecting the quartic $\mathcal{Q}(\alpha, \beta, \gamma, \delta , \varepsilon, \zeta, \eta, \iota, 0, 1)$ with the pencil $L_2(u, v)=0$ in Equation~(\ref{eqn:pencil_standard1}).
\begin{figure}
\begin{subfigure}{.5\textwidth}  
	\centering
	\scalebox{\MyScalePicMed}{%
		\begin{tikzpicture}
		\input{pic15std1.tex}
		\end{tikzpicture}}
	\caption{\phantom{A}}
	\label{fig:pic15std1}
\end{subfigure}%
\begin{subfigure}{.5\textwidth}
	\centering
	\scalebox{\MyScalePicMed}{%
		\begin{tikzpicture}
       		\input{pic15std2.tex}
    		\end{tikzpicture}}
  	\caption{\phantom{B}}
  	\label{fig:pic15std2}
\end{subfigure}%
\caption{The standard fibration on $\mathcal{X}$ for Picard number 15}
\label{fig:pic15std}
\end{figure}
\par Applying the Nikulin involution in Proposition~\ref{NikulinInvolution}, we obtain the fiber configuration in Figure~\ref{fig:pic15std2} with
\beq
{\color{green}\widetilde{E}_7}=   \langle R_1,  b_2,  L_1, a_9, \underline{L_4}, a_8, a_7, a_6 \rangle \,, \qquad
{\color{blue}\widetilde{D}_6} =   \langle R_2,   R_3,  a_1, a_2, a_3, L_2, a_4 \rangle \, .
\eeq
The smooth fiber class is now given by
\beq
\label{eqn:F2_std}
\begin{split}
 \check{F}_{\text{std}} & =  R_1 + 3 L_1 + 2 L_4 + a_6 + 2 a_7 + 3 a_8  + 4 a_9   \\
 & =    R_2 +  R_3 + L_2  + 2 a_1 + 2 a_2 + 2 a_3 + a_4  \,,
\end{split} 
\eeq
 and the class of the section is ${\color{red} a_5}$. Using the polarizing divisor $\mathcal{H}$ in Equation~(\ref{linepolariz}), one checks that
\beq
\begin{split}
&  3 \mathcal{H} -  \check{F}_{\text{std}} - 2L_1  - 2L_2  - L_3 \\
\equiv  \; & 3 a_1 + 5 a_2 + 7 a_3 + 7 a_4 + 7 a_5 + 6 a_6 + 5 a_7 + 4 a_8  + 3 a_9 + 3 b_1 + 4 b_2 + 3 b_3 \,.
\end{split}
\eeq
This is consistent with the fact that this fibration is also obtained by intersecting the quartic $\mathcal{Q}(\alpha, \beta, \gamma, \delta , \varepsilon, \zeta, \eta, \iota, 0, 1)$ with the pencil $C_2(u, v)=0$ in Equation~(\ref{eqn:pencil_standard2}), which for $(\kappa, \lambda)=(0,1)$ is also tangent to $L_2$. 
\subsection{The base-fiber dual fibration}
\label{ssec:bfd}
There are two ways of embedding the corresponding reducible fibers of case~(3) in Theorem~\ref{thm1} into the graph given by Figure~\ref{fig:pic15}.  They are depicted in Figure~\ref{fig:pic15bfd}. In the case of Figure~\ref{fig:pic15bfd1}, we have
\beq
{\color{green}\widetilde{E}_8}=   \langle a_1,  a_2,  \underline{L_2}, a_3, a_4, a_5, a_6,  a_7, a_8 \rangle \,, \quad
{\color{blue}\widetilde{D}_4} =   \langle R_1, b_1,  b_2 , b_3, L_1 \rangle \, , \quad
{\color{orange}\widetilde{A}_1} =   \langle R_4 ,  L_4 \rangle \, .
\eeq
Thus, the smooth fiber class is given by
\beq
\label{eqn:F1_bfd}
\begin{split}
 F_{\text{bfd}}
 & = 3 L_2 + 2 a_1 +  4 a_2 + 6 a_3 + 5 a_4 + 4 a_5 + 3 a_6  + 2 a_7 + a_8   \\
 &=  R_1  + L_1 +  b_1 + 2 b_2 + b_3  \ =  R_4 +  L_ 4\,,
\end{split} 
\eeq
and the class of a section is ${\color{red} a_9}$. Using the polarizing divisor $\mathcal{H}$ in Equation~(\ref{linepolariz}), one checks that
\beq
 \mathcal{H} -  F_{\text{bfd}} - L_3 \equiv  a_1 +\dots  + a_9    \,.
\eeq
This is consistent with the fact that this fibration is obtained by intersecting the quartic $\mathcal{Q}(\alpha, \beta, \gamma, \delta , \varepsilon, \zeta, \eta, \iota, 0, 1)$ with the pencil $L_3(u, v)=0$ in Equation~(\ref{eqn:pencil_bfd1}).
\begin{figure}
\begin{subfigure}{.5\textwidth}  
	\centering
  	\scalebox{\MyScalePicMed}{%
    		\begin{tikzpicture}
       		\input{pic15bfd1.tex}
    		\end{tikzpicture}}
  	\caption{\phantom{A}}
  	\label{fig:pic15bfd1}
\end{subfigure}%
\begin{subfigure}{.5\textwidth}  
	\centering
 	\scalebox{\MyScalePicMed}{%
		\begin{tikzpicture}
       		\input{pic15bfd2.tex}
    		\end{tikzpicture}}
  	\caption{\phantom{B}}
  	\label{fig:pic15bfd2}
\end{subfigure}%
\caption{The base-fiber dual fibration on $\mathcal{X}$ for Picard number 15}
\label{fig:pic15bfd}
\end{figure}
\par Applying the Nikulin involution in Proposition~\ref{NikulinInvolution}, we obtain the fiber configuration in Figure~\ref{fig:pic15bfd2} with
\beq
{\color{green}\widetilde{E}_8} 	= \langle b_2,   L_1,  \underline{L_4}, a_9,   a_8,  a_7, a_6,   a_5,   a_4  \rangle \, , \quad
{\color{blue}\widetilde{D}_4}	= \langle R_2,  R_3,  a_1, L_3, a_2   \rangle \,, \quad
{\color{orange}\widetilde{A}_1}	= \langle R_5 , L_2 \rangle \,.
\eeq
The smooth fiber class is given by
\beq
\label{eqn:F2_bfd}
\begin{split}
 \check{F}_{\text{bfd}} &= 4 L_1 + 3 L_4 + a_4 + 2 a_5 + 3 a_6 + 4 a_7 + 5 a_8 + 6 a_9 +  2 b_2 \\
 & =  R_2 + R_3 + L_3 + 2 a_1 + a_2 \ =  R_5 + L_2 \,,
\end{split} 
\eeq
and the class of the section is ${\color{red} a_3}$. Using the polarizing divisor $\mathcal{H}$ in Equation~(\ref{linepolariz}), one checks that
\beq
\begin{split}
& 3 \mathcal{H} -  \check{F}_{\text{bfd}} - 2L_1  - 3L_2  \\
 \equiv \; &3 a_1 + 6 a_2 + 9 a_3 + 8 a_4 + 7 a_5 + 6 a_6 + 5 a_7 + 4 a_8 + 3 b_1 + 4 b_2 + 3 b_3\,.
\end{split} 
\eeq
This is consistent with the fact that this fibration is also obtained by intersecting the quartic $\mathcal{Q}(\alpha, \beta, \gamma, \delta , \varepsilon, \zeta, \eta, \iota, 0, 1)$ with the pencil of cubic surfaces $C_3(u, v)=0$ in Equation~(\ref{eqn:pencil_bfd2}), which for $(\kappa, \lambda)=(0,1)$ has vanishing trace of the Hessian along $L_2$.
\subsection{The maximal fibration}
\label{ssec:max}
There are two ways of embedding the corresponding reducible fibers of case~(4) in Theorem~\ref{thm1} into the graph given by Figure~\ref{fig:pic15}. They are depicted in Figure~\ref{fig:pic15max}. In the case of Figure~\ref{fig:pic15max1}, we have
\beq
{\color{orange}\widetilde{D}_{12}} =   \langle b_1, b_3, b_2, L_1, a_9,  a_8, a_7,  a_6, a_5, a_4, a_3, L_2, a_2 \rangle  \,, \qquad
{\color{blue}\widetilde{A}_{1}} =   \langle  R_4 , L_3 \rangle \,.
\eeq
Thus, the smooth fiber class is given by
\beq
\label{eqn:F1_max}
\begin{split}
 F_{\text{max}}  & = 2 L_1 + L_2 + a_2 + 2 a_3 + 2 a_4 + 2 a_5 + 2 a_6 + 2 a_7 + 2 a_8 + 2 a_9 + b_1 +  2 b_2  + b_3  \\
 & =  R_4 +  L_3 \,,
\end{split} 
\eeq
and the class of a section is ${\color{red} a_1}$. Using the polarizing divisor $\mathcal{H}$ in Equation~(\ref{linepolariz}), one checks that
\beq
  \mathcal{H} -  F_{\text{max}} - L_4 \equiv  a_1 + a_2 + a_3 + a_4 + a_5 + a_6 + a_7 + a_8 + a_9    \,.
\eeq
This is consistent with the fact that this fibration is obtained by intersecting the quartic $\mathcal{Q}(\alpha, \beta, \gamma, \delta , \varepsilon, \zeta, \eta, \iota, 0, 1)$ with the pencil $L_4(u, v)=0$ in Equation~(\ref{eqn:pencil_max1}).
\begin{figure}
\begin{subfigure}{.5\textwidth}  
  	\centering
  	\scalebox{\MyScalePicMed}{%
    		\begin{tikzpicture}
       		\input{pic15max1.tex}
    		\end{tikzpicture}}
  	\caption{\phantom{A}}
  	\label{fig:pic15max1}
\end{subfigure}%
\begin{subfigure}{.5\textwidth}  
  	\centering
  	\scalebox{\MyScalePicMed}{%
    		\begin{tikzpicture}
       		\input{pic15max2.tex}
    		\end{tikzpicture}}
  	\caption{\phantom{B}}
  	\label{fig:pic15max2}
\end{subfigure}%
\caption{The maximal fibration on $\mathcal{X}$ for Picard number 15}
\label{fig:pic15max}
\end{figure}
\par Applying the Nikulin involution in Proposition~\ref{NikulinInvolution}, we obtain the fiber configuration in Figure~\ref{fig:pic15max2} with
\beq
{\color{orange}\widetilde{D}_{12}} =   \langle R_2,  R_3, a_1, a_2, a_ 3, a_4, a_5, a_6 ,  a_7,  a_8, a_9, L_4, L_1  \rangle \,, \qquad
{\color{blue}\widetilde{A}_{1}} =   \langle  R_5 , R_1 \rangle \,.
\eeq
The smooth fiber class is given by
\beq
\label{eqn:F2_max}
\begin{split}
 \check{F}_{\text{max}} 
 & =  R_2 + R_3 +  L_1 + L_4 + 2 a_1 + 2 a_2 + 2 a_ 3 + 2 a_5 + 2 a_6 + 2 a_7 + 2 a_8 + 2 a_9   \\ 
 & =   R_1 + R_5 \,,
\end{split} 
\eeq
and the class of the section is ${\color{red} b_2}$.  Using the polarizing divisor $\mathcal{H}$ in Equation~(\ref{linepolariz}), one checks that
\beq
\begin{split}
&  4 \mathcal{H} -   \check{F}_{\text{max}} - 3L_1  - 3L_2  - L_3  \\
\equiv \; &  L_2 + 4 a_1 +  7 a_2 + 10 a_3 + 9 a_4 + 8 a_5 + 7 a_6 + 6 a_7 + 5 a_8 + 4 a_9 + 4 b_1 + 6 b_2 + 4 b_3\,.
\end{split}
\eeq
This is consistent with the fact that this fibration is also obtained by intersecting the quartic $\mathcal{Q}(\alpha, \beta, \gamma, \delta , \varepsilon, \zeta, \eta, \iota, 0, 1)$ with the pencil of quartic surfaces $C_4(u, v)=0$ in Equation~(\ref{eqn:pencil_max2}), which for $(\kappa,\lambda)=(0,1)$ also has a vanishing trace of the Hessian along $L_2$.
\end{appendix}
\bibliographystyle{amsplain}
\bibliography{references}{}
\end{document}